\documentclass[12pt,reqno]{amsart}

\usepackage[T1]{fontenc}
\usepackage{lmodern}
\usepackage{amsmath,amssymb,amscd,mathtools,mathrsfs, stmaryrd}
\usepackage{microtype}

\usepackage[nodisplayskipstretch]{setspace}
\usepackage[colorlinks=true, linkcolor=blue, citecolor=blue, pagebackref=true]{hyperref}
\usepackage{enumitem}

\usepackage{appendix}

\usepackage{palatino, mathpazo}
\usepackage{amsmath}
\usepackage{dsfont}
\usepackage[all]{xy} 

\usepackage[letterpaper,margin=25mm]{geometry}
\usepackage{eepic}
\usepackage{epsfig}
\usepackage{graphicx}
\usepackage{mathrsfs}
\usepackage[english]{babel}
\usepackage{color}
\usepackage{mathabx}

\allowdisplaybreaks
\numberwithin{equation}{section}

\newtheorem{theorem}{Theorem}[section]
\newtheorem{proposition}[theorem]{Proposition}
\newtheorem{lemma}[theorem]{Lemma}
\newtheorem{corollary}[theorem]{Corollary}
\newtheorem{definition}[theorem]{Definition}
\theoremstyle{remark}
\newtheorem{remark}[theorem]{Remark}

\newcommand{\C}{\mathbb C}
\newcommand{\A}{\mathbb A}

\newcommand{\cO}{\mathcal O}
\newcommand{\Mbar}{\overline M}
\newcommand{\ev}{\operatorname{ev}}
\newcommand{\Nm}{\operatorname{Nm}}
\newcommand{\Pic}{\operatorname{Pic}}
\newcommand{\Spec}{\operatorname{Spec}}
\newcommand{\Tr}{\operatorname{Tr}}

\newcommand{\End}{\operatorname{End}}
\newcommand{\ch}{\operatorname{ch}}
\newcommand{\red}{\mathrm{red}}
\newcommand{\vir}{\mathrm{vir}}

            {\nolinebreak $\Box$ \end{trivlist}}

\newcommand{\noprint}[1]{}

\newcommand{\Z}{\mathbb Z}
\providecommand{\Q}{\mathbb Q}

\newcommand{\XX}{{\mathfrak X}}

\newcommand{\zz}{{\mathbb Z}}

\newcommand{\qq}{{\mathbb Q}}

\newcommand{\cc}{{\mathbb C}}

\newcommand{\sC}{{\mathcal C}}

\newcommand{\sP}{{\mathcal P}}

\newcommand{\cK}{{\mathcal K}}

\newcommand{\sM}{{\mathcal M}}

\DeclareMathOperator{\id}{id}

\DeclareMathOperator{\rig}{rig}

\DeclareMathOperator{\CR}{CR}
\DeclareMathOperator{\PD}{PD}

\DeclareMathOperator{\Supp}{Supp}

\DeclareMathOperator{\SL}{SL}
\DeclareMathOperator{\PGL}{PGL}

\DeclareMathOperator{\IC}{IC}

\DeclareMathOperator{\SU}{SU}

\DeclareMathOperator{\age}{age}

\DeclareMathOperator{\orb}{orb}

\DeclareMathOperator{\Td}{Td}

\DeclareMathOperator{\Act}{Act}
\DeclareMathOperator{\FM}{FM}
\DeclareMathOperator{\KU}{KU}
\DeclareMathOperator{\per}{per}
\DeclareMathOperator{\GS}{GS}
\newcommand{\limi}{\varprojlim}
\DeclareMathOperator{\del}{del}
\DeclareMathOperator{\Br}{Br}
\DeclareMathOperator{\triv}{triv}

\newcommand{\Pcal}{\mathsf P}

\newcommand{\rk}{\mathop{\rm rk}}

\newcommand{\pr}{\mathop{\rm pr}\nolimits}

\newcommand{\cl}{\mathop{\rm cl}\nolimits}

\newcommand{\ord}{\mathop{\rm ord}\nolimits}

\newcommand{\Jac}{\mathop{\rm Jac}\nolimits}

\newcommand{\Nnil}{\mathcal N}

\newcommand{\St}{\mathsf S}
\newcommand{\Jrig}{\mathsf J}

\newcommand{\thickslash}{\mathbin{\!\!\pmb{\fatslash}}}

\newcommand {\mat}      [1] {\left(\begin{array}{#1}}
\newcommand {\rix}          {\end{array}\right)}

\title[Quantum cohomology of Higgs moduli]{Quantum cohomology, Hitchin systems, and Fourier--Mukai comparison}
\author{Yunfeng Jiang}
\address{Department of Mathematics\\ University of Kansas\\ 405 Snow Hall 1460 Jayhawk Blvd\\Lawrence KS 66045 USA} 
\email{y.jiang@ku.edu}

\author{Hsian-Hua Tseng}
\address{Department of Mathematics\\ Ohio State University\\ 100 Math Tower, 231 West 18th Ave. \\ Columbus,  OH 43210\\ USA}
\email{hhtseng@math.ohio-state.edu}

\date{\today}



\begin{document}

\begin{abstract}
Let $X$ be the moduli space of stable trace-free rank $r$ Higgs bundles over a smooth projective curve $C$ with fixed determinant $\Lambda\in\Pic(C)$ and $\gcd(r,\deg\Lambda)=1$. We prove that the $\cc^*$-equivariant quantum product of a divisor on $X$ is given by the Steinberg correspondence of the Lagrangian Steinberg cycle in $X\times X$. The Steinberg cycle is the image of the reduced virtual fundamental cycle of the $2$-pointed Kontsevich stable map space to $X$ under evaluation map. For $\Gamma=\Pic^0(C)[r]$, we give a character decomposition for $\Gamma$-invariant curve classes or finite orbit sums. In rank two and  genus two we explicitly determine its fifteen endoscopic lines and the reduced spectral discriminant. 

For $\widehat X=[X/\Gamma]$, which is the moduli space of stable $\PGL_r$-Higgs bundles over $C$. We consider the  orbifold
quantum cohomology theory from a gerbe-twisted  theory. The complex K-theory $\KU^*(X)$ and twisted K-theory $\KU^*([X/\Gamma],\alpha)$ by the lifting $\mu_r$-gerbe $\alpha$ of the universal projective bundle is an isomorphism, and is given by  the Fourier-Mukai transformation.  We prove that the Fourier-Mukai transformation is compatible with the quantum product by the divisors. 
\end{abstract}

\maketitle

\setcounter{tocdepth}{1}

\tableofcontents

\section{Introduction}

We work over $\C$ and fix a smooth projective connected curve $C$ of genus $g\ge2$. 

\subsection{Motivation}
Fix an integer \(r\ge2\) and a line bundle \(\Lambda\in\text{Pic}(C)\) of degree \(d\) with $\gcd(r,d)=1$. Let $M:=M_{r,\Lambda}(C)$ be the moduli space of semi-stable vector bundles $E$ with rank $r$ and determinant $\det(E)=\Lambda\in\Pic(C)$. 
Let $d=\deg(\Lambda)$ be the degree of $\Lambda$. When $\gcd(r,d)=1$, semi-stability of $E$ is equivalent to stability, and $M_{r,\Lambda}(C)$ is a smooth projective variety of dimension $(r^2-1)(g-1)$.  
For simplicity we assume\footnote{For $\deg(L)=d$ odd, $M_{r,\Lambda}(X)$ is isomorphic to $M_{r,\Lambda'}(C)$ where $\deg(\Lambda')=1$, by tensoring.} that $d=\deg(\Lambda)=1$. 

 Let  $M^{\PGL_r}_{r,d}(C)$ be the moduli space of stable $\PGL_r$-bundles $E$ with $\deg(E)=d$.  We let $1\le d\le r-1$.  Let $\Gamma:=\Pic^0(X)[r]$ be the  subgroup of $r$-torsions of the Picard group of degree zero line bundles on $C$. Then $M^{\PGL_r}_{r,d}(X)=[M_{r,\Lambda}(C) /\Gamma]$ is a smooth global quotient Deligne-Mumford stack, where the action of $\Gamma$ is given by tensor product. 
It is known from Harder-Narasimhan \cite{Harder-Narasimhan}, that their cohomologies are isomorphic, i.e., 
$$H^*(M)\cong H^*(M/\Gamma).$$
The quantum cohomoloy of $M$ and $[M/\Gamma]$ were studied in \cite{Munoz} and \cite{DJT}. 
The Langlands dual of group of $\SL_r$ is $\PGL_r$.  Thus, it is interesting to explore some Langlands duality properties of these moduli spaces, and their Higgs bundles enhancement.

Let
\[
 X=M_\Lambda^{\text{SL}_r}(C)
\]
denote the coarse moduli space of stable trace-free Higgs pairs with determinant isomorphic to \(\Lambda\). Its closed points are stable Higgs pairs \((E,\phi)\) satisfying
\[
 \det E\simeq\Lambda,\qquad
 \Tr\phi=0,\qquad
 \phi:E\longrightarrow E\otimes K_C.
\]
$X$ is smooth and quasi-projective of complex dimension $$m:=2(r^2-1)(g-1).$$ 

The group $\Gamma=\Pic^0(C)[r]$ of $r$-torsion line bundles on $C$ acts on $X$ by tensor products.  The quotient stack $$\widehat{X}=[M^{\text{SL}_r}_{r,\Lambda}(C)/\Gamma]$$ can be identified with the moduli of stable $\text{PGL}_r$ Higgs bundles. 

The smooth variety $X$ and smooth Deligne--Mumford stack $\widehat X$ have Hitchin maps to the same affine base:
\begin{equation}\label{eqn_Hitchin_map}
 \xymatrix{X\ar[dr]^h&&\widehat X\ar[dl]_{h'}\\&
 \A=\bigoplus_{i=2}^rH^0(C,K_C^i).&}
\end{equation}
The topological mirror symmetry results \cite{HT,GWZ,MS,GS} motivate a comparison of their quantum cohomology rings. 

\subsection{Main results}
The torus $T=\C^*$ acts on $X$ by scaling the Higgs fields. The symplectic form $\omega$ on $X$ satisfies $t^*\omega=t\omega$, and we write $\hbar=c_1^T(\C_1)$. Let
\[
 \sM_{n,\beta}=\overline M_{0,n}(X,\beta)
\]
be the moduli space of stable maps from $n$-marked, genus zero curves to $X$ of degree $\beta\in H_2(X,\zz)$. For $\beta\ne0$, the symplectic cosection of the absolute obstruction sheaf is a surjection to $\cO\otimes\C_1$. The Kiem-Li kernel-cone construction of Appendix~\ref{sec:reduced} defines a cycle\footnote{The reduced class is not the Kiem--Li class localized on the empty degeneracy locus; it is the zero-section Gysin class of the intrinsic cone in the kernel obstruction bundle stack.}, without asserting the existence of a reduced perfect obstruction theory, and proves
\[
 [\sM_{n,\beta}]^{\vir}_T =\hbar[\sM_{n,\beta}]^{\red}_T,\qquad  \operatorname{vdim}^{\red}=m-2+n.
\]

Evaluation maps are proper over the Hitchin base, and the fixed loci of
the stable-map stacks are proper.  Thus localized invariants and the
nondegenerate localized pairing are defined. For classes 
$\delta_i\in H_T^*(X)_{\mathrm{loc}}$ in localized equivariant cohomology, set
\[
 \langle\delta_1,\delta_2,\delta_3\rangle_{0,3,\beta}^{X,T}
 =\int^T_{[\sM_{3,\beta}]^{\vir}_T}
       \prod_{i=1}^3\ev_i^*\delta_i.
\]
If $\{\delta_j\}$ and $\{\delta^j\}$ are Poincar\'e dual bases of $H_T^*(X)_{\text{loc}}$, then the small quantum product on $H_T^*(X)_{\text{loc}}$ is defined by 
\begin{equation}\label{eqn_quantum_product}
 \gamma_1\star\gamma_2
 =\sum_{\beta,j}q^\beta
    \langle\gamma_1,\gamma_2,\delta_j\rangle_{0,3,\beta}^{X,T}\delta^j, \quad \gamma_1,\gamma_2\in H_T^*(X)_{\text{loc}}.
\end{equation}
We put $M_D=D_T\star(-)$ for a divisor equipped with an equivariant lift. The Novikov conventions and Borel--Moore interpretation of every correspondence operator are specified in Section~\ref{sec:divisor}.

For a nonzero effective curve class \(\beta\in H_2(X,\Z)\), from the virtual dimension formula: 
$$\operatorname{dim}_{\cc}[\sM_{2,\beta}]^{\vir}=m-1,\qquad \operatorname{dim}_{\cc}[\sM_{2,\beta}]^{\red}=m.$$  
We set
\begin{equation}\label{eq:Zbeta-def}
 Z_{2,\beta}^{\red}
 :=
 (\ev_1,\ev_2)_*[\sM_{2,\beta}]^{\red} \in A_m^T(X\times X) .
\end{equation}
and 
\begin{equation}\label{eq:Cbeta-def}
 C_{2,\beta}
 :=
 (\ev_1,\ev_2)_*[\sM_{2,\beta}]^{\vir} \in A_{m-1}^T(X\times X) .
\end{equation}

A Steinberg subvariety for the Hitchin morphism is a \(T\)-invariant irreducible Lagrangian subvariety
\[
 L\subset (X\times X,p_1^*\omega-p_2^*\omega)
\]
such that $L\subset X\times_\A X$, where $\omega$ is the symplectic form of $X$. 
The \emph{Steinberg cycle space} $\text{St}_{\rm cyc}(X/\A)$ is the \(\qq\)-span of their equivariant cycle classes. We denote its Novikov completion by
\[
 \widehat{\text{St}}_{\rm cyc}(X/\A)
 :=\text{St}_{\rm cyc}(X/\A)\widehat\otimes\Lambda_H,
\]
where the \(H\)-degree Novikov ring \(\Lambda_H\) is defined in Section
\ref{sec:divisor}.

\begin{theorem}[see Theorems \ref{thm:Z-Steinberg} and \ref{thm:divisor}]\label{thm:main}
For every nonzero effective curve class \(\beta\in H_2(X,\Z)\),
\[
 C_{2,\beta}=\hbar Z_{2,\beta}^{\red},
 \qquad
 Z_{2,\beta}^{\red}\in \text{St}_{\rm cyc}(X/\A).
\]
More explicitly, there are finitely many Steinberg subvarieties
\(L_{\beta,j}\subset X\times_\A X\) and rational numbers \(c_{\beta,j}\in\qq\) such that
\begin{equation}\label{eq:Zbeta-Steinberg}
 Z_{2,\beta}^{\red} = \sum_j c_{\beta,j}[L_{\beta,j}]_T .
\end{equation}

Moreover, for every divisor \(D\) equipped with an equivariant lift
\(D_T\), the quantum product with $D$ is given by 
\begin{equation}\label{eq:divisor-main}
 M_D
 =
 D_T\cup
 +
 \hbar
 \sum_{\beta>0}
 (D\cdot\beta)q^\beta\mathbf{F}_{2,\beta},
\end{equation}
where $\mathbf{F}_{2,\beta}$ is the correspondence 
$$
\mathbf{F}_{2,\beta}: H_T^*(X)_{\mathrm{loc}}\to H_T^*(X)_{\mathrm{loc}}
$$
defined by 
$$
\mathbf{F}_{2,\beta}(\delta)=\operatorname{PD}_X^{-1}p_{2*}\bigl(p_1^*\delta\cap\cl_T^{\mathrm{BM}}(Z_{2,\beta}^{\red})\bigr).$$
Here $p_i: X\times X\to X$ are projections for $i=1,2$.

Let $\Theta$ be the theta class of $X$, then
\begin{equation}\label{eq:theta-main}
 \frac{M_\Theta-\Theta_T\cup}{\hbar}
 =
 \sum_{\beta>0}
 (\Theta\cdot\beta)q^\beta \mathbf{F}_{2,\beta}
 \in \Act\!\left(\widehat{\text{St}}_{\rm cyc}(X/\A)\right).
\end{equation}
\end{theorem}

This is the self-correspondence specialization of Maulik--Okounkov's definition of a Steinberg correspondence \cite[\S3.2.6, Definition 3.2.9]{MO}. Our terminology is modeled on that definition: the proper Hitchin map to the affine base replaces the
symplectic-resolution morphism in their setting.

\begin{remark}\label{cor_theta_square}
For $\beta\ne0$, define the one-pointed reduced evaluation class in localized cohomology by
\[
 Z_{1,\beta}^{\red}
 :=\operatorname{PD}_X^{-1}\ev_*[\overline M_{0,1}(X,\beta)]^{\red}_T.
\]
Applying the ordinary divisor equation twice, and using the standard/reduced identity, gives
\[
 \Theta_T\star\Theta_T=\Theta_T^2+
 \hbar\sum_{\beta>0}(\Theta\cdot\beta)^2q^\beta Z_{1,\beta}^{\red}.
\]

In \cite{JKT}, the following formula will be calculated in rank two case. Let $E_4(q)=1+240\sum_{n\ge1}\sigma_3(n)q^n$ and $\sigma_3(n)=\sum_{k\mid n}k^3$, Then we have
\[
 \Theta\star\Theta-\Theta^2=
 \begin{cases}
 \displaystyle\frac{\hbar^2}{60}
 \bigl(E_4(q)-1+12(E_4(q^2)-1)\bigr)1_X,&g=2,\\[4pt]
 \displaystyle\frac{g-1}{3}\hbar^2(E_4(q^2)-1)1_X,&g>2.
 \end{cases}
\]
\end{remark}

\subsection{Endoscopic decomposition}
Consider the group $\Gamma=\Pic^0(C)[r]$, $|\Gamma|=r^{2g}$. Tensoring by $\eta\in\Gamma$ gives an automorphism\footnote{$\tau_\eta$ need not preserve a specified integral curve class.} $\tau_\eta: X\to X$ over $\A$ which preserves $\omega$ and $\Theta_T$. 

Over $E_r=\mathbb Q(\mu_r)$, the character projectors and projected kernels are
\begin{equation}\label{eq:projector}
 \Pcal_\kappa=\frac1{|\Gamma|}\sum_{\eta\in\Gamma}
    \kappa(\eta)^{-1}\tau_\eta^*,\quad \kappa\in \widehat{\Gamma}.
\end{equation}
Character orthogonality gives
\begin{equation}\label{eq:projector-relations}
 \Pcal_\kappa\Pcal_{\kappa'}=\delta_{\kappa,\kappa'}\Pcal_\kappa,
 \qquad \sum_{\kappa\in\widehat{\Gamma}}\Pcal_\kappa=\id,
\end{equation}
and, with $E_r$ coefficients,
\begin{equation}\label{eq:cohomology-character-decomposition}
 H_T^*(X)=\bigoplus_{\kappa\in\widehat\Gamma}H_T^*(X)_\kappa,
 \qquad H_T^*(X)_\kappa=\Pcal_\kappa H_T^*(X).
\end{equation}

$\Gamma$ acts on the set of effective curve classes of $X$. For a finite $\Gamma$-stable set $B$ of nonzero effective classes, set
\[
 Z_B^{\red}=\sum_{\beta\in B}Z_{2,\beta}^{\red},\qquad
 \mathbf F_B=\mathbf F_{Z_B^{\red}}.
\]
Define the projected kernels to be
\[
 Z_{B,\kappa}^{\red}=\frac1{|\Gamma|}\sum_{\eta\in\Gamma}
    \kappa(\eta)^{-1}(\tau_\eta\times\id)_*Z_B^{\red}, \quad \kappa\in\widehat{\Gamma}.
\]
\begin{theorem}[Character blocks]\label{thm:intro-character}
The class $Z_B^{\red}$ is symmetric and diagonally $\Gamma$-invariant, and
\[
 Z_B^{\red}=\sum_\kappa Z_{B,\kappa}^{\red},\qquad
 \mathbf F_{Z_{B,\kappa}^{\red}}
   =\Pcal_\kappa\mathbf F_B\Pcal_\kappa,\qquad
 (Z_{B,\kappa}^{\red})^{\mathsf t}=Z_{B,\kappa^{-1}}^{\red}.
\]
For $\kappa\ne1$, both contractions with $\Theta_T$ vanish in $E_r$-valued equivariant Chow groups. Individual blocks are symmetric when $\kappa=\kappa^{-1}$, in particular for rank two. For a $\Gamma$-fixed class $\beta$, its positive-degree quantum
coefficient on a character block is
\[
 [q^\beta](M_\Theta)|_\kappa
  =(\Theta\cdot\beta)\hbar\,\mathbf F_{Z_{\{\beta\},\kappa}^{\red}}.
\]
For arbitrary classes, character-block preservation holds for $\Gamma$-invariant Novikov specializations, as explained in Corollary~\ref{cor:quantum-blocks}.
\end{theorem}

Section~\ref{sec_rank2} determines the endoscopic lines in rank $r=2$, genus $g=2$, and the support of the associated derived cohomological blocks. 

\subsection{Fourier--Mukai comparison}\label{subsec_Fourier-Mukai_intro}
Let $\alpha_T$ be a specified coherently $T$-equivariant finite lifting gerbe on $\widehat X$, with underlying twist $\alpha$. The precise datum is fixed in Definition~\ref{def:equivariant-gerbe}. We choose\footnote{We do not promote it to a global equivariant Fourier--Mukai kernel.} a nonequivariant Groechenig--Shen $K$-theory isomorphism \cite{GS}.  


There is a diagram:
\begin{equation}\label{diagram_FM_Steinberg_intro}
    \xymatrix{
KU_{T}^*(X)\otimes \cc\ar[r]^--{\FM_{\sP}}\ar[d]_{\cong} &  KU_{T}^*([X/\Gamma], \alpha_T)\otimes \cc\ar[d]^{\cong}\\
\widehat H^k_{T,\per}(X)\ar[d]_{F_{\beta}^{X}}\ar[r]^--{\Phi_P^{\rig,\CR}}& \widehat H^k_{T,\del,\per}(\widehat X,\alpha_T)\ar[d]^{F^{\widehat X,\rig,\CR}_{\bar\beta}}  \\
\widehat H^k_{T,\per}(X)\ar[r]^--{\Phi_P^{\rig,\CR}}& \widehat H^k_{T,\del,\per}(\widehat X,\alpha_T)
    }
\end{equation}
where 
$\widehat H^k_{T,\per}(X)$ is the total degree periodic cohomology group, and 
$\widehat H^k_{T,\del,\per}(\widehat X,\alpha_T)$ is the total degree periodic orbifold cohomology group, see Section \ref{sec_Groechenig-Shen} for details. The vertical isomorphisms in the top square are from the work of Freed-Hopkins-Teleman \cite{FHT}. The vertical arrows in the bottom square are the Steinberg enumerative operators.

\begin{theorem}[see Theorem \ref{thm:additive-lift}]\label{thm:GS-isomorphism_intro}
For $k=0,1$ there is a noncanonical $\C[\![z]\!]$-linear isomorphism
\[
 \widehat\Phi^{H,\mathrm{add}}_{T,k}:
 \widehat H^k_{T,\per}(X)\xrightarrow{\,\sim\,}
 \widehat H^k_{T,\del,\per}(\widehat X,\alpha_T)
\]
reducing to the chosen nonequivariant comparison. Via the completed Chern characters this is an isomorphism of completed complexified $K$-groups. It is not asserted to be multiplicative, canonical, kernel-induced, or a map of equivariant spectra.
\end{theorem}

The following result concerns the bottom square of (\ref{diagram_FM_Steinberg_intro}):
\begin{theorem}
\label{thm:F_Z_intro}
Under the independently specified kernel, coefficient-theory, and convolution hypotheses \textup{(C1)--(C4)} of Theorem~\ref{thm:cohomological-95}, the completed cohomological kernel transform intertwines the reduced two-point operators. Its
rigidified Chen--Ruan realization is
$$\Phi_P^{\rig,\CR}=\Jrig^{-1}\St_{\mathrm{age}}^{-1}\Phi^H_{\Xi_P},\quad \text{and }\Phi^H_{\Xi_P}=\widehat{\Phi}^{H,\mathrm{add}}_{T},$$
and
\begin{equation}\label{diagram_FM_Steinberg_intro}
 F^{\widehat X,\rig,\CR}_{\bar\beta}\Phi_P^{\rig,\CR}
   =\Phi_P^{\rig,\CR}F^X_\beta.
\end{equation}
\end{theorem}

\begin{remark}
This is a sufficient criterion, not an unconditional compatibility theorem. In particular, character decomposition does not imply its convolution hypothesis.  Compatibility with full quantum divisor operators additionally requires the classical divisor terms and Novikov degrees to be matched.  

Theorem \ref{thm:F_Z_intro} implies that the Fourier-Mukai transform $\text{FM}_{\mathcal{P}}$ is compatible with the quantum product. Hence this implies the Monodromy conjecture saying that the quantum connection and the Fourier-Mukai transformation have the same Monodromy  once the Fourier-Mukai kernel can be extended to the full Hitchin base $\A$, see \cite{BMO}, \cite{MO}. 
\end{remark}

\subsection{Comments}
The coprimality $(r,d)=1$ is essential in this paper. 
When \(\gcd(r,d)\ne1\), strictly semistable Higgs bundles generally occur and the coarse moduli space is singular.  One must then decide whether the correct theory is formulated on a stack, a derived enhancement, intersection cohomology, or a chosen symplectic resolution when one exists.

\subsection{Related and future work}
The structural result does not identify the individual Steinberg components, their rational coefficients, or a closed Novikov resummation. A Poincar\'e sheaf and a degree-dependent reduced cycle also have different geometric types, and  an equality cannot be inferred merely from their occurrence in the same Hitchin system. The relationship with Vafa--Witten theory and S-duality remains a motivation \cite{Jiang_Adv,JK,Nesterov}. Fixed-locus localization takes place over the nilpotent cone, but this does not imply that the global correspondence is supported there.

Concurrent work of Bousseau and Chattopadhyay \cite{BC} studies intrinsic mirror symmetry for the \(\PGL_2\) and \(\SL_2\) character varieties of the once-punctured torus. They realize a compactification of the \(\PGL_2\) character variety as an orbifold cubic surface, show that its intrinsic mirror recovers the \(\SL_2\) character-variety family, and identify its higher-genus quantum mirror with the spherical \(A_1\) double affine Hecke algebra. Their construction is on the Betti/logarithmic side and is distinct from the reduced Gromov–Witten theory considered here. Nevertheless, both constructions are governed by the finite \(2\)-torsion quotient relating \(\SL_2\) and \(\PGL_2\), suggesting a possible relation between the Steinberg quantum operators of Higgs moduli and quantizations of Langlands-dual character varieties.

\subsection{Outline}
Section~\ref{sec_preliminary} fixes the moduli, theta, and gerbe conventions.  Sections~\ref{sec_Steinberg}--\ref{sec:divisor} construct the Kiem-Li kernel-cone cycle, prove isotropy and the Steinberg theorem, and derive the divisor formulas, with separate orbifold conventions. Sections~\ref{sec_property_Steinberg}--\ref{sec_rank2} treat support, character blocks, and the rank-two genus-two spectral geometry.
Section~\ref{sec_Groechenig-Shen} proves completed additive comparison statements.  Section~\ref{sec_Fourier-Mukai_Steinberg} proves the trivial-torsor quotient formula and formulates conditional supported Fourier--Mukai comparison criteria. Section \ref{sec:reduced} recalls the construction of reduced virtual fundamental cycle. 

\subsection*{Acknowledgement}
Y. J. thanks Professor Simon Donaldson for pointing out the possible connection from quantum cohomology of moduli space of  Higgs bundles to Khovanov homology. We thank S. Das, P. Kundu, D. Nesterov for valuable discussions. Some results are checked with GPT-6. Y. J. is partially supported by  NSF DMS-2401484,  a Simon Collaboration Grant, and College Black-Babcock Professorship Funds from University of Kansas.

\section{Preliminary}\label{sec_preliminary}

In this section we collect basic knowledge on the moduli space of stable $\SL_r$-Higgs bundles and stable $\PGL_r$-Higgs bundles. 

\subsection{Moduli space of \texorpdfstring{$\SL_r$}{SLr}-Higgs bundles}\label{subsec_SL_Higgs}
Let \(C\) be a smooth projective connected complex curve of genus \(g\ge2\).

Let \(\mathfrak X_\Lambda\) be the Deligne--Mumford stack whose objects over a scheme \(S\) are triples (write \(p_C:C\times S\to C\))
\[
 (\mathcal E,\phi,\eta),\qquad
 \phi:\mathcal E\longrightarrow\mathcal E\otimes p_C^*K_C,
 \qquad
 \eta:\det\mathcal E\xrightarrow{\sim}p_C^*\Lambda,
\]
on \(C\times S\), satisfying the scheme-theoretic equation
\[
 \Tr(\phi)=0
 \quad\text{in }\Gamma(C\times S,p_C^*K_C),
\]
and having stable geometric fibers. The trace equation is imposed on the family before passage to geometric fibers and is preserved by arbitrary base change. An isomorphism is required to commute with \(\phi\) and preserve the determinant framing \(\eta\).

Scalar multiplication defines a central, pullback-compatible subgroup \(\mu_{r,S}\) of the automorphism group of every family. At a geometric point, stability makes the underlying Higgs pair simple, and preservation of \(\eta\) cuts its scalar automorphisms down precisely to \(\mu_r\). Thus this central subgroup has every geometric fiber equal to the full inertia. We may rigidify \(\mathfrak X_\Lambda\) by removing it. By the rigidification theorem \cite[Appendix~A]{AOV}, there is a morphism
\[
 \rho:\mathfrak X_\Lambda\longrightarrow \mathfrak X_\Lambda\thickslash \mu_r
\]
which is an fppf gerbe banded by \(\mu_r\), and the geometric inertia of its target
is the quotient by \(\mu_r\). The target therefore has trivial inertia and is an algebraic space.  Every morphism from this gerbe to an algebraic space descends uniquely through its base, so the rigidification has the same coarse moduli universal property as \(\mathfrak X_\Lambda\). Uniqueness of the coarse space therefore canonically identifies \(\mathfrak X_\Lambda\thickslash \mu_r\) with \(X\). We write
\[
 \pi:\mathfrak X_\Lambda\longrightarrow X
\]
for the resulting banded \(\mu_r\)-gerbe.\footnote{The stack obtained by imposing only \(\det E\simeq\Lambda\), without remembering such an isomorphism \(\eta\), retains scalar
\(\mathbb G_m\)-inertia and is not the stack denoted by \(\mathfrak X_\Lambda\) here.}
The coprimality assumption implies that semistability equals stability. In particular, the coarse moduli space is smooth; this is recalled in Section \ref{sec:smoothness}.  We work only in this smooth coprime case.

The $1$-dimensional torus $T=\cc^*$ acts on the stack and its coarse space by scaling the Higgs fields:
\[
 t\cdot(\mathcal E,\phi,\eta)=(\mathcal E,t\phi,\eta),
 \qquad
 t\cdot(E,\phi)=(E,t\phi).
\]
We write
\[
 \hbar=c_1^T(\C_1)\in H_T^2(\mathrm{pt}).
\]

\subsection{The normalized determinant-of-cohomology theta class}

The determinant framing $\eta:\det E\xrightarrow{\sim}\Lambda$, makes the tautological bundle
\(\mathcal E\) available on \(C\times\mathfrak X_\Lambda\).  Choose once and
for all an auxiliary vector bundle \(F_0\) on \(C\) with
\[
 \rk F_0=r,
 \qquad
 \deg F_0=r(g-1)-d.
\]
For every rank-\(r\), degree-\(d\) bundle \(E\), Riemann--Roch gives
\begin{equation}\label{eq:theta-Euler-zero}
 \chi(C,E\otimes F_0)
 =r\deg F_0+r d+r^2(1-g)=0.
\end{equation}
Let
\[
 p:C\times\mathfrak X_\Lambda\to\mathfrak X_\Lambda,
 \qquad
 p_C:C\times\mathfrak X_\Lambda\to C
\]
be the projections and define the determinant-of-cohomology line
\begin{equation}\label{eq:theta-line-stack}
 \lambda_{F_0}
 :=\left(\det Rp_*\bigl(\mathcal E\otimes p_C^*F_0\bigr)\right)^{-1}
 \quad\text{on }\mathfrak X_\Lambda.
\end{equation}
The inverse fixes the conventional sign of the theta line. For a banded \(\mu_r\)-gerbe, pullback identifies line bundles on the base with line bundles on the gerbe on which the central inertia has trivial character; equivalently, a line bundle of inertia weight zero is the pullback of a unique line bundle on the base. This is fpqc descent on the gerbe; see e.g. \cite[Lemma~2.1.1.13]{Lieblich}. The scalar \(\zeta\in\mu_r\) acts on its fiber with weight \(\zeta^{-\chi(E\otimes F_0)}=1\) by
\eqref{eq:theta-Euler-zero}. Hence the inertia acts trivially and \(\lambda_{F_0}\) descends uniquely to a line bundle \(\overline\lambda_{F_0}\) on \(X\), characterized by
\(\pi^*\overline\lambda_{F_0}\cong\lambda_{F_0}\) under $\pi: \XX_{\Lambda}\to X$.  We fix the normalization
\begin{equation}\label{eq:theta-normalized}
 \Theta:=c_1(\overline\lambda_{F_0})\in H^2(X,\zz),
\end{equation}
We use the image of this integral class in rational cohomology.
This chosen auxiliary bundle fixes the normalization used in every factor
\(\Theta\cdot\beta\) below. We keep \(F_0\) fixed throughout and make no assertion that \(\overline\lambda_{F_0}\) or \(\Theta\) is independent of \(F_0\). The general divisor theorem applies to any specified divisor with an equivariant lift.

Scaling the Higgs fields changes only \(\phi\), leaving \((E,\eta)\) and the complex in \eqref{eq:theta-line-stack} unchanged. It therefore gives \(\lambda_{F_0}\) a canonical \(T\)-linearization. This linearization commutes with the central \(\mu_r\)-action; uniqueness in weight-zero descent makes its descent data satisfy the equivariant cocycle, and hence gives \(\overline\lambda_{F_0}\) a canonical \(T\)-linearization.  We denote the resulting equivariant class by \(\Theta_T\). For the structural theorem below, no special property of \(\Theta\) is used beyond the fact that it is a divisor.  We therefore prove the result for an arbitrary divisor \(D\), and then specialize to \(D=\Theta\).

Assume that for a divisor $D$ a \(T\)-equivariant lift \(D_T\) is chosen. The determinant line has a natural Higgs-scaling linearization, so this condition holds for \(\Theta\). The divisor equation shows that the positive-degree correction depends only on the underlying nonequivariant divisor, and the coefficient is the ordinary degree \(D\cdot\beta\).

\subsection{The \texorpdfstring{$\PGL_r$}{PGLr} quotient and rational theta descent}
\label{subsec_PGL_r}
Let $\Gamma=\Pic^0(C)[r]$ and $q:X\to\widehat X=[X/\Gamma]$. This is a \'etale atlas of degree $|\Gamma|$, and the quotient $\widehat{X}$ is a smooth Deligne-Mumford stack even when its coarse space is singular. 

\begin{proposition}[Rational theta descent]\label{prop:theta-descent}
There is a unique class $\widehat\Theta_T\in H_T^2(\widehat X,\Q)$ with $q^*\widehat\Theta_T=\Theta_T$. The same statement holds for rational equivariant Chow classes of codimension one.
\end{proposition}
\begin{proof}
Tensoring preserves the equivariant first Chern class of the chosen determinant line by Lemma~\ref{lem:theta-symplectic-invariant}. For a finite group, rational transfer gives
$q^*:H_T^*(\widehat X,\Q)\simeq H_T^*(X,\Q)^\Gamma$. The Chow version follows from
$q_*q^*=|\Gamma|$ and $q^*q_*a=\sum_\eta\tau_\eta^*a$. Thus $\widehat\Theta_T=|\Gamma|^{-1}q_*\Theta_T$ has the claimed pullback, and injectivity proves uniqueness.
\end{proof}
A $\Gamma$-invariant line bundle need not carry a specified $\Gamma$-linearization.  Conversely, the presence of a lifting gerbe does not by itself prove that the theta line fails to descend. No integral nondescent claim is needed here. The tensor product
$\bigotimes_{\eta\in\Gamma}\tau_\eta^*\overline\lambda_{F_0}$ has a natural permutation linearization and its rational first Chern class is $|\Gamma|\Theta$.

\subsection{Smoothness and symplectic geometry}\label{sec:smoothness}
We recall the proof of smoothness of the moduli stacks $X$ and $[X/\Gamma]$.

Because both the determinant framing and the equation \(\Tr(\phi)=0\) are imposed scheme-theoretically in the moduli functor, their infinitesimal deformations remove the scalar summands in both terms. Thus, by the deformation theory of Higgs bundles (see \cite{HitchinSD,HitchinSI}, \cite{Nitsure}, and \cite{Simpson}), at a stable Higgs pair \((E,\phi)\), the fixed-determinant, trace-free deformation complex is
\begin{equation}\label{eq:def-complex}
 \mathcal C^\bullet_{(E,\phi)}
 =
 \left[
 \End_0(E) \xrightarrow{[\phi,-]} \End_0(E)\otimes K_C
 \right],\quad \text{placed in degrees } 0, 1.
\end{equation}

\begin{lemma}[Smoothness and dimension]\label{lem:smooth}
The variety \(X\) is smooth and
\[
 \dim_\C X=2(r^2-1)(g-1).
\]
\end{lemma}

\begin{proof}
If an endomorphism \(a:E\to E\) commutes with \(\phi\), stability implies that \(a\) is scalar. Since the deformation complex uses
\(\End_0(E)\),
\(
 \mathbb H^0(C,\mathcal C^\bullet_{(E,\phi)})=0
\). 
The trace pairing
\[
 \End_0(E)\otimes\End_0(E)\longrightarrow\cO_C
\]
together with Serre duality identifies the degree-two hypercohomology with the dual of degree zero one:
\[
 \mathbb H^2(C,\mathcal C^\bullet_{(E,\phi)})
 \cong
 \mathbb H^0(C,\mathcal C^\bullet_{(E,\phi)})^\vee
 =0.
\]
Thus \((E,\phi)\in\mathfrak X_\Lambda\) is unobstructed and its infinitesimal deformation space is
\(
 \mathbb H^1(C,\mathcal C^\bullet_{(E,\phi)})
\).

Riemann--Roch gives
\[
 \chi(\End_0(E))=(r^2-1)(1-g), \quad \text{and }
 \chi(\End_0(E)\otimes K_C)=(r^2-1)(g-1).
\]
Hence
\[
 \chi(\mathcal C^\bullet)
 =
 \chi(\End_0(E))-\chi(\End_0(E)\otimes K_C)
 =
 -2(r^2-1)(g-1).
\]
Since the degree-zero and degree-two hypercohomology vanish,
\[
 \dim \mathbb H^1
 =
 -\chi(\mathcal C^\bullet)
 =
 2(r^2-1)(g-1).
\]
The determinant-framed stack \(\mathfrak X_\Lambda\) has only the finite central stabilizer \(\mu_r\), with no infinitesimal stabilizer. This stabilizer acts on \eqref{eq:def-complex} by conjugation and hence acts trivially on \(\mathbb H^1\). Therefore the smooth gerbe \(\mathfrak X_\Lambda\to X\) has trivial stabilizer action on its local deformation space, and its coarse space \(X\) is smooth with
$T_{(E,\phi)}X =\mathbb H^1(C,\mathcal C^\bullet_{(E,\phi)})$ and the displayed dimension.
\end{proof}

The trace pairing on \eqref{eq:def-complex} and Serre duality yield a nondegenerate alternating pairing on \(\mathbb H^1(\mathcal C^\bullet)\), giving the standard holomorphic symplectic form
\[
 \omega\in H^0(X,\Omega_X^2).
\]
This is the algebraic form of Hitchin's hyperk\"ahler symplectic structure
\cite{HitchinSD,HitchinSI}, which implies \(K_X\simeq\cO_X\) and $c_1(T_X)=0$.

\begin{lemma}[Weight of the symplectic form]\label{lem:omega-weight}
For the Higgs-scaling action,
\[
 t^*\omega=t\,\omega.
\]
Thus \(\omega\) has \(T\)-weight \(1\).
\end{lemma}

\begin{proof}
At the level of the deformation complex, a tangent vector consists schematically of a deformation of the bundle and a deformation of the Higgs field. The symplectic form pairs the bundle deformation in one tangent vector with the Higgs-field deformation in the other, minus the reverse pairing. Higgs scaling acts with weight \(0\) on the bundle direction and weight \(1\) on the Higgs-field direction. Every summand in the symplectic pairing therefore has weight \(1\).
\end{proof}

\subsection{Hitchin morphism}\label{subsec_Hitchin_map}

For every \(S\)-family in \(\mathfrak X_\Lambda\), the linear coefficient of the characteristic polynomial is \(-\Tr(\phi)=0\) on \(C\times S\). Consequently, the characteristic-polynomial map factors scheme-theoretically, and compatibly with arbitrary base change, through the trace-free Hitchin base below. The \(\SL_r\) Hitchin morphism is
\[
 h:X\longrightarrow
 \A= \bigoplus_{i=2}^r H^0(C,K_C^i),
\]
given by the nontrivial coefficients of the characteristic polynomial of \(\phi\).  Since \(i\ge2\),
\(
 h^0(C,K_C^i)=(2i-1)(g-1)\), 
so
\[
 \dim\A
 = \sum_{i=2}^r(2i-1)(g-1)
 = (r^2-1)(g-1)
 = \frac12\dim X.
\]
The \(T\)-action on the \(i\)-th summand has weight \(i\). Hence $\A^T=\{0\}$.

\begin{theorem}\label{thm:hitchin-proper}
The Hitchin morphism \(h:X\to\A\) is proper, hence projective.
\end{theorem}

\begin{proof}
Properness of the characteristic-polynomial morphism for the moduli of semistable pairs can be found in \cite[Theorem 6.1]{Nitsure}; see also \cite{HitchinSI}. The fixed-determinant and trace-free locus is obtained by restricting to a closed condition and to the corresponding linear subspace of the Hitchin base, so properness is preserved by base change and passage to the closed subspace.

The variety \(X\) is quasi-projective by the GIT construction. A proper quasi-projective morphism is projective; see e.g. \cite[Th\'eor\`eme~5.5.3(ii)]{EGAII}. Thus \(h\) is projective.
\end{proof}

The morphism $p: X\to [X/\Gamma]$ commutes with the Hitchin morphism $h: X\to \A$. Thus, we have diagram (\ref{eqn_Hitchin_map}).

\section{Steinberg cycles}\label{sec_Steinberg}

\subsection{Symplectic cosection}

Apply Proposition \ref{prop:kernel-cone-general} to \(M=\sM_{n,\beta}\), its standard absolute obstruction theory $E^\bullet$, and \(\mathcal L=\cO_{\sM_{n,\beta}}\otimes\C_1\). Write
\[
 \mathfrak E=h^1/h^0\bigl((E^\bullet)^\vee\bigr),
 \qquad
 \mathfrak K=\ker(\widetilde\sigma_\omega).
\]
Kiem--Li's cone statement becomes
\begin{equation}\label{eq:cone-kernel}
 [\mathfrak C_{\sM_{n,\beta}}]\in Z_*^T(\mathfrak K).
\end{equation}
We define the unlocalized equivariant reduced class by
\begin{equation}\label{eq:red-class-kernel}
 [\sM_{n,\beta}]^{\red}_T
 :=0^!_{\mathfrak K}[\mathfrak C_{\sM_{n,\beta}}]
 \in A_{m-2+n}^T(\sM_{n,\beta})_\qq.
\end{equation}
This is the kernel-cone reduced cycle described in \cite[Introduction, p.~1026]{KL}. It does not require asserting that the homotopy fiber of \eqref{eq:cosection-complex} is itself a global Behrend--Fantechi perfect obstruction theory. In special symplectic settings, the kernel-cone cycle is induced by a reduced perfect obstruction theory; see e.g.
\cite[\S3.4]{OPHilb} and \cite[\S4.2]{BMO}. No such assertion is needed for the Higgs target here.

\begin{proposition}\label{prop:vir-red}
For every \(\beta\ne0\),
\begin{equation}\label{eq:vir-red}
 [\sM_{n,\beta}]^{\vir}_T
 =c_1^T(\C_1)\cap[\sM_{n,\beta}]^{\red}_T
 =\hbar[\sM_{n,\beta}]^{\red}_T .
\end{equation}
\end{proposition}
\begin{proof}
This is \eqref{eq:abstract-excess} with \(\mathcal L=\cO\otimes\C_1\). Since \(c_1^T(\C_1)=\hbar\), the sign is positive: the removed obstruction quotient is the weight-\(+1\) line.
\end{proof}

\begin{remark}
After forgetting equivariance the quotient line is trivial, so \([\sM_{n,\beta}]^{\vir}=0\).  This also follows directly from Kiem--Li cosection localization, since the degeneracy locus is empty. The equivariant Euler class of the removed line gives the nonzero factor \(\hbar\).
\end{remark}

\subsection{The isotropy theorem for two-point evaluation images}\label{sec:isotropy}

Here we consider the geometric heart of the Steinberg statement. We prove a result stronger than what is needed for the virtual fundamental class: \emph{every} irreducible two-point evaluation image of a genus-zero family is isotropic.

Put
\[
 \Omega=p_1^*\omega-p_2^*\omega \in H^0(X\times X,\Omega_{X\times X}^2).
\]

\begin{lemma}\label{lem:basicness}
Let \(S\) be a smooth integral scheme, let \(\pi:\sC\to S\) be a smooth proper family of geometrically connected genus-zero curves, and let \(F:\sC\to X\) be a morphism. There is a unique \(\eta_F\in H^0(S,\Omega_S^2)\) such that
\[
 F^*\omega=\pi^*\eta_F.
\]
\end{lemma}

\begin{proof}
Since \(\pi\) is smooth of relative dimension one, the cotangent filtration gives
\[
 0\longrightarrow\pi^*\Omega_S^2
 \longrightarrow\Omega_\sC^2
 \longrightarrow
 \pi^*\Omega_S^1\otimes\Omega_{\sC/S}^1
 \longrightarrow 0.
\]
For a smooth proper genus-zero family,
\[
 \pi_*\Omega_{\sC/S}^1=0,
 \qquad
 \pi_*\cO_\sC=\cO_S.
\]
The projection formula therefore shows that the last term has no global sections, whereas
\[
 H^0(\sC,\pi^*\Omega_S^2)=H^0(S,\Omega_S^2).
\]
Thus every regular two-form on \(\sC\), in particular \(F^*\omega\), is the pullback of a unique two-form on \(S\).
\end{proof}

\begin{lemma}\label{lem:nodal-basicness}
Let \(\pi:\mathcal D\to S\) be a proper flat family of connected nodal curves of arithmetic genus $0$ over a smooth integral scheme. Let \(s_1,s_2:S\to\mathcal D\) be sections through the smooth locus, and let \(F:\mathcal D\to X\).  Then
\begin{equation}\label{eq:marking-forms}
 (F\circ s_1)^*\omega=(F\circ s_2)^*\omega \quad\text{in }H^0(S,\Omega_S^2).
\end{equation}
\end{lemma}

\begin{proof}
First, shrink to a dense equisingular open subset of \(S\). After a finite \'etale surjective base change, the components, nodes, and branches are labelled, and the simultaneous normalization is a disjoint union of smooth proper genus-$0$ families
\[
 \bigsqcup_a \sC_a\longrightarrow S.
\]
Write \(F_a:\sC_a\to X\) for the induced maps. By Lemma \ref{lem:basicness},
\[
 F_a^*\omega=\pi_a^*\eta_a
\]
for two-forms \(\eta_a\) on \(S\). If branch sections \(q_a:S\to \sC_a\) and \(q_b:S\to \sC_b\) form a node, then
\(F_a\circ q_a=F_b\circ q_b\), whence
\[
 \eta_a=q_a^*F_a^*\omega
 =(F_a\circ q_a)^*\omega
 =(F_b\circ q_b)^*\omega
 =q_b^*F_b^*\omega=\eta_b.
\]
The dual graph (which is a tree) is connected, so all \(\eta_a\) are equal. Pulling back at the two marked sections proves \eqref{eq:marking-forms} after the \'etale cover, and faithfully flat descent proves it before the cover.

Finally, the two sides of \eqref{eq:marking-forms} are regular sections of the locally free sheaf \(\Omega_S^2\). Equality on the dense equisingular open subset extends over all of \(S\). This extension is what includes tangent
directions which smooth nodes occurring only in special fibers; moving markings and moving node sections were already included in the argument.
\end{proof}

\begin{proposition}\label{prop:isotropy}
Let \(W\subset \sM_{2,\beta}\) be an integral locally closed substack, and let
\[
 Z=\overline{(\ev_1,\ev_2)(W)}\subset X\times X.
\]
Then
\[
 \Omega|_{Z_{\mathrm{sm}}}=0.
\]
Consequently
\[
 \dim Z\le m.
\]
If \(\dim Z=m\), then \(Z\) is Lagrangian.
\end{proposition}
\begin{proof}
Choose an irreducible component of an \'etale scheme atlas of \(W\) which dominates \(W\), and replace it by a dense smooth integral open \(S\). Pull back the universal curve, its two sections, and its map to \(X\). Applying Lemma \ref{lem:nodal-basicness} gives
\[
 g^*\Omega=0, \qquad g=(\ev_1,\ev_2)|_S:S\longrightarrow X\times X.
\]
Finite stabilizers cause no issue, since the identity of differential forms is \'etale local on the source.

The morphism \(S\to Z\) is dominant. Over \(\C\), the induced extension of function fields is separable, so generic smoothness gives a dense open of \(Z_{\mathrm{sm}}\) over which \(dg:T_sS\to T_lZ\) is surjective at general points.  The equality \(g^*\Omega=0\) then implies \(\Omega|_{T_lZ}=0\). As \(\Omega|_{Z_{\mathrm{sm}}}\) is a regular two-form, vanishing on a dense open implies vanishing everywhere on \(Z_{\mathrm{sm}}\). Thus \(Z\) is isotropic. Since \(X\times X\) is symplectic of dimension \(2m\), every isotropic subvariety has dimension at most \(m\), and equality is exactly the Lagrangian condition.
\end{proof}

\begin{remark}
The argument does not use the virtual fundamental class. This is important: it controls the geometric support of the entire stable-map stack and thereby excludes equivariant terms supported on subvarieties of dimension \(>m\).
\end{remark}

\subsection{Hitchin-fiber support}\label{subsec_Hitchin_fiber_support}

By Lemma \ref{lem:hitchin-constant}, the two evaluations of every stable map have the same Hitchin value. Therefore,
\begin{equation}\label{eq:fiber-product-support}
 (\ev_1,\ev_2)(\sM_{2,\beta}) \subset X\times_\A X.
\end{equation}
By Proposition \ref{prop:stablemap-proper}, the evaluation map is proper, so its reduced image
\[
 Y_\beta:=\operatorname{Im}(\ev_1,\ev_2)_{\red}
\]
is a closed finite-type reduced subscheme. It is \(T\)-stable: every element of the connected torus acts on \(X\) by a map homotopic to the identity, hence acts trivially on \(H_2(X,\zz)\). Thus \(T=\mathbb{C}^*\) acts on \(M_\beta\) and preserves its image. The evaluation map first factors through its scheme-theoretic image; equivariant Chow groups are unchanged by nilpotent thickenings, so all pushforward classes may be viewed on the reduction \(Y_\beta\).

\begin{corollary}\label{cor:Y-isotropic}
Every irreducible component of \(Y_\beta\) is isotropic in \((X\times X,\Omega)\). Hence
\[
 \dim Y_\beta\le m.
\]
Every \(m\)-dimensional irreducible component of \(Y_\beta\) is a \(T\)-invariant Steinberg subvariety.
\end{corollary}

\begin{proof}
Apply Proposition \ref{prop:isotropy} to irreducible components of \(\sM_{2,\beta}\) whose images dominate the irreducible components of \(Y_\beta\), and use \eqref{eq:fiber-product-support}. If an image has dimension \(m\), it is
Lagrangian. Since \(T\) is connected, it preserves every member of the finite set of irreducible components of the \(T\)-stable scheme \(Y_\beta\). Thus each top-dimensional component is a Steinberg subvariety.
\end{proof}

\subsection{A top-dimensional equivariant Chow lemma}

The following elementary observation is what upgrades the geometric support statement to the parameter-independent cycle statement \eqref{eq:Zbeta-Steinberg}.

\begin{lemma}\label{lem:top-equiv-chow}
Let \(Y\) be a finite-type \(T\)-scheme of dimension at most \(m\). Then \(A_m^T(Y)_\qq\) is the \(\qq\)-span of the equivariant fundamental classes of the \(m\)-dimensional irreducible components of \(Y\). In particular, no positive power of \(\hbar\) can occur as the coefficient of a higher-dimensional geometric support, because such a support does not exist.
\end{lemma}
\begin{proof}
We use the finite-dimensional approximation \cite{EG}: consider
\[
 U_N=\mathbb A^N\setminus\{0\},
\]
on which \(T\) acts freely by scalar multiplication. The $N$-th Borel construction is 
\begin{equation}\label{eqn:Nth_Borel}
Y\times^T U_N.
\end{equation}
 For \(N\) sufficiently large in the range under consideration, we have 
\[
 A_m^T(Y)
 =
 A_{m+N-1}(Y\times^T U_N).
\]
The mixed quotient has dimension at most \(m+N-1\). Hence its Chow group in dimension \(m+N-1\) is generated by the fundamental classes of its top-dimensional irreducible components. Those components occur precisely from the \(m\)-dimensional \(T\)-invariant irreducible components of \(Y\). Because \(T\) is connected, it preserves each member of the finite set of irreducible components. This proves the claim.  
\end{proof}

\subsection{The reduced correspondence is Steinberg}

\begin{theorem}\label{thm:Z-Steinberg}
For every \(\beta\ne0\), the class
\[
 Z_{2,\beta}^{\red} = (\ev_1,\ev_2)_*[\sM_{2,\beta}]^{\red}_T
\]
is a finite linear combination of (classes of) Steinberg subvarieties with $\mathbb{Q}$-coefficients. Hence the standard evaluation cycle satisfies
\[
 C_\beta=\hbar  Z_{2,\beta}^{\red} =\hbar\sum_j c_{\beta,j}[L_{\beta,j}]_T
 \quad\text{in }A_{m-1}^T(Y_\beta)_\qq.
\]
\end{theorem}

\begin{proof}
The reduced virtual class has complex dimension \(m\) by \eqref{eq:vdim-reduced}. Since the evaluation map is proper, we have 
\[
  Z_{2,\beta}^{\red} \in A_m^T(Y_\beta)_\qq.
\]
By Corollary \ref{cor:Y-isotropic},
\[
 \dim Y_\beta\le m.
\]
Apply Lemma \ref{lem:top-equiv-chow}. We obtain
\[
  Z_{2,\beta}^{\red} = \sum_j c_{\beta,j}[L_{\beta,j}]_T, \qquad c_{\beta,j}\in\qq\setminus \{0\},
\]
where the \(L_{\beta,j}\) are the \(m\)-dimensional irreducible components that occur with nonzero coefficient. By
Corollary \ref{cor:Y-isotropic}, each \(L_{\beta,j}\) is Lagrangian and lies in \(X\times_\A X\). Hence each is a Steinberg subvariety.
\end{proof}

\begin{remark}
We give another proof of the existence of the Steinberg factor after evaluation, which does not use equation \eqref{eq:red-class-kernel}. Forgetting the torus action, the everywhere surjective Kiem--Li cosection has empty degeneracy locus. Their cosection localization theorem therefore gives
\[
 [\sM_{2,\beta}]^{\vir}=0.
\]
It follows that the image of \(C_\beta\in A_{m-1}^T(Y_\beta)_\qq\) under the forgetful map to \(A_{m-1}(Y_\beta)_\qq\) is zero. Brion's augmentation theorem for arbitrary finite-type torus schemes \cite[Corollary~2.3.1]{Brion} states that
\[
 A_*^T(Y_\beta)/(\hbar)\cong A_*(Y_\beta),
 \qquad T=\cc^*.
\]
Hence there exists \(\delta_\beta\in A_m^T(Y_\beta)_\qq\) with
\[
 C_\beta=\hbar\delta_\beta.
\]
By Corollary \ref{cor:Y-isotropic} and Lemma \ref{lem:top-equiv-chow}, \(\delta_\beta\) is a $\mathbb{Q}$-linear combination of the \(m\)-dimensional Lagrangian components of \(Y_\beta\). This proves the required Steinberg divisibility independently. The Brion lift can be nonunique when \(\hbar\)-torsion is present; the kernel-cone construction gives the distinguished lift \(\delta_\beta= Z_{2,\beta}^{\red}\).
\end{remark}

\begin{remark}(Comparison with \cite{MO})
For an equivariant symplectic resolution $X\to X_0$, Maulik--Okounkov state that the pushforward of the reduced two-point virtual class is a rational linear combination of Steinberg correspondences \cite[\S 7.2.1]{MO}. The preceding argument proves the analogous assertion for the smooth coprime Higgs moduli space directly, using the Hitchin morphism in place of the affinization map. 
\end{remark}

\section{Quantum multiplication by divisors}\label{sec:divisor}

\subsection{Equivariant quantum cohomology of the nonproper target}\label{subsec_equivariant_quan}

Because \(X^T\subset h^{-1}(0)\) and the nilpotent cone is projective, \(X^T\) is proper. Thus the localized equivariant Poincar\'e pairing
\[
 (\alpha,\gamma)_T = \int_X^T \alpha\cup\gamma
\]
is defined by localization. It is nondegenerate after inverting \(\hbar\): the localization isomorphism identifies localized equivariant cohomology of \(X\) with that of the smooth proper fixed locus \(X^T\), and the pairing
becomes the proper Poincar\'e pairing on \(X^T\), multiplied on each fixed component by the inverse equivariant Euler class of its normal bundle. Likewise, Corollary \ref{cor:fixed-proper} makes the fixed-locus contributions proper
and therefore allows the standard \(T\)-equivariant Gromov--Witten invariants to be pushed to a point after applying virtual localization \cite{GP,Givental}.

Fix an \(h\)-ample line bundle \(H\). Because the Hitchin base $\mathbb{A}$ is affine, \(H\) is ample on \(X\), and \(H\cdot\beta>0\) for every nonzero effective class. Let \(\mbox{Eff}(X)\subset H_2(X,\zz)\) be the semigroup of integral
homology classes represented by effective algebraic curves. Define
\[
 \Lambda_H
 =
 \left\{
 \sum_{\beta\in\text{Eff}(X)}a_\beta q^\beta \, \Big| \, a_\beta\in\qq,\ 
 \#\{\beta\, |\, a_\beta\ne0,\ H\cdot\beta\le N\}<\infty
 \text{ for every }N
 \right\}.
\]
This is a completion of the effective semigroup algebra with the displayed support condition. Multiplication on $\Lambda_H$ is well defined since only finitely many terms of each factor occur in any bounded degree. The genus-zero Gromov-Witten generating series used here satisfy this support condition: for fixed markings, the relative stable-map stacks of each $H$-degree are of finite type, and the homology class is locally constant on their connected components, as in Proposition~\ref{prop:stablemap-proper}. Taking finitely many degrees therefore gives finitely many classes with nonempty stable-map spaces. 

To consider generating series of Gromov-Witten invariants of $X$, we work over $\qq(\hbar)\widehat\otimes\Lambda_H$.

\subsection{The correspondence action in Borel--Moore homology}

Put \(W=X\times_\A X\), let \(S_T=H_T^*(\mathrm{pt},\qq)=\qq[\hbar]\), and write 
\((-)_{\mathrm{loc}}=(-)\otimes_{S_T}\qq(\hbar)\). For \(Z\in A_d^T(W)_\qq\), the equivariant cycle-class map gives
\[
 \operatorname{cl}_T^{\mathrm{BM}}(Z)
 \in H_{2d}^{T,\mathrm{BM}}(W,\qq).
\]
The second projection
\[
 p_2:W\longrightarrow X
\]
is proper, because it is the base change of the proper Hitchin morphism. Thus proper pushforward is available in equivariant Borel--Moore homology. Since \(X\) is smooth of complex dimension \(m\), equivariant Poincar\'e
duality gives
\[
 \operatorname{PD}_X: H_T^k(X)_{\mathrm{loc}} \xrightarrow{\,\sim\,} H_{2m-k}^{T,\mathrm{BM}}(X)_{\mathrm{loc}}.
\]
We define
\begin{equation}\label{eq:Act-def}
 \mathbf{F}_{Z}(\alpha)
 :=\operatorname{PD}_X^{-1}
 \left((p_2)_*\bigl(p_1^*\alpha\cap\operatorname{cl}_T^{\mathrm{BM}}(Z)\bigr)\right).
\end{equation}
This definition requires neither \(X\) nor \(W\) to be proper. It shifts cohomological degree by \(2(m-d)\), and in particular a Steinberg cycle, for which \(d=m\), acts with $0$ degree shift.

Because \(X^T\) is proper, \(W^T\) is proper as well. It is closed in the proper inverse image \(p_2^{-1}(X^T)\).  Localized equivariant integration and the projection formula therefore give
\begin{equation}\label{eq:BM-adjointness}
 \bigl(\mathbf{F}_Z(\alpha),\gamma\bigr)_T
 =
 \int_W^T p_1^*\alpha\,p_2^*\gamma \cap\operatorname{cl}_T^{\mathrm{BM}}(Z)
 =
 \int_Z^T p_1^*\alpha\,p_2^*\gamma.
\end{equation}
The familiar shorthand \((p_2)_*(p_1^*\alpha\cap\operatorname{cl}^{\text{BM}}_T(Z))\) is henceforth understood as \eqref{eq:Act-def}. Let \(\Act(\St_{\rm cyc})\) denote the linear space of operators induced by Steinberg cycles. 
The resulting linear map from the Steinberg cycle space to endomorphisms of localized equivariant cohomology need not be
injective, which is why we distinguish the space \(\St_{\rm cyc}\) of Steinberg ccyles from \(\Act(\St_{\rm cyc})\).

\subsection{Divisor operators on $X$}

Let \(D_T\in H^2_T(X)\) be an equivariant lift of an ordinary divisor \(D\). By definition,
\[
 (D_T\star\alpha,\gamma)_T
 =
 \sum_{\beta\ge0} q^\beta \langle D_T,\alpha,\gamma\rangle_{0,3,\beta}^{X,T}.
\]
The degree-zero term is ordinary cup product:
\[
 \langle D_T,\alpha,\gamma\rangle_{0,3,0}
 =
 \int_X^T D_T\cup\alpha\cup\gamma.
\]
For \(\beta\ne0\), the divisor equation gives
\begin{equation}\label{eq:divisor-equation}
 \langle D_T,\alpha,\gamma\rangle_{0,3,\beta}
 =
 (D\cdot\beta) \cdot\langle\alpha,\gamma\rangle_{0,2,\beta}.
\end{equation}

By Proposition \ref{prop:vir-red},
\begin{equation} \label{eq:two-point-reduced}
 \langle\alpha,\gamma\rangle_{0,2,\beta}
 =
 \int^T_{[\sM_{2,\beta}]^{\vir}_T}\ev_1^*\alpha\,\ev_2^*\gamma
 =
 \hbar\int^T_{[\sM_{2,\beta}]^{\red}_T}\ev_1^*\alpha\,\ev_2^*\gamma
 =
 \hbar\int^T_{Z_{2,\beta}^{\red}} p_1^*\alpha\,p_2^*\gamma.
\end{equation}
The last equality is the projection formula for the proper evaluation map.

By equation \eqref{eq:Act-def} and the projection formula,
\[
 \int_Z^T p_1^*\alpha\,p_2^*\gamma
 =
 (\mathbf{F}_Z(\alpha),\gamma)_T.
\]
Substituting into \eqref{eq:divisor-equation} gives
\begin{align*}
 (D_T\star\alpha,\gamma)_T
 =
 (D_T\cup\alpha,\gamma)_T
 +
 \left(\hbar \sum_{\beta>0} (D\cdot\beta)q^\beta \mathbf{F}_{Z_{2,\beta}^{\red}}(\alpha), \gamma \right)_T.
\end{align*}
Nondegeneracy of the localized equivariant pairing proves:

\begin{theorem}[The divisor correction is Steinberg-valued]\label{thm:divisor}
For every divisor \(D\) equipped with an equivariant lift \(D_T\),
\[
 M_D
 =
 D_T\cup
 +
 \hbar\sum_{\beta>0}(D\cdot\beta)q^\beta\mathbf{F}_{Z_{2,\beta}^{\red}}.
\]
Equivalently,
\[
 \frac{M_D-D_T\cup}{\hbar}
 \in
 \Act\!\left(\widehat{\St}_{\rm cyc}(X/\A)\right).
\]
Taking \(D=\Theta\) gives the formula for $\Theta$ in Theorem  \ref{thm:main}.
\end{theorem}

\begin{lemma}\label{lem:lift-independence}
If \(D_T'\) and \(D_T\) are two equivariant lifts of the same ordinary divisor, then the positive-degree parts of their quantum multiplication operators agree.
\end{lemma}
\begin{proof}
The equivariant divisor equation in \eqref{eq:divisor-equation} uses only the degree of the underlying ordinary divisor on \(\beta\). Both lifts have the same underlying class, so their coefficients \(D\cdot\beta\) are equal for every 
\(\beta\). The positive-degree formula in Theorem \ref{thm:divisor} is therefore unchanged. The classical operators differ by cup product with \(D_T'-D_T\), and no assertion that this difference is a scalar is needed.
\end{proof}

\subsection{Divisor operators on \texorpdfstring{$[X/\Gamma]$}{the quotient}}
\label{subsec_quantum_orbifold}
Put $\widehat X=[X/\Gamma]$. Its unrigidified inertia stack is
\begin{equation}\label{eqn_inertia_stack}
 I\widehat X=\coprod_{\gamma\in\Gamma}[X^\gamma/\Gamma].
\end{equation}
Each connected sector $\mathcal F=[F/\Gamma_F]$ is smooth and symplectic: $F$ is a fixed component of a finite symplectic automorphism. Set $d_{\mathcal F}=\dim_\C F$, $a_{\mathcal F}=\age(\gamma|T_X)$, and 
$s_{\mathcal F}=\operatorname{ord}(\gamma)$. The eigenvalue pairing proves
\[
 d_{\mathcal F}=m-2a_{\mathcal F},\qquad a_{\mathcal F}\in\Z.
\]
The rigidified sector is $\overline{\mathcal F}=[F/(\Gamma_F/\langle\gamma\rangle)]$; write $\rho_{\mathcal F}:\mathcal F\to\overline{\mathcal F}$. The standard evaluation maps for stable maps take values in $\overline I\widehat X$, not directly in $I\widehat X$, see \cite{AGV}. 

For completeness, the cover description of a nonidentity fixed locus uses a cyclic \'etale cover $\pi:C'\to C$ of degree $s$ and bundles of rank $r/s$ on $C'$. The fixed-determinant, trace-free fiber is cut out by the map
\[
 (E',\phi')\longmapsto
 \bigl(\det(\pi_*E'),\operatorname{Tr}_{C'/C}(\operatorname{tr}\phi')\bigr)
 \in\Pic^d(C)\times H^0(C,K_C).
\]
The pushforward $\pi_*E'$ has rank $r$, so it cannot itself be the value of a map to a rank-one Higgs moduli space.  We need only the fixed-locus geometry, not a global identification of its base with an unqualified quotient; see \cite{MS} for the cover construction.

For a nonzero effective coarse curve class $\bar\beta$, let $\mathfrak M_{n,\bar\beta}(\mathcal F_1,\ldots,\mathcal F_n)$ be the stack of (representable) twisted stable maps with the specified marking sectors. On this stack, the standard virtual dimension is
\[
 m-3+n-\sum_i a_{\mathcal F_i}.
\]
The symplectic cosection is constructed by 
$f^*T_{\widehat X}\to\Omega_{\mathcal C}\otimes\C_1\to\omega_{\mathcal C}\otimes\C_1$ and relative duality.
The absolute compatibility calculation is the same alternating-form calculation as for smooth targets in \cite[Lemma~6.1]{KL}: on the smooth locus the composite involving two copies of $df$ vanishes by skew-symmetry, and balanced node branches give opposite residues. This calculation is valid in tame \'etale charts of the twisted curve
and is invariant under their finite stabilizers. It therefore descends to the absolute obstruction sheaf. More explicitly, pulling back $X\to[X/\Gamma]$ along the map gives a principal $\Gamma$-cover $P\to\mathcal C$ and a $\Gamma$-equivariant map $P\to X$. Representability makes $P$ a nodal algebraic space curve, and the usual two-form calculation applies to it, componentwise if it is disconnected. Deformations of the twisted source curve induce deformations of $P$, whose contribution is eliminated by the absolute Kiem--Li calculation. Taking finite-group invariants is exact in characteristic zero; normalized trace identifies the invariant duality map with the one on $\mathcal C$. This verifies descent through the absolute obstruction quotient, including smoothing nodes. Duality on a connected tame twisted curve gives $H^1(\omega_{\mathcal C})=\C$; the dual map is nonzero on a nonconstant component. Thus it is surjective in nonzero degree. Proposition~\ref{prop:kernel-cone-general} gives a reduced \emph{cycle}, not an asserted reduced perfect obstruction theory, with
\begin{equation}\label{eqn_virtual_dimension_orbifold}
 [\mathfrak M_{n,\bar\beta}]^{\vir}_T
  =\hbar[\mathfrak M_{n,\bar\beta}]^{\red}_T,\qquad
 \operatorname{vdim}^{\red}=m-2+n-\sum_i a_{\mathcal F_i}.
\end{equation}
In particular, in the three-pointed case, the reduced dimension is $m+1-\sum_i a_{\mathcal F_i}$, whereas the dimension in the two-pointed case is
\[
 D_{12}=m-a_{\mathcal F_1}-a_{\mathcal F_2} =\tfrac12(d_{\mathcal F_1}+d_{\mathcal F_2}).
\]

\begin{lemma}[Hitchin value and isotropy for twisted maps]
\label{lem_hitchin_constant_twisted}
Every connected proper twisted curve maps constantly to the affine Hitchin base. Every two-point evaluation image of a genus-zero twisted stable-map family is isotropic in the product of its sector symplectic forms, with the difference sign.
\end{lemma}
\begin{proof}
A morphism from a twisted curve to an affine scheme factors through the coarse curve, whose global regular functions are constant. This also holds in families, and defines the Hitchin-value map. Properness of twisted stable maps
to a proper tame target, applied relatively over $\A$, makes the evaluation maps proper; bounded degree is measured by an ample line on the coarse target \cite{AGV}.

For isotropy, shrink a smooth parameter space to a dense equisingular open and label components and branches after an \'etale cover. The normalizations are smooth proper twisted genus-zero families. For a smooth twisted rational curve with coarse map $c$, a local root coordinate $z=x^s$ shows $(c_*\Omega^1)=(\Omega^1_{\mathbb P^1})$: an invariant regular differential is a regular multiple of $x^{s-1}dx=dz/s$. Hence the pushforward of relative one-forms has no sections and the pushforward of the structure sheaf is $\cO$. The cotangent-filtration proof of Lemma~\ref{lem:basicness} therefore applies on each component. The pulled-back two-form is basic. At a node its two restrictions agree after pullback to the branch gerbe, so the basic forms agree on adjacent components and hence on the whole tree. Pullback at the marking gerbes gives equality of the two evaluation forms. The gerbes can be neutralized \'etale locally for this calculation, and the equality descends. Extension over the parameter space and generic smoothness prove isotropy, exactly as in Proposition~\ref{prop:isotropy}.
\end{proof}

The two-pointed cycles are
\begin{equation}\label{eq:fiber-product-support_stacky}
 Z_{\bar\beta}^{\red,(\mathcal F_1,\mathcal F_2)}
 :=(\ev_1,\ev_2)_*[\mathfrak M_{2,\bar\beta}(\mathcal F_1,\mathcal F_2)]^{\red}_T
 \in A_{D_{12}}^T(\overline{\mathcal F}_1\times_\A\overline{\mathcal F}_2)_\Q.
\end{equation}
The isotropy lemma bounds the image dimension by $D_{12}$. The top-dimensional equivariant Chow argument, applied on finite quotient presentations, shows that each cycle is a finite rational linear combination of Lagrangian fundamental classes. The term ``Lagrangian cycle'' permits the zero class and coefficients of either sign, and it does not mean one irreducible reduced subvariety. Write $Z^{\PGL_r}_{\bar\beta}$ for this finite collection of sectorwise cycles, rather than for a disjoint union with one uniform dimension.
\begin{equation}\label{eqn_Z_beta_PGL}
 Z^{\PGL_r}_{\bar\beta}
   :=\bigl(Z_{\bar\beta}^{\red,(\mathcal F_1,\mathcal F_2)}\bigr)_{\mathcal F_1,\mathcal F_2}.
\end{equation}
The marked group elements satisfy $\gamma_1\gamma_2=1$ by
Lemma~\ref{lem:two-point-monodromy}.

We fix all operator normalizations explicitly. Let $\bar\iota$ be the inversion on rigidified inertia and set
$\check\ev_2=\bar\iota\circ\ev_2$. On cohomology og the rigidified inertia, the orbifold Poincar\'e pairing is
\begin{equation}\label{eqn:pairing_2}
 (a,b)_{\orb,T}
=\sum_{\overline{\mathcal F}} \frac1{s_{\mathcal F}}\int^T_{\overline{\mathcal F}}a\cup\bar\iota^*b.
\end{equation}
This is the unrigidified stack pairing transported by $\rho^*$. Poincar\'e duality on a sector uses $2d_{\mathcal F}-p$, not $2m-p$. 
\begin{definition}[Normalized reduced orbifold operator]
\label{defn_Steinberg_PGL}
For an input on $\overline{\mathcal F}_1$, the contribution with second marking in $\overline{\mathcal F}_2$ to the output on $\bar\iota(\overline{\mathcal F}_2)$ is
\begin{equation}\label{eq:Act-def_stacky}
 F^{\widehat X,\rig,\CR}_{\bar\beta}(a)
  :=\sum_{\mathcal F_2}s_{\mathcal F_2}
   \operatorname{PD}^{-1}_{\bar\iota(\overline{\mathcal F}_2)}
    (\check\ev_2)_*\left(\ev_1^*a\cap[\mathfrak M_{2,\bar\beta} (\mathcal F_1,\mathcal F_2)]^{\red}_T\right).
\end{equation}
Cycle classes are understood in the Borel--Moore expression. The factor $s_{\mathcal F_2}$ is forced by the pairing (\ref{eqn:pairing_2}). The operator preserves Chen--Ruan degree; on underlying ordinary cohomology the degree shift is $2(a_{\mathcal F_1}-a_{\mathcal F_2})$.
\end{definition}

\begin{theorem}[Untwisted orbifold divisor formula]
\label{thm:divisor_PGL_r}
For an untwisted divisor $D\in H^2(\widehat X,\Q)$ with equivariant lift $D_T$,
\[
 M_D=D_T\cup_{\CR}(-)+
   \hbar\sum_{\bar\beta>0}(D\cdot\bar\beta) q^{\bar\beta}F^{\widehat X,\rig,\CR}_{\bar\beta}.
\]
Its positive-degree correction is sectorwise Steinberg-valued with the normalizations in Definition~\ref{defn_Steinberg_PGL}.
\end{theorem}
\begin{proof}
Pair the proposed operator with a class on the second marking sector. The output factor $s_{\mathcal F_2}$ cancels the pairing factor $1/s_{\mathcal F_2}$, and the inversion converts the tested output back to $\ev_2^*b$. The result is exactly the reduced two-point invariant. Multiplication by $\hbar$ recovers the standard invariant. The orbifold divisor equation for an untwisted insertion \cite{AGV} gives $(D\cdot\bar\beta)$; the degree-zero term is $D_T\cup_{\CR}-$. Nondegeneracy after localization proves the formula. 
\end{proof}

\section{Properties of Steinberg cycles}\label{sec_property_Steinberg}

\subsection{Support over the spectral discriminant}

Let \(\A^{\mathrm{sp,sm}}\subset\A\) be the open locus on which the spectral curve is smooth, and define
\[
 \Delta:=\A\setminus\A^{\mathrm{sp,sm}}.
\]
For \(a\in\A^{\mathrm{sp,sm}}\), the Beauville--Narasimhan--Ramanan description identifies the fixed-determinant
fiber with a torsor under the kernel of the norm map. Consequently, each connected component of \(h^{-1}(a)\) is a torsor under the identity component
\[
 \operatorname{Prym}(\Sigma_a/C)^0
 :=\ker(\operatorname{Nm}:\operatorname{Pic}(\Sigma_a) \to\operatorname{Pic}(C))^0,
\]
which is an abelian variety \cite{BNR}. Thus no component contains a rational curve. In what follows, we just denote $Z_{\beta}^{\red}:=Z_{2,\beta}^{\red}$.

\begin{corollary}[Discriminant support]\label{cor:discriminant}
For every \(\beta\ne0\), we have
\[
 \Supp Z_\beta^{\red}
 \subset
 h^{-1}(\Delta)\times_\A h^{-1}(\Delta).
\]
\end{corollary}

\begin{proof}
By Lemma \ref{lem:hitchin-constant}, every genus-$0$ stable map lies in one Hitchin fiber. If its Hitchin value lies in
\(\A^{\mathrm{sp,sm}}\), the restriction of the map to every irreducible \(\mathbb P^1\)-component is constant. Thus
the entire connected stable map is constant, contradicting \(\beta\ne0\). 
\end{proof}

\subsection{Nilpotent-cone localization versus global support}

There are two different support statements which should not be conflated.

Globally,
\[
 \Supp Z_\beta^{\red} \subset X\times_\A X
\]
and, more strongly for \(\beta\ne0\), the support lies over \(\Delta\).

On the other hand, when the \(T\)-equivariant invariant is computed by virtual localization, a \(T\)-fixed stable map has Hitchin value in \(\A^T=\{0\}\). Therefore all localization graphs lie in the nilpotent cone
\[
 \Nnil=h^{-1}(0).
\]
Thus, global Steinberg support is Hitchin-fiberwise, whereas fixed-locus localization takes place over $0$. The latter statement does not imply that the global correspondence itself is supported only on \(\Nnil\times\Nnil\).

\section{Character decomposition and its endoscopic interpretation}
\label{sec_endoscopic_decomposition}

\subsection{The torsion action and theta class}\label{sec:torsion}
For $\eta\in\Gamma=\Pic^0(C)[r]$, let
\begin{equation}\label{eq:tau-action}
 \tau_\eta(E,\phi)=(E\otimes\eta,\phi\otimes\id_\eta).
\end{equation}
This automorphism preserves the trace-free condition, determinant and characteristic polynomial, so it acts over $\A$.

\begin{lemma}\label{lem:theta-symplectic-invariant}
The action of $\Gamma$ commutes with $T$, preserves $\Omega_X$, and fixes $\Theta_T$ in rational equivariant Chow and cohomology.
\begin{equation}\label{eq:theta-invariant}
 \tau_\eta^*\Theta_T=\Theta_T.
\end{equation}
\end{lemma}
\begin{proof}
The identification $\End_0(E\otimes\eta)=\End_0(E)$ preserves the trace pairing on the deformation complex \eqref{eq:def-complex}. It therefore preserves the symplectic form and commutes with Higgs scaling. Use the normalization of \eqref{eq:theta-Euler-zero}, then the auxiliary bundle $F_0$ has rank $r$ and degree $r(g-1)-d$, where $d=\deg\Lambda$. Therefore
\[
 \chi(C,E\otimes F_0)=r\deg F_0+rd+r^2(1-g)=0.
\]
The determinant line
\begin{equation}\label{eq:theta-detline}
 \mathscr L_{F_0}=\det R\pi_*(\mathcal E\otimes F_0)^{-1}
\end{equation}
has trivial scalar weight and descends from the fixed-determinant stable Higgs moduli stack. Its linearization is the one induced by the action on the universal Higgs object; Higgs scaling leaves the underlying bundle unchanged.
Tensoring by $\eta$ replaces $F_0$ by $F_0\otimes\eta$. On the universal $\mu_r$-gerbe, equivariant Grothendieck--Riemann--Roch and $\ch(\eta)=1$ in rational Chow imply 
\begin{equation}\label{eq:theta-GRR}
 c_1^T(\mathscr L_{F_0\otimes\eta})=c_1^T(\mathscr L_{F_0}).
\end{equation}
Pullback to that gerbe is an isomorphism on rational equivariant Chow; thus this global equality descends to $X$. In particular, there is no undetermined multiple of $\hbar$ in \eqref{eq:theta-invariant}.
\end{proof}

\begin{proposition}\label{prop:red-invariance}
Postcomposition by $\tau_\eta$ identifies the reduced cycle for $\beta$ with that for $(\tau_\eta)_*\beta$. Consequently,
\begin{equation}\label{eq:diagonal-invariance}
 (\tau_\eta\times\tau_\eta)_*Z_\beta^{\red} =Z_{(\tau_\eta)_*\beta}^{\red}.
\end{equation}
For a finite $\Gamma$-stable set $B$ of nonzero effective curve classes, put 
$Z_B^{\red}=\sum_{\beta\in B}Z_\beta^{\red}$ and let $\mathbf F_B$ be its correspondence operator. Then $Z_B^{\red}$ is
symmetric and diagonally $\Gamma$-invariant.
\end{proposition}
\begin{proof}
The differential of $\tau_\eta$ identifies the standard obstruction complexes and commutes with the symplectic cosections. It identifies the intrinsic cones inside the kernels and hence their zero-section Gysin classes. This uses the kernel-cone construction, not an unproved reduced perfect obstruction theory. Proper evaluation pushforward gives \eqref{eq:diagonal-invariance}. Summation over $B$ gives diagonal invariance; interchange of markings gives symmetry.
\end{proof}

\subsection{Fourier decomposition}\label{subsec_endoscopic}
Work over $E_r=\Q(\mu_r)$ so that all $\Gamma$-characters split. The projectors $\Pcal_\kappa$ of \eqref{eq:projector} have denominator $|\Gamma|=r^{2g}$, and satisfy \eqref{eq:projector-relations}. Define
\begin{equation}\label{eq:kernel-projector-body}
 Z_{B,\kappa}^{\red}=
 \frac1{|\Gamma|}\sum_{\eta\in\Gamma}\kappa(\eta)^{-1}(\tau_\eta\times\id)_*Z_B^{\red}.
\end{equation}
The coefficients of this projected cycle lie in $E_r$, not necessarily in $\Q$. Its support is a union of translated Steinberg subvarieties. Write $\mathbf{F}_{B,\kappa}$ for the operator associated to $ Z_{B,\kappa}^{\red}$.

\begin{theorem}[Character blocks]\label{thm:block-decomposition}
For every such $B$ and every character $\kappa$,
\begin{align}
 \mathbf F_{B,\kappa}
 &=\mathbf F_B\Pcal_\kappa
  =\Pcal_\kappa\mathbf F_B
  =\Pcal_\kappa\mathbf F_B\Pcal_\kappa,
       \label{eq:block-operator}\\
 Z_B^{\red}&=\sum_\kappa Z_{B,\kappa}^{\red},
       \label{eq:block-kernel-sum}\\
 (Z_{B,\kappa}^{\red})^{\mathsf t}&=Z_{B,\kappa^{-1}}^{\red}.
       \label{eq:inverse-character-transpose}
\end{align}
In particular, individual blocks are symmetric for $r=2$; in general, symmetry is guaranteed for self-inverse characters, not for every character. For $\kappa\ne1$ the two theta contractions vanish in equivariant Chow:
\begin{equation}\label{eq:nontrivial-contractions}
 (p_2)_*(p_1^*\Theta_T\cap Z_{B,\kappa}^{\red})=0,
 \qquad
 (p_1)_*(p_2^*\Theta_T\cap Z_{B,\kappa}^{\red})=0.
\end{equation}
\end{theorem}
\begin{proof}
Diagonal invariance and symmetry give, for $g=\tau_\eta$,
\begin{equation}\label{eq:kernel-move-action}
 (g\times\id)_*Z_B^{\red}=(\id\times g^{-1})_*Z_B^{\red}.
\end{equation}
The projection formula therefore gives
\begin{equation}\label{eq:operator-equivariance}
 \mathbf F_Bg^*=g^*\mathbf F_B,
\end{equation}
where all cohomological pushforwards use Poincar\'e duality as in \eqref{eq:Act-def}. Also
\begin{equation}\label{eq:first-factor-action}
 \mathbf F_{(g\times\id)_*Z_B^{\red}}=\mathbf F_Bg^*.
\end{equation}
Averaging gives \eqref{eq:block-operator}; character orthogonality gives \eqref{eq:block-kernel-sum}. For transposition, substitute
$\xi=\eta^{-1}$:
\[
 (Z_{B,\kappa}^{\red})^{\mathsf t}
 =\frac1{|\Gamma|}\sum_\eta\kappa(\eta)^{-1}
          (\tau_\eta^{-1}\times\id)_*Z_B^{\red}
 =\frac1{|\Gamma|}\sum_\xi\kappa(\xi)
          (\tau_\xi\times\id)_*Z_B^{\red}
 =Z_{B,\kappa^{-1}}^{\red}.
\]
Finally, theta-invariance and the projection formula imply
\begin{equation}\label{eq:Chow-character-contraction}
 (p_2)_*(p_1^*\Theta_T\cap Z_{B,\kappa}^{\red})
 =\left(\frac1{|\Gamma|}\sum_\eta\kappa(\eta)^{-1}\right)(p_2)_*(p_1^*\Theta_T\cap Z_B^{\red})=0.
\end{equation}
Transposition and $\kappa^{-1}\ne1$ give the second contraction.
\end{proof}

\begin{corollary}[Character blocks for divisor operator]\label{cor:quantum-blocks}
The full theta operator is equivariant for the semilinear action 
$\tau_\eta^*\otimes(q^\beta\mapsto q^{(\tau_\eta^{-1})_*\beta})$. After any Novikov specialization invariant under the curve-class action, it preserves the character decomposition \eqref{eq:cohomology-character-decomposition}. For an individually invariant class $\beta>0$,
\begin{equation}\label{eq:quantum-character-blocks}
 [q^\beta](\Theta\star-)|_{H_T^*(X)_\kappa}
   =(\Theta\cdot\beta)\hbar\,\mathbf F_{\{\beta\},\kappa}.
\end{equation}
For an orbit sum, the analogous formula is the sum of these weighted operators over the orbit. After invariant specialization, the theta-square correction is in the trivial character block.
\end{corollary}
\begin{proof}
Functoriality gives $\tau_\eta^*\mathbf F_\beta(\tau_\eta^*)^{-1}=\mathbf F_{(\tau_\eta^{-1})_*\beta}$. Theta degrees are unchanged, so this identity proves semilinear equivariance of the Novikov series. An invariant specialization makes
the action on its coefficients trivial. Apply the theorem to each orbit; nontrivial character blocks annihilates $\Theta_T$.
\end{proof}

\begin{remark}\label{rem:not-source-splitting}
This is a Fourier decomposition of evaluation correspondences, not a geometrically disjoint decomposition of the stable-map space or of its virtual class. It does not determine the individual irreducible $L_{\beta,j}$ in \eqref{eq:Zbeta-Steinberg}, their multiplicities, or their Chow-theoretic support. The endoscopic interpretation below uses an
additional direct-image support theorem.
\end{remark}

\section{Endoscopic lines in rank two and genus two}\label{sec_rank2}

We calculate the $15$ nontrivial endoscopic base lines and the reduced discriminant. The resulting support statement concerns the induced cohomological character blocks, and it does not determine the full reduced Chow cycle. In this section $r=g=2$, and we identify the Hitchin coordinate with $a=-\det\phi$, so the spectral equation is $\lambda^2=a$.
For a correspondence block, $B$ denotes a finite $\Gamma$-stable set as in Proposition~\ref{prop:red-invariance}.

\subsection{Spectral construction}

Fix \(0\neq\gamma\in\Gamma\). Let
\begin{equation}\label{eq:etale-cover}
 \pi_\gamma:C_\gamma
 =\operatorname{Spec}_C(\cO_C\oplus\gamma^{-1}) \longrightarrow C
\end{equation}
be the corresponding connected \'etale double cover, and let \(\iota\) be its deck involution. Since \(g(C)=2\), Riemann--Hurwitz gives
\(g(C_\gamma)=3\).

The projection formula gives the $\gamma$-eigenspace decomposition
\begin{equation}\label{eq:differential-eigenspaces}
 H^0(C_\gamma,K_{C_\gamma})
 =H^0(C,K_C)\oplus H^0(C,K_C\otimes\gamma).
\end{equation}
The second summand is anti-invariant. Riemann--Roch gives
\begin{equation}\label{eq:h0-Kgamma}
 h^0(C,K_C\otimes\gamma)=1,
\end{equation}
because \(h^0(C,\gamma)=0\) for a nontrivial degree-zero line bundle.

\begin{proposition}
\label{prop:fixed-endoscopic-base}
Let \(X^\gamma\) be the fixed locus of tensorization by \(\gamma\). Its Hitchin image is contained in
\begin{equation}\label{eq:Agamma}
 \A_\gamma =\{s^2\, | \, s\in H^0(C,K_C\otimes\gamma)\}\subset \A.
\end{equation}
The closure of the image equals \(\A_\gamma\). In particular, \(\A_\gamma\) is a one-dimensional linear subspace of 
\(\A\).
\end{proposition}

\begin{proof}
A fixed stable Higgs bundle admits an isomorphism
\begin{equation}\label{eq:fixed-isomorphism}
 \psi:(E,\phi)\xrightarrow{\sim}(E\otimes\gamma,
 \phi\otimes\id_\gamma).
\end{equation}
Stability implies that an automorphism of \((E,\phi)\) is scalar. After rescaling \(\psi\), the square of \(\psi\) is the identity. The pair \((E,\psi)\) is therefore a module over
\(\cO_C\oplus\gamma^{-1}=(\pi_{\gamma})_*\cO_{C_\gamma}\). It has the form
\begin{equation}\label{eq:E-pushforward}
 E\cong(\pi_{\gamma})_* L
\end{equation}
for a line bundle \(L\) on \(C_\gamma\); this is the usual fixed-point description of Narasimhan--Ramanan \cite{NarasimhanRamanan}.

The Higgs-isomorphism condition in \eqref{eq:fixed-isomorphism} makes the Higgs field linear over the \'etale algebra in \eqref{eq:etale-cover}. After pullback to \(C_\gamma\), it is therefore diagonal. The trace-free condition says that the two eigenvalues sum to $0$, so the eigenvalue is anti-invariant. Thus
\begin{equation}\label{eq:diagonal-higgs}
 \pi_\gamma^*\phi =\begin{pmatrix}s&0\\0&-s\end{pmatrix},
 \qquad
 s\in H^0(C,K_C\otimes\gamma).
\end{equation}
Its characteristic polynomial is \(\lambda^2-s^2\). Hence \(h(E,\phi)=s^2\), proving the containment.

Conversely, let \(0\neq s\in H^0(C,K_C\otimes\gamma)\). The norm map \(\Nm:\Pic(C_\gamma)\to\Pic(C)\) is surjective on the required degree component: on \(\Pic^0\), its composite with pullback is multiplication by $2$, and multiplication by $2$ on an abelian variety is surjective. The determinant of a finite pushforward satisfies
\begin{equation}\label{eq:determinant-norm-formula}
 \det((\pi_{\gamma})_*L)=\Nm(L)\otimes\det((\pi_{\gamma})_*\cO_{C_\gamma}).
\end{equation}
We may therefore choose \(L\) with \(\det((\pi_{\gamma})_*L)=\Lambda\).

Let \(\Sigma_s\subset\operatorname{Tot}(K_C)\) be the spectral curve \(\lambda^2=s^2\). Its normalization is the connected \'etale cover \(C_\gamma\), and  the two local roots are exchanged by the nontrivial cocycle defining 
\(\gamma\). The zeros of $s$ are two distinct Weierstrass points, by Lemma~\ref{lem:two-torsion-labels} and \eqref{eq:Kgammaij} below. At each zero, the normalization identifies the two points above it, and $\Sigma_s$ has an ordinary node. In particular, \(\Sigma_s\) is integral, although it need not be smooth. The spectral sheaf associated with \(((\pi_{\gamma})_*L,\phi)\) is the rank-one torsion-free sheaf obtained from \(L\) by finite pushforward to 
\(\Sigma_s\).

The generic fiber of \(\Sigma_s\to C\) is \(\operatorname{Spec}K(C_\gamma)\), which is a field because the \'etale double cover is connected. A \(\phi\)-invariant line subbundle would give a one-dimensional invariant subspace over 
\(K(C)\), hence a \(K(C)\)-rational eigenvalue and a splitting of this quadratic field extension. This is impossible.  Equivalently, the BNR correspondence identifies such a subbundle with a spectral subsheaf whose generic support is a degree-one component of \(\Sigma_s\); see \cite[Section~3]{BNR}. The integral spectral curve has no such component. Hence the resulting Higgs bundle is stable. Thus every nonzero point of \(\A_\gamma\) occurs in the Hitchin image of 
\(X^\gamma\), and taking the closure gives equality. Equation \eqref{eq:h0-Kgamma} shows that \(\A_\gamma\) is a line.
\end{proof}

\subsection{Hyperelliptic coordinates}

Let
\begin{equation}\label{eq:hyperelliptic-map}
 \pi:C\longrightarrow\mathbb P^1
\end{equation}
be the hyperelliptic double cover. Let \(w_1,\ldots,w_6\) be its ramification points and \(b_i=\pi(w_i)\). Since
\begin{equation}\label{eq:canonical-hyperelliptic}
 K_C=\pi^*\cO_{\mathbb P^1}(1),
\end{equation}
we have
\begin{equation}\label{eq:A-binary-quadratics}
 \A=H^0(C,K_C^2) =H^0(\mathbb P^1,\cO_{\mathbb P^1}(2)).
\end{equation}
For each \(i\), set
\begin{equation}\label{eq:Hi}
 H_i=\{a\in \A \, | \, a(b_i)=0\}.
\end{equation}

\begin{lemma}
\label{lem:two-torsion-labels}
The line bundles
\begin{equation}\label{eq:gammaij}
 \gamma_{ij}=\cO_C(w_i-w_j),
 \qquad 1\leq i<j\leq6,
\end{equation}
are the fifteen nonzero elements of \(\Jac(C)[2]=(\zz/2\zz)^4\).
\end{lemma}

\begin{proof}
The divisor \(2w_i\) is a fiber of the hyperelliptic pencil.  Hence
\(2w_i\sim2w_j\), and \(\gamma_{ij}^2\cong\cO_C\).
The line bundle \(\gamma_{ij}\) is nontrivial, since distinct points of a
positive-genus curve are not linearly equivalent.

Suppose \(\gamma_{ij}=\gamma_{k\ell}\).  Then
\begin{equation}\label{eq:degree-two-equivalence}
 w_i+w_\ell\sim w_k+w_j.
\end{equation}
If the divisor on the left moves, it belongs to the unique hyperelliptic pencil. A fiber containing a Weierstrass point is twice that point, so \(i=\ell\). The divisor on the right then also belongs to the pencil, which forces \(k=j\). The two unordered pairs are the same. Otherwise the divisor on the left is rigid. Linear equivalence in \eqref{eq:degree-two-equivalence} then implies equality of the two effective divisors, and again the unordered pairs agree. Thus the $15$ classes in \eqref{eq:gammaij} are distinct. Since \(|\Jac(C)[2]|=2^{2g}=16\), they exhaust the nonzero elements.
\end{proof}

\begin{theorem}
\label{thm:explicit-endoscopic-lines}
For \(1\leq i<j\leq6\),
\begin{equation}\label{eq:Agamma-Hij}
 \A_{\gamma_{ij}}=H_i\cap H_j.
\end{equation}
For comparison, the ordinary squaring map on \(H^0(C,K_C)\) has image the quadratic cone
\begin{equation}\label{eq:square-cone}
 Q=\{\ell^2\, |\, \ell\in H^0(\mathbb P^1,\cO(1))\}\subset \A.
\end{equation}
\end{theorem}

\begin{proof}
Use \(K_C\cong\cO_C(2w_j)\). Then
\begin{equation}\label{eq:Kgammaij}
 K_C\otimes\gamma_{ij} \cong\cO_C(w_i+w_j).
\end{equation}
A nonzero section of this line bundle, unique up to scalar, has divisor \(w_i+w_j\). Its square has divisor 
\(2w_i+2w_j\), which is the pullback of \(b_i+b_j\). Under \eqref{eq:A-binary-quadratics}, the square therefore spans the line of binary quadratics vanishing at both \(b_i\) and \(b_j\). This line is \(H_i\cap H_j\).

The space \(H^0(C,K_C)\) is two-dimensional. Its squares form the affine cone over the Veronese conic in
\(\mathbb P H^0(C,K_C^2)\). This is \(Q\). This comparison does not identify \(Q\) with the support of the trivial character: the identity fixes all of \(X\), and the stable direct-image summand is supported on all of \(\A\).
\end{proof}

\begin{proposition}
\label{prop:discriminant-decomposition}
The reduced Hitchin discriminant is
\begin{equation}\label{eq:discriminant-decomposition}
 \Delta_{\red}=Q\cup H_1\cup\cdots\cup H_6.
\end{equation}
A general point of \(H_i\) gives an integral spectral curve with one node. A general point of \(Q\) gives two copies of \(C\) meeting transversely in two points.
\end{proposition}

\begin{proof}
A Hitchin coordinate is a binary quadratic \(a\in H^0(\mathbb P^1,\cO(2))\). Its pullback to \(C\) has a multiple
zero in exactly two cases. First, the two roots of \(a\) coincide; then \(a=\ell^2\) and \(a\in Q\). Second, one root is a branch point \(b_i\); then \(a\in H_i\). These conditions give \eqref{eq:discriminant-decomposition} set-theoretically, and each component is reduced at its generic point.

For the local types, write the hyperelliptic map near a ramification point as \(w=z^2\). At a general point of \(H_i\), the spectral equation is \(\lambda^2-z^2u=0\), with \(u\) a unit, and has one ordinary node. At a general point of 
\(Q\), write \(a=\ell^2\). Then \((\lambda-\ell)(\lambda+\ell)=0\), and the two components meet transversely at the two simple zeros of \(\ell\).
\end{proof}

\begin{remark}\label{rem:discriminant-versus-endoscopy}
The six divisors \(H_i\) are generic nodal discriminant components. They are not the nontrivial endoscopic bases. The latter are the $15$ codimension-two intersections \(H_i\cap H_j, 1\leq i< j\leq 6\).
\end{remark}

\subsection{The support theorem}

The principal polarization of \(\Jac(C)\) gives the nondegenerate Weil pairing
\begin{equation}\label{eq:Weil-pairing}
 e_2:\Gamma\times\Gamma\longrightarrow\{\pm1\}.
\end{equation}
It identifies \(\Gamma\) with its character group by
\begin{equation}\label{eq:Weil-character}
 \gamma\longmapsto\kappa_\gamma,
 \qquad
 \kappa_\gamma(\eta)=e_2(\gamma,\eta).
\end{equation}

Let
\begin{equation}\label{eq:direct-image-complex}
 \cK:=Rh_*\IC_X=Rh_*\qq_X[6]\in D^b_c(\A,\qq).
\end{equation}
Tensoring acts on \(\cK\). Since \(D^b_c(\A,\qq)\) is idempotent complete, the derived character projectors
\begin{equation}\label{eq:derived-projector}
 \mathbf P_\kappa =\frac1{16}\sum_{\eta\in\Gamma}\kappa(\eta)^{-1}\tau_\eta^*
\end{equation}
split \(\cK\):
\begin{equation}\label{eq:direct-image-decomposition}
 \cK=\bigoplus_{\gamma\in\Gamma}\cK_{\kappa_\gamma},
 \qquad
 \cK_{\kappa_\gamma}:=\operatorname{Im}(\mathbf P_{\kappa_\gamma}).
\end{equation}

We now state precisely the external support theorem used below. In the fixed-determinant notation of Maulik--Shen \cite{MS}, the character attached to \(\gamma\in\Jac(C)[2]\) in determinant degree \(\delta\) is
\begin{equation}\label{eq:degree-dependent-character}
 \kappa_{\gamma,\delta}(\eta)=e_2(\gamma,\eta)^\delta.
\end{equation}
Since $\sM^{\SL_2}(C,\Lambda) \xrightarrow{\,\sim\,} \sM^{\SL_2}(C,\Lambda\otimes M^2)$ after tensoring with line bundle $M$, \(\delta=1\); for the original determinant degree the same character occurs because that degree is odd.

\begin{theorem}\label{thm:endoscopic-support-external}
Let \(h:X\to \A\) be the trace-free, fixed-determinant rank-$2$ Hitchin fibration with \(\gcd(2,\deg\Lambda)=1\). For every \(0\neq\gamma\in\Jac(C)[2]\) and every integer \(j\), we have
\begin{equation}\label{eq:endoscopic-support-perverse}
 \Supp\bigl({}^p\!\mathcal H^j (\cK_{\kappa_{\gamma,\deg\Lambda}})\bigr)\subset \A_\gamma.
\end{equation}
Equivalently,
\begin{equation}\label{eq:endoscopic-support}
 \cK_{\kappa_{\gamma,\deg\Lambda}}|_{\A\setminus \A_\gamma}=0.
\end{equation}
The summand \(\cK_1\) is the stable, or invariant, summand. The nontrivial summands are the variant, or endoscopic, summands.
\end{theorem}

\begin{proof} 
This is the coprime fixed-determinant specialization of \cite[Theorem~0.2]{MS}. The support construction is given in
\cite[Section~3]{MS}. The degree-dependent character in that theorem is \eqref{eq:degree-dependent-character}. Its hypotheses apply because the rank is $2$ and the determinant degree is odd. The endoscopic base in their notation is the image of \(H^0(C,K_C\otimes\gamma)\) under the square map. It is therefore \(\A_\gamma\) in \eqref{eq:Agamma}.
The rank-$2$ stable--variant identification is also described in \cite[Sections~4.4--4.5]{DCHM}. Its original mirror-symmetry formulation is \cite{HT}. We use only the support inclusion \eqref{eq:endoscopic-support-perverse}. We do not use equality of supports or a compactly supported variant.
\end{proof}

Theorem \ref{thm:endoscopic-support-external} is nonequivariant. Its \(T\)-equivariant enhancement is obtained by Borel construction (\ref{eqn:Nth_Borel}). The underlying nonequivariant summand vanishes on \(\A\setminus \A_\gamma\). Its equivariant enhancement therefore vanishes over the quotient \([(\A\setminus \A_\gamma)/T]\); on every finite Borel approximation \(U\), its support is contained in \(\A_\gamma\times^T U\).

\subsection{The relative cohomological correspondence}

Set
\begin{equation}\label{eq:relative-correspondence-space}
 W=X\times_\A X,
\end{equation}
and let
\begin{equation}\label{eq:correspondence-projections}
 p_i:W\longrightarrow X, \qquad i=1,2,
\end{equation}
be the two projections. Each projection is projective, since it is a base change of the Hitchin morphism. We work on the canonical ambient correspondence space \(W\), rather than on a closed support chosen from a cycle representative. The Borel--Moore class \(\cl(Z_{B}^{\red})\) gives a degree-zero cohomological correspondence
\begin{equation}\label{eq:cohomological-correspondence-arrow}
 u_B:p_1^*\IC_X\longrightarrow p_2^!\IC_X
\end{equation}
in \(D^b_c(Y,\qq)\). The degree is zero because \(\dim_\C X=\dim_\C Z_{B}^{\red}=6\). This is the standard cycle-class
construction of cohomological correspondences \cite[Section~1.1]{Varshavsky}, and its Borel--Moore realization is the usual middle-dimensional correspondence action \cite[Section~2.7]{ChrissGinzburg}.

Put \(g=h\circ p_1=h\circ p_2\). Proper adjunction defines 
\begin{equation}\label{eq:six-functor-endomorphism}
\begin{aligned}
 \mathbf{F}_{Z^{\red}_{B}}:Rh_*\IC_X
 &\xrightarrow{\,\mathrm{unit}\,}
 Rg_*p_1^*\IC_X
 \xrightarrow{\, Rg_*(u_B)\,}
 Rg_*p_2^!\IC_X
 \xrightarrow{\,\mathrm{counit}\,}
 Rh_*\IC_X.
\end{aligned}
\end{equation}
This constructs the morphism in the derived category, and no passage from a Chow action to a sheaf morphism is assumed.

The cycle-class compatibility with the six operations gives, on hypercohomology,
\begin{equation}\label{eq:derived-action-hypercohomology}
 \alpha\longmapsto
 \operatorname{PD}_X^{-1}(p_2)_*\bigl(p_1^*\alpha\cap\cl^{\mathrm{BM}}(Z_{B}^{\red})\bigr).
\end{equation}
Thus \eqref{eq:six-functor-endomorphism} induces the operator $\mathbf{F}_{Z_{B}^{\red}}$. Diagonal \(\Gamma\)-invariance of the cycle class makes \(u_B\), and hence \(\mathbf{F}_{Z^{\red}_{B}}\), \(\Gamma\)-equivariant.
Therefore \(\mathbf{F}_{Z^{\red}_{B}}\) commutes with the derived idempotents \eqref{eq:derived-projector}. Applying the construction to finite Borel approximations gives the equivariant hypercohomology operator used in quantum multiplication.

For a character \(\kappa\), let \(\mathbf{F}_{Z^{\red}_{B},\kappa}\) be the derived endomorphism obtained by applying the same cycle-class construction to \(\cl(Z_{B,\kappa}^{\red})\). Linearity of cohomological correspondences and functoriality under translation of the first factor give
\begin{equation}\label{eq:derived-projected-block}
 \mathbf{F}_{Z^{\red}_{B},\kappa}
 =\mathbf P_\kappa\mathbf{F}_{Z^{\red}_{B}}
 =\mathbf{F}_{Z^{\red}_{B}}\mathbf P_\kappa
 =\mathbf P_\kappa\mathbf{F}_{Z^{\red}_{B}}\mathbf P_\kappa.
\end{equation}
Indeed, translation of the first factor precomposes the correspondence operator by the corresponding pullback, and equivariance of \(\mathbf{F}_{Z^{\red}_{B}}\) gives the remaining equalities. Thus the algebraically projected kernel induces exactly the derived projected block.

\begin{theorem}\label{thm:stable-endoscopic-blocks}
The endomorphism \(\mathbf{F}_{Z^{\red}_{B}}\) preserves \eqref{eq:direct-image-decomposition}. Its block on \(\cK_1\) is the stable block. For \(0\neq\gamma\in\Gamma\), its block
\begin{equation}\label{eq:endoscopic-block}
 \mathbf{F}_{Z^{\red}_{B},\gamma}
 :=\mathbf{F}_{Z^{\red}_{B},\kappa_\gamma}:\cK_{\kappa_\gamma}\longrightarrow\cK_{\kappa_\gamma}
\end{equation}
is induced by the cycle class of \(Z_{B,\kappa_\gamma}^{\red}\) and is supported on \(\A_\gamma\), in the sense that its restriction to \(\A\setminus \A_\gamma\) is $0$. In genus $2$,
\begin{equation}\label{eq:endoscopic-block-line}
 \Supp(\mathbf{F}_{Z^{\red}_{B},\gamma_{ij}}) \subset H_i\cap H_j.
\end{equation}
\end{theorem}

\begin{proof}
Proposition~\ref{prop:red-invariance} makes \(\mathbf{F}_{Z^{\red}_{B}}\) a \(\Gamma\)-equivariant endomorphism. It therefore commutes with every character projector and preserves each direct summand. For \(\gamma\neq0\), the source and target of the block \eqref{eq:endoscopic-block} restrict to $0$ on \(\A\setminus \A_\gamma\) by Theorem
\ref{thm:endoscopic-support-external}. The block has the same property. Equation \eqref{eq:endoscopic-block-line} follows from Theorem \ref{thm:explicit-endoscopic-lines}.
\end{proof}

\begin{corollary}\label{cor:generic-no-variant}
Let \(D^\circ\) be a dense open subset of one of \(Q,H_1,\ldots,H_6\), obtained by deleting all pairwise intersections and all worse spectral strata. Then
\begin{equation}\label{eq:variant-zero-generic-D}
 \cK_{\kappa_\gamma}|_{D^\circ}=0
\end{equation}
for every \(\gamma\neq0\). Hence every endoscopic block of \(\mathbf{F}_{Z^{\red}_{B}}\) restricts to $0$ on 
\(D^\circ\).
\end{corollary}

\begin{proof}
Every nontrivial endoscopic base is a line \(H_i\cap H_j\), and this line is deleted from each \(D^\circ\). Apply
Theorem~\ref{thm:endoscopic-support-external}.
\end{proof}

\begin{remark}\label{rem:Chow-support-warning}
Theorem~\ref{thm:stable-endoscopic-blocks} does not assert that
\begin{equation}\label{eq:unsupported-Chow-claim}
Z_{B,\kappa_\gamma}^{\red}
\in
\operatorname{Im}\left[A_6^T(X_{\A_\gamma}\times_{\A_\gamma}X_{\A_\gamma})\longrightarrow A_6^T(X\times_\A X)\right].
\end{equation}
The projected kernel may contain homologically trivial pieces away from \(\A_\gamma\). Establishing \eqref{eq:unsupported-Chow-claim} would require a separate Chow-theoretic support theorem. 
\end{remark}

\section{Equivariant Groechenig--Shen morphism}\label{sec_Groechenig-Shen}

We construct a noncanonical completed complex additive lift of a chosen nonequivariant Groechenig--Shen isomorphism.  We do not construct an equivariant equivalence of spectra or a global Fourier--Mukai kernel. The completion and coefficient conventions are part of the statement. 

\subsection{Set up, Equivariant twist and inertia components}
As before, let $\widehat{X}:=[X/\Gamma]$. Let $$\A^{\mathrm{ell}}\subset \A$$ be the locus of integral spectral curves and let $$\A^{\mathrm{sm}}\subset \A^{\mathrm{ell}}$$ be the locus of smooth
spectral curves. Superscripts $\mathrm{ell}$ and $\mathrm{sm}$ denote inverse images of these loci.

The torus $T=\C^*$ acts by scaling the Higgs fields: $t\cdot(E,\phi)=(E,t\phi)$, and it acts on the $i$th summand of
$\A$ with weight $i$. The $\Gamma$-action commutes with $T$ and preserves the holomorphic symplectic form on $X$. The equivariant twisted groups used below require a coherent equivariant gerbe, not merely a nonequivariant Brauer class.  We fix the following datum once and for all.

\begin{definition}
\label{def:equivariant-gerbe}
Fix a banded $\mu_r$-gerbe,
\[
 \mathcal G_T\longrightarrow \widehat X=[X/\Gamma],
\]
which represents the lifting class $\alpha$, together with a coherent lift of the Higgs-scaling $T$-action to $\mathcal{G}_T$. The $T$-action is required to commute with the quotient-groupoid structure. Write
\[
 \alpha_T\in\Br_T(\widehat X)[r]
\]
for the class of this fixed equivariant twist, $\alpha$ for its underlying nonequivariant twist, and $\widetilde\alpha_T$ for its topological realization. No canonicity among equivariant structures is asserted.
\end{definition}

\begin{remark}\label{rem:gerbe-construction}
Choose $c_0\in C$ and evaluate the universal projective bundle on the $\text{PGL}_r$ moduli stack at $c_0$. Its lifting gerbe for $1\to\mu_r\to\SL_r\to\PGL_r\to1$ is a banded $\mu_r$-gerbe. A coherent Higgs-scaling action on this evaluated projective object induces an equivariant lifting datum as in Definition~\ref{def:equivariant-gerbe}. We fix the datum and the identification of its underlying twist with the one used in the nonequivariant comparison as inputs. Lifts of a projective bundle \emph{on the whole curve} differ by $r$-torsion line bundles and form, when nonempty, a $\Gamma$-torsor of isomorphism classes. That fibered category must not be substituted for the single $\mu_r$-gerbe just described.
\end{remark}

The nonequivariant input is a chosen Groechenig--Shen~\cite{GS} equivalence of complex topological $K$-theory spectra
\begin{equation}\label{eq:GS-noneq}
 \Phi^{\KU}_{\GS}:\KU^*(X^{\mathrm{an}})
 \xrightarrow{\ \simeq\ }
 \KU^*(\widehat X^{\mathrm{an}},\widetilde\alpha).
\end{equation}
Only existence of \eqref{eq:GS-noneq} is used. In particular, we do not assume that it is induced by a global algebraic kernel on $X\times_\A\widehat X$ or that it has a $T$-equivariant refinement.

Let $$K=S^1\subset T$$
be the maximal compact subgroup. {\em Topological} $T$-equivariant groups below mean $K$-equivariant groups. Let $\chi$ denote the standard character and put
\[
 R_T=R(K)\otimes_\zz\cc=\cc[\chi,\chi^{-1}],\qquad
 I_T=(\chi-1),\qquad
 \widehat R_T=\limi_N R_T/I_T^N.
\]
Let $\hbar=c_1^T(\C_1)$ and let $u$ be the Bott variable. Both $\hbar$ and $u$ have cohomological degree $2$.

Consider the unrigidified inertia stack of $X$:
\[
 I\widehat X=\coprod_{\gamma\in\Gamma}[X^\gamma/\Gamma].
\]
The notation on the right is not a connected-component decomposition. A connected component $\mathcal F\subset I\widehat X$ can be represented by a pair $(\gamma,F)$, where $F$ is a connected component of $X^\gamma$. Put
\[
 \Gamma_F=\{\eta\in\Gamma\, |\, \eta(F)=F\}.
\]
Then $\mathcal F\simeq[F/\Gamma_F]$. Different pairs can represent the same connected quotient component, but the quantities introduced below are invariant under that equivalence. Let
$$\mathscr L_{\alpha_T,\mathcal F}$$ 
be the $T$-equivariant finite-monodromy transgression line of $\alpha_T$ on $\mathcal F$.

\begin{definition}[Fixed-total-degree periodic groups]
For an integer $k$, a finite-type $T$-space or finite quotient stack $Z$, and a $T$-equivariant local system $\mathscr L$, set
\begin{equation}\label{eq:periodic-degree}
 H^k_{T,\per}(Z;\mathscr L):=\bigoplus_{n\in\zz}H_T^{k-2n}(Z;\mathscr L)u^n.
\end{equation}
Multiplication by
\begin{equation}\label{eq:zdef}
 z:=\hbar u^{-1}
\end{equation}
preserves total degree. Define
\begin{equation}\label{eq:zcompletion}
 \widehat H^k_{T,\per}(Z;\mathscr L):=\limi_N H^k_{T,\per}(Z;\mathscr L)/(z^N).
\end{equation}
For the quotient stack $\widehat{X}$, set
\begin{align}
 H^k_{T,\del,\per}(\widehat X,\alpha_T)
 &:={\bigoplus_{\mathcal F\in\pi_0(I\widehat X)}}
 H^k_{T,\per}(\mathcal F;\mathscr L_{\alpha_T,\mathcal F}),
 \label{eq:delperiodic}\\
 \widehat H^k_{T,\del,\per}(\widehat X,\alpha_T)
 &:={H^k_{T,\del,\per}(\widehat X,\alpha_T)}^{\wedge}_{(z)}.
 \label{eq:completed-delperiodic}
\end{align}
There is no Chen--Ruan age shift in these definitions. Multiplication by $u$ identifies degree $k$ with degree $k+2$, so $k=0,1$ suffice.
\end{definition}

\begin{lemma}\label{lem:point-test}
One has
\[
 H^0_{T,\per}(\mathrm{pt})=\C[z],\qquad
 \widehat H^0_{T,\per}(\mathrm{pt})=\C[\![z]\!].
\]
If one instead keeps only parity and completes $\C[\hbar]\otimes\C[u,u^{-1}]$ at $(\hbar)$, the result is not finite
free over $\C[\![\hbar]\!]$.
\end{lemma}

\begin{proof}
In total degree zero, the only monomials are $\hbar^ju^{-j}=z^j$. If total degree is forgotten, every Laurent
monomial $u^n$ survives independently, giving infinitely many module generators over $\C[\![\hbar]\!]$.
\end{proof}

The periodic equivariant Chern character of a weight-$a$ line is
\begin{equation}\label{eq:weight-line-character}
 \ch_{T,u}(\C_a)=e^{a\hbar/u}=e^{az}.
\end{equation}
Therefore
\begin{equation}\label{eq:ringiso}
 \widehat R_T\xrightarrow{\ \sim\ }\C[\![z]\!],\qquad \chi-1\longmapsto e^z-1.
\end{equation}
This is an isomorphism of complete local rings because $e^z-1$ has linear term $z$.

\subsection{Semiprojectivity}\label{subsec_semi-proj-support}

\begin{lemma}\label{lem:semiprojective}
The $T$-varieties $X$ and $X^\gamma$, for every $\gamma\in\Gamma$, are smooth semiprojective varieties, and their $T$-fixed loci are projective and every point has a limit as $t\to0$.
\end{lemma}

\begin{proof}
All weights on $\mathbb{A}$ are positive, so $\mathbb{A}^T=\{0\}$. Hence $X^T\subset h^{-1}(0)$, which is projective because $h$ is proper. The fixed locus of a finite-order automorphism of a smooth complex variety is smooth. Thus $X^\gamma$ is smooth, and $(X^\gamma)^T=X^\gamma\cap X^T$ is projective.

For $x\in X$, the Hitchin image of the orbit extends to $t=0$ by
\[
 t\longmapsto(t^2a_2,\ldots,t^ra_r).
\]
The valuative criterion for the proper map $h$ uniquely extends the corresponding $\Spec\C(\!(t)\!)$-point of $X$ to $\Spec\C[\![t]\!]$. Its special point is $\lim_{t\to0}t\cdot x$.  If $x\in X^\gamma$, uniqueness and commutation of the two actions show that the limit remains in $X^\gamma$.
\end{proof}

We use the semiprojective, filtrable form of the Bia\l ynicki--Birula theorem: a smooth semiprojective variety has
finitely many projective fixed components and a finite filtration by closed $T$-stable subvarieties whose successive strata are equivariant affine bundles over those fixed components. This is the standard semiprojective extension of the projective theorem; see \cite{Hesselink} and \cite{HauselCompactification}.

\begin{lemma}
\label{lem:stabilizer-trivial}
Let a connected complex algebraic group $G$ act trivially on a connected space $F$. If $\mathscr L$ is a $G$-equivariant local system on $F$, the stabilizer action on every fiber is trivial. Hence
\[
 H_G^\bullet(F;\mathscr L) \cong H^\bullet(F;\mathscr L)\otimes H^\bullet(BG;\C).
\]
\end{lemma}

\begin{proof}
Restrict the equivariance isomorphism to $G\times\{x\}$. It is an automorphism of a constant local system and is therefore locally constant in $G$. Connectedness makes it constant, and the cocycle condition at the identity forces it to be the identity. The Borel construction then has the exterior-product coefficient system, and the K\"unneth theorem applies.
\end{proof}

\begin{theorem}[Equivariant formality with finite-monodromy coefficients]
\label{thm:equivariant-formality}
Let $Z$ be a smooth semiprojective finite-type complex $T$-variety and let $\mathscr L$ be a finite-rank, finite-monodromy, $T$-equivariant local system. Then:
\begin{enumerate}[label=\textup{(\roman*)}]
\item $H_T^\bullet(Z;\mathscr L)$ is finite free over $\C[\hbar]$;

\item the forgetful map induces
\[
 H_T^\bullet(Z;\mathscr L)/(\hbar) \xrightarrow{\ \sim\ }H^\bullet(Z;\mathscr L);
\]

\item for every $k$, $H^k_{T,\per}(Z;\mathscr L)$ is finite free over $\C[z]$, reduction modulo $z$ is ordinary fixed-degree periodized cohomology, and its completion is finite free over $\C[\![z]\!]$.
\end{enumerate}
The same conclusions hold after taking invariants under a finite group action commuting with $T$ and preserving $\mathscr L$.
\end{theorem}

\begin{proof}
A finite-monodromy algebraic local system is a direct summand of the pushforward of the constant system from a finite \'etale cover. For an equivariant local system the cover can be chosen $T$-equivariant: use the finite monodromy cover of the Borel local system and pull it back to $Z$. The connected stabilizers act trivially on its finite fibers. The cover is smooth and semiprojective. Indeed it is finite over $Z$, its fixed locus is projective, and the lifted algebraic orbit extends across zero by properness. It suffices to prove the constant coefficient assertion on this cover and then take the appropriate finite-group representation summand.

The filtrable Bia\l ynicki--Birula decomposition has strata which are affine bundles of rank $a_i$ over smooth projective fixed components $F_i$. Their compactly supported cohomology is $H_c^j(\text{stratum}_i)=H^{j-2a_i}(F_i)(-a_i)$ and is pure of weight $j$. Induction on the closed filtration, using its compact-support long exact sequence, proves purity of $H_c^j(Z)$: the boundary maps between degrees $j$ and $j+1$ vanish because their weights differ. Poincar\'e duality then gives purity of $H^j(Z)$ of weight $j$.

The Borel spectral sequence
\[
 E_2^{p,q}=H^p(BT)\otimes H^q(Z)\Longrightarrow H_T^{p+q}(Z)
\]
is a spectral sequence of mixed Hodge structures, as seen on algebraic finite Borel approximations. Each $E_2^{p,q}$ is pure of weight $p+q$; every differential changes that weight by one. It therefore degenerates. The edge map to $H^*(Z)$ is surjective. Homogeneous lifts of a finite basis, together with the spectral-sequence filtration, give a free
$\C[\hbar]$-basis and the asserted reduction modulo $\hbar$. The same argument on the finite cover followed by the finite-monodromy summand proves (i)--(ii) for $\mathscr L$.

If $a_j$ are homogeneous lifts, then in fixed degree $k$ the elements $a_j u^{(k-\deg a_j)/2}$ with $\deg a_j\equiv k\pmod2$ form a $\C[z]$-basis. This proves (iii). For an additional finite group, Reynolds averaging is an exact projector preserving all filtrations; its image is a finite projective, hence free, module over $\C[\hbar]$ and over $\C[z]$. Invariants commute with reduction and completion.
\end{proof}

\begin{corollary}\label{cor:fixed-degree-description}
For a connected inertia component $\mathcal F=[F/\Gamma_F]$, put
\[
 V^k_{\mathcal F}:=
 \bigoplus_{p\equiv k\, (2)} H^p(\mathcal F;\mathscr L_{\alpha,\mathcal F})u^{(k-p)/2}.
\]
Then $\widehat H^k_{T,\per}(\mathcal F;\mathscr L_{\alpha_T,\mathcal F})$ is finite free over $\C[\![z]\!]$, its reduction is $V^k_{\mathcal F}$, and, after choosing homogeneous equivariant lifts, there is an isomorphism
\begin{equation}\label{eq:fixed-degree-expansion}
 \widehat H^k_{T,\per}(\mathcal F;\mathscr L_{\alpha_T,\mathcal F})
 \cong V^k_{\mathcal F}\widehat\otimes_\C\C[\![z]\!].
\end{equation}
The reduction and associated graded are canonical; the displayed splitting need not be canonical.
\end{corollary}

\begin{proof}
Apply Theorem~\ref{thm:equivariant-formality} to $F$ and then take $\Gamma_F$-invariants. The basis construction in its proof gives \eqref{eq:fixed-degree-expansion}.
\end{proof}

\begin{lemma}\label{lem:finite-KCW}
The analytic spaces $X$ and $X^\gamma$, with their actions of $K$ and of any commuting subgroup of $\Gamma$, have finite equivariant CW  homotopy type. Finite equivariant covers have the same property.
\end{lemma}

\begin{proof}
In the differential-geometric description of Higgs moduli, the squared norm of the Higgs field is a proper moment map for the circle action \cite{HitchinSD,HauselCompactification}. The tensor action of the finite group $\Gamma$ is isometric and preserves this function. Restriction to a fixed locus $X^\gamma$ is again a proper Hamiltonian circle moment map; the fixed locus is a smooth K\"ahler submanifold. Its critical set is its circle-fixed locus. There are finitely many compact critical components, and the Hamiltonian circle local normal form makes the function Morse--Bott, with finite-rank equivariant negative normal bundles.
Equivariant Morse theory therefore gives finitely many disk-bundle attachments over compact critical manifolds. Those manifolds, carrying only the residual finite-group action, have finite equivariant triangulations \cite{IllmanTriangulation}. The attachments give finite $K$-equivariant CW homotopy type, simultaneously with the commuting finite group. A finite equivariant cover carries the pullback K\"ahler form and proper moment map. Its critical manifolds are finite covers of the preceding compact ones, so the same argument applies.
\end{proof}

\subsection{Circle-relative completion and the Chern character}
Let $Y=[V/G]$, where $G$ is finite and $V$ has finite $(K\times G)$-CW homotopy type. Let $\tau_T$ be an ungraded finite gerbe twist. All $K$-groups in this subsection are complexified \emph{before} completion. The augmentation ideal is only $I_T\subset R_K:=\C[\chi,\chi^{-1}]$; we do not complete the finite-group representation ring at its augmentation ideal.

\begin{lemma}\label{lem:trace-block-finite-generation}
For $k=0,1$, $M^k=\KU_{K\times G}^k(V;\tau_T)\otimes\C$ is finite over $R_K$. The same holds for finite relative equivariant CW pairs.
\end{lemma}

\begin{proof}
Use induction on equivariant cells. A cell has stabilizer a closed subgroup $H\subset K\times G$. If $H$ is finite, its (twisted) representation group has finite complex dimension. Otherwise $H^0=K$, and $H/K$ is finite. The corresponding finite central
extension for the twist has only finitely many weight classes modulo tensoring with $\chi$; within each class a finite group has finitely many irreducible representations. Thus its twisted representation module is finite over $R_K$. The six-term exact sequences for cell attachment and noetherianess of $R_K$ complete the induction. The same argument applies to relative pairs.
\end{proof}

Put $t=1-\chi^{-1}$; it generates $I_T$. Consider the sphere $E_N=S(\C^{N+1})$ with scalar $K$-action and $Y_N=[(E_N\times_KV)/G]$, carrying the descended twist $\tau_N$.

\begin{theorem}\label{thm:relative-completion}
For $k=0,1$ there is a natural isomorphism
\begin{equation}\label{eq:relative-completion}
 (\KU_T^k(Y,\tau_T)\otimes\C)^{\wedge}_{I_T}
   \xrightarrow{\sim}
 \varprojlim_N\bigl(\KU^k(Y_N,\tau_N)\otimes\C\bigr).
\end{equation}
The statement also holds for finite relative equivariant CW pairs. The finite-stage complexified tower is Mittag--Leffler, and its $\varprojlim^1$ vanishes. No completion in the finite-$G$ direction occurs.
\end{theorem}
\begin{proof}
Apply the equivariant Thom exact sequence to the representation disk $D(\C^{N+1})$ and sphere $E_N$, crossed with $V$ and with the same pulled-back twist. The Euler class is $t^{N+1}$, so it yields
\begin{equation}\label{eq:sphere-completion-sequence}
 0\longrightarrow M^k/t^{N+1}M^k
 \longrightarrow \KU^k(Y_N,\tau_N)\otimes\C
 \longrightarrow M^{k+1}[t^{N+1}]\longrightarrow0.
\end{equation}
Here the bracket denotes the subgroup killed by $t^{N+1}$, and $k+1$ is read modulo two. Naturality of the Thom class for $\C^{N+1}\subset\C^{N+2}$ identifies transition maps on the right with multiplication by $t$. On a finite $R_K$-module the $t$-primary
torsion is killed by a fixed power of $t$. The right-hand tower is therefore pro-zero.  The middle tower is pro-isomorphic to the left adic quotient tower, whose maps are surjective. Taking the inverse limit proves the claim and the Mittag--Leffler assertion. The disk, sphere and Thom arguments apply equally to a relative pair.
\end{proof}

\begin{remark}\label{rem:compensated-isotropy}
Let $V=K$ with $K$ acting by translation and $G=\mu_2$ acting by multiplication by $-1$. The nontrivial element $g\in G$ has no fixed point on $V$. Nevertheless $K\times G$ has stabilizer generated by $(-1,g)$, whose nontrivial trace block in equivariant $K$-theory is nonzero. In that block $\chi=-1$, so $\chi-1$ is invertible and
circle-relative completion annihilates the block. Thus literal $g$-fixed loci cannot in general describe uncompleted equivariant trace blocks. We do not posit an integral invertible $\KU$-module local system whose monodromy is an arbitrary complex character: a root of unity need not be an integral unit in $\pi_0\KU=\Z$.
\end{remark}

For a connected component $\mathcal F\subset IY$, let $\mathscr L_{\tau,\mathcal F}$ be the complex transgression line and put
\begin{equation}\label{eq:general-del-periodic}
 H^k_{T,\del,\per}(Y,\tau_T)
 =\bigoplus_{\mathcal F\subset IY}\ \bigoplus_{n\in\Z}
 H_T^{k-2n}(\mathcal F;\mathscr L_{\tau,\mathcal F})u^n.
\end{equation}
Its hat version $\widehat H^k_{T,\del,\per}(Y,\tau_T)$ means $z$-adic completion, with $z=\hbar/u$.

\begin{theorem}[Completed twisted equivariant Chern character]
\label{thm:general-completed-ch}
Under these hypotheses, the finite-stage complex twisted delocalized Chern characters induce a natural isomorphism
\begin{equation}\label{eq:general-completed-ch}
 \widetilde\ch^{\wedge}_{T,\tau,u}:
 (\KU_T^k(Y,\tau_T)\otimes\C)^{\wedge}_{I_T}
 \xrightarrow{\sim}\widehat H^k_{T,\del,\per}(Y,\tau_T),
 \qquad k=0,1.
\end{equation}
The target uses ordinary fixed-locus degrees on unrigidified inertia, not Chen--Ruan age shifts. The coefficient-ring map is $\chi\mapsto e^z$.
\end{theorem}
\begin{proof}
At each finite free Borel approximation, only the finite group $G$ remains. The complex twisted delocalized Chern character \cite{FHT} gives
\[
 \KU^k(Y_N,\tau_N)\otimes\C \xrightarrow{\sim}H^k_{\del,\per}(Y_N,\tau_N),
\]
naturally in $N$. Here $(E_N\times_KV)^g=E_N\times_KV^g$, since $K$ acts freely on $E_N$. Theorem~\ref{thm:relative-completion} identifies the inverse limit on the left.
On each fixed sector, the equivariant sphere Gysin sequence has Euler class $\hbar^{N+1}$. In fixed total degree, after multiplication by the compensating power  of Bott variable, its Euler class is $z^{N+1}$. The cohomological version of \eqref{eq:sphere-completion-sequence} therefore has a quotient term by $z^{N+1}$ and a torsion term in the opposite parity, with transition multiplication by $z$. Finite equivariant CW type makes these periodized modules finite over $\C[z]$, by the same cell induction as in Lemma \ref{lem:trace-block-finite-generation}. The torsion tower is pro-zero. The limit on the right is consequently the $z$-adic completion. Finite-group invariants and the finite inertia sum commute with this construction. The point computation \eqref{eq:weight-line-character} identifies the coefficient map.
\end{proof}

\begin{corollary}[The two Higgs state spaces]\label{thm:completed-ch}
For $k=0,1$, Theorem~\ref{thm:general-completed-ch} gives natural isomorphisms
\begin{align}
 \ch_{T,u}^{\wedge}:
 \left(\KU_T^k(X)\otimes\C\right)^{\wedge}_{I_T}
 &\xrightarrow{\ \sim\ }
 \widehat H^k_{T,\per}(X),
 \label{eq:completed-source-ch}\\
 \widetilde\ch_{T,\alpha,u}^{\wedge}:
 \left(\KU_T^k(\widehat X,\widetilde\alpha_T)\otimes\C\right)^{\wedge}_{I_T}
 &\xrightarrow{\ \sim\ }
 \widehat H^k_{T,\del,\per}(\widehat X,\alpha_T).
 \label{eq:completed-target-ch}
\end{align}
By Theorem~\ref{thm:equivariant-formality}, both targets are finite free
$\C[\![z]\!]$-modules; reduction modulo $z$ is the ordinary and twisted
delocalized nonequivariant periodic Chern character.
\end{corollary}
\begin{proof}
For $X$, take $G=1$. For $\widehat X=[X/\Gamma]$, take $G=\Gamma$. 
Lemma~\ref{lem:finite-KCW} verifies the finite $K$-CW hypothesis, and
Theorem~\ref{thm:equivariant-formality} together with
Corollary~\ref{cor:fixed-degree-description} give the finite freeness and the reduction statements.
\end{proof}

\subsection{Completed cohomology with support and Groechenig-Shen morphism}\label{subsec_GS_morphism}
Let $W$ be a finite quotient $T$-stack, let $Z\subset W$ be closed and $T$-invariant, and let $\mathscr{M}$ be a finite monodromy coefficient system. Define
\begin{align}
 H^k_{T,Z,\per}(W;\mathscr M)
 &:=\bigoplus_{n\in\Z}H^{k-2n}_{T,Z}(W;\mathscr M)u^n,
 \label{eq:supported-periodic}\\
 \widehat H^k_{T,Z,\per}(W;\mathscr M)
 &:=\limi_q H^k_{T,Z,\per}(W;\mathscr M)/(z^q).
 \label{eq:supported-completion-def}
\end{align}
For a twist $\tau_T$, the notation $\widehat H^k_{T,Z,\del,\per}(W,\tau_T)$ means the finite direct sum of \eqref{eq:supported-completion-def} over connected components of
unrigidified inertia, with support $IZ=I W\times_W Z$ and the corresponding transgression lines.

\begin{theorem}[Supported relative completion and Chern character]
\label{thm:supported-completion}
Assume that $(W,W\setminus Z)$ is represented by a finite $(K\times G)$-CW pair and that the twist and coefficient systems are finite. Put $W_N=E_N\times_KW$ and $Z_N=E_N\times_KZ$. Then, for $k\in\zz/2$, there are natural isomorphisms
\begin{equation}\label{eq:supported-completion}
 \left(\KU^k_{T,Z}(W,\tau_T)\otimes\C\right)^{\wedge}_{I_T}
 \xrightarrow{\ \sim\ }
 \varprojlim_N\bigl(\KU^k_{Z_N}(W_N,\tau_N)\otimes\C\bigr)
 \xrightarrow{\ \sim\ }
 \widehat H^k_{T,Z,\del,\per}(W,\tau_T).
\end{equation}
The two inverse systems are Mittag--Leffler; no $\varprojlim^1$ term occurs. The identifications are natural for morphisms of relative pairs, pullback, products, and excision.  
\end{theorem}

\begin{proof}
Use the relative-pair form of Theorem~\ref{thm:relative-completion} for $(W,W\setminus Z)$. Its proof is the sphere Thom sequence \eqref{eq:sphere-completion-sequence}, with the relative twisted groups in place of $M^k$; the same finite-generation and pro-zero torsion arguments apply. At finite Borel approximations, the complex delocalized
Chern character is an isomorphism on relative pairs by the absolute isomorphism and the long exact sequence. On each inertia component, the relative sphere Gysin sequence similarly identifies the inverse limit with $z$-adic completion of cohomology with support. These constructions are natural for maps of pairs, products and excision, and both towers are pro-isomorphic to their surjective adic quotient towers. Proper pushforward is treated separately by Riemann--Roch; no unsupported integral trace-coefficient construction is involved.
\end{proof}

Let
\[
 \Phi^H_{\GS,k}:H^k_{\per}(X) \xrightarrow{\ \sim\ }
 H^k_{\del,\per}(\widehat X,\alpha)
\]
be the complexified cohomological realization of the chosen nonequivariant spectrum equivalence \eqref{eq:GS-noneq}.

\begin{theorem}[Noncanonical completed additive lift]\label{thm:additive-lift}
For $k=0,1$, there exists a noncanonical $\C[\![z]\!]$-linear isomorphism
\begin{equation}\label{eq:additive-lift}
 \widehat\Phi^{H,\mathrm{add}}_{T,k}:
 \widehat H^k_{T,\per}(X)
 \xrightarrow{\ \sim\ }
 \widehat H^k_{T,\del,\per}(\widehat X,\alpha_T)
\end{equation}
whose reduction modulo $z$ is $\Phi^H_{\GS,k}$. Via Theorem~\ref{thm:completed-ch}, this is equivalent to a noncanonical isomorphism of the completed complexified equivariant $K$-groups as $\widehat R_T$-modules.
\end{theorem}
\begin{proof}
By Theorem~\ref{thm:completed-ch} and Corollary~\ref{cor:fixed-degree-description}, the source and target of \eqref{eq:additive-lift} are finite free $\C[\![z]\!]$-modules.  Their reductions modulo $z$ are the source and target of $\Phi^H_{\GS,k}$, so they have the same rank. Choose bases lifting bases of the reductions and let $B\in\mathrm{GL}_N(\C)$ be the matrix of $\Phi^H_{\GS,k}$. The same constant matrix over $\C[\![z]\!]$ is invertible and gives \eqref{eq:additive-lift}.

If $\Phi_1$ and $\Phi_2$ are two lifts, then $U=\Phi_2\Phi_1^{-1}$ is an automorphism congruent to the identity modulo $z$; conversely every such $U$ produces another lift.  This proves noncanonicity and shows why the construction carries none of the
additional structures excluded in the statement.
\end{proof}

\begin{corollary}\label{cor:genuine-equivariant}
Suppose, as an independent input, that there is a $T$-equivariant equivalence of twisted $\KU$-spectra
\[
 \Phi^{\KU,T}:\KU_T(X)\xrightarrow{\ \simeq\ }
 \KU_T(\widehat X,\widetilde\alpha_T)
\]
whose forgetful map is \eqref{eq:GS-noneq}. Then completion and Theorem~\ref{thm:completed-ch} canonically produce completed additive maps \eqref{eq:additive-lift}. If the spectrum equivalence is induced by a global supported $T$-equivariant twisted kernel, its completed cohomological realization is the full inertial Riemann--Roch transform of that kernel, not an arbitrary matrix lift.
\end{corollary}

\begin{proof}
Complete the assumed spectrum equivalence at $I_T$ and conjugate its maps on homotopy groups by the natural isomorphisms in Theorem~\ref{thm:completed-ch}. A kernel-induced assertion additionally uses Theorem~\ref{thm:bivariant-RR} below.
\end{proof}

\subsection{Ages, monodromy, and inertia conventions}

\subsubsection{Age and Poincar\'e duality are componentwise}

Fix a connected component $\mathcal F=[F/\Gamma_F]\subset I\widehat X$ represented by $(\gamma,F)$. Define
\[
 a_{\mathcal F}:=\age(\gamma\mid T_xX),\qquad
 d_{\mathcal F}:=\dim_\C F,
\]
for $x\in F$. $a_{\mathcal{F}}$ is locally constant and therefore constant on $F$. 

\begin{lemma}\label{lem:symplectic-age}
For every connected inertia component $\mathcal F$,
\begin{equation}\label{eq:component-dimension}
 d_{\mathcal F}=m-2a_{\mathcal F},\qquad
 a_{\mathcal F}\in\Z_{\ge0}.
\end{equation}
Inversion sends $\mathcal F$ to a component $\iota(\mathcal F)$ with
$d_{\iota(\mathcal F)}=d_{\mathcal F}$ and $a_{\iota(\mathcal F)}=a_{\mathcal F}$.
\end{lemma}

\begin{proof}
The $\Gamma$-action preserves the holomorphic symplectic form. On the moving tangent space, an eigenvalue $e^{2\pi i\theta}$ with $0<\theta<1$ is paired with
$e^{-2\pi i\theta}=e^{2\pi i(1-\theta)}$. Each paired two-dimensional summand contributes $1$ to the age and $2$ to the complex codimension of $F$. Summing proves \eqref{eq:component-dimension} and integrality. Replacing $\gamma$ by $\gamma^{-1}$ exchanges the paired eigenvalues and leaves both the fixed locus and the total contribution unchanged.
\end{proof}

Transgression of a gerbe has a canonical inversion duality
\begin{equation}\label{eq:inversion-line}
 \iota^*\mathscr L_{\alpha_T,\iota(\mathcal F)}
 \xrightarrow{\ \sim\ }
 \mathscr L_{\alpha_T,\mathcal F}^{\vee}.
\end{equation}
The correct sectorwise equivariant Poincar\'e duality is
\begin{equation}\label{eq:component-PD}
 \PD_{\mathcal F}:
 H_T^p(\mathcal F;\mathscr L_{\alpha_T,\mathcal F})
 \xrightarrow{\ \sim\ }
 H^{T,\mathrm{BM}}_{2d_{\mathcal F}-p}
 (\mathcal F;\mathscr L_{\alpha_T,\mathcal F}).
\end{equation}
Together with \eqref{eq:inversion-line}, this gives the unrigidified orbifold pairing between $\mathcal F$ and $\iota(\mathcal F)$.

The reduced virtual dimension is computed componentwise. For a three-pointed genus-zero twisted stable map with evaluations in $\mathcal F_1,\mathcal F_2,\mathcal F_3$, the
standard dimension is
\[
 m-\sum_{i=1}^3a_{\mathcal F_i}.
\]
Removing the $1$-dimensional symplectic obstruction quotient raises, rather than lowers, virtual dimension. Thus
\begin{equation}\label{eq:correct-three-point-dimension}
 \operatorname{vdim}^{\red}=m+1-\sum_{i=1}^3a_{\mathcal F_i}.
\end{equation}
After the divisor equation, the two-point reduced virtual dimension is
\begin{equation}\label{eq:two-point-dimension}
 D_{12}=m-a_{\mathcal F_1}-a_{\mathcal F_2}
       =\frac{d_{\mathcal F_1}+d_{\mathcal F_2}}2.
\end{equation}
A coefficient-valued correspondence of this dimension sends ordinary degree $p$ on $\mathcal F_1$ to
\begin{equation}\label{eq:ordinary-degree-shift}
 p+2(d_{\mathcal F_2}-D_{12})
 =p+2(a_{\mathcal F_1}-a_{\mathcal F_2})
\end{equation}
on $\mathcal F_2$. Hence it preserves Chen--Ruan degree.

\subsubsection{The two directions of the age regrading}

Define the Chen--Ruan-periodized group by
\begin{equation}\label{eq:CR-periodic}
 H^k_{T,\CR,\per}(\widehat X,\alpha_T)
 :=\bigoplus_{\mathcal F}\bigoplus_{n\in\Z}
 H_T^{k-2n-2a_{\mathcal F}}(\mathcal F;\mathscr L_{\alpha_T,\mathcal F})u^n.
\end{equation}
The degree-preserving map from the Chen--Ruan grading to the unshifted delocalized grading is
\begin{equation}\label{eq:age-map-forward}
 \St_{\mathrm{age}}: H^k_{T,\CR,\per}\longrightarrow H^k_{T,\del,\per},
 \qquad
 \St_{\mathrm{age}}(xu^n)=xu^{n+a_{\mathcal F}}.
\end{equation}
Its inverse, which is the map needed to convert the unshifted Chern-character target into a Chen--Ruan-periodized target, is
\begin{equation}\label{eq:age-map-inverse}
 \St_{\mathrm{age}}^{-1}(xu^n)=xu^{n-a_{\mathcal F}}.
\end{equation}
Thus a horizontal map from $X$ into the Chen--Ruan state space would be
$\St_{\mathrm{age}}^{-1}\circ\Phi^{\del}$, not $\St_{\mathrm{age}}\circ\Phi^{\del}$.  These are additive regradings; they do not prove compatibility with the Chen--Ruan product.

\subsubsection{Two-point monodromy}

\begin{lemma}[Genus-zero two-point monodromy]
\label{lem:two-point-monodromy}
Let a connected genus-zero twisted nodal curve carry a principal $\Gamma$-bundle and two marked monodromies $\gamma_1,\gamma_2$. Then
\begin{equation}\label{eq:monodromy-product}
 \gamma_1\gamma_2=1.
\end{equation}
Consequently the two-evaluation cycle has labels $(\gamma,\gamma^{-1})$. After the output involution in Definition~\ref{defn_Steinberg_PGL}, the operator has label
$\gamma\to\gamma$. If both marked monodromies are trivial, all node monodromies are trivial and the bundle is trivial on each geometric source curve. Over a family it can be a nontrivial pullback of a $\Gamma$-torsor on the base.
\end{lemma}

\begin{proof}
On the normalization of each irreducible component, the product of the monodromies at marked points and node branches is $1$. Multiply these relations over all vertices of the tree dual graph. The two balanced monodromies at each node are inverse and cancel, giving \eqref{eq:monodromy-product}.

If the marked monodromies are both $1$, take a leaf of the dual tree. The relation on that component forces its unique node monodromy to be $1$. Remove the leaf and iterate. The bundle is then trivial on every coarse $\mathbb P^1$ because
$H^1_{\mathrm{\acute et}}(\mathbb P^1,\Gamma)=0$, and the gluing class is trivial because the dual graph is a tree.
\end{proof}

A correspondence can connect distinct connected components of the same fixed locus. Without equidimensionality of that fixed locus one must retain the componentwise shift \eqref{eq:ordinary-degree-shift}. Inversion itself preserves the age of each component. Independently of any degree-preservation result, the Chern-character target is
unshifted, and the displayed age conversion is necessary before comparing it with a Chen--Ruan state space.

\subsubsection{Inner-local-system data and rigidification}

The family of sectorwise lines alone is not the full coefficient theory. Transgression of the gerbe $\alpha_T$ supplies, on unrigidified inertia:
\begin{enumerate}[label=\textup{(\roman*)}]
\item the inversion duality \eqref{eq:inversion-line};

\item over double inertia, a fusion isomorphism
\begin{equation}\label{eq:fusion}
 \mu^*\mathscr L_{\alpha_T}
 \xrightarrow{\ \sim\ } e_1^*\mathscr L_{\alpha_T}\otimes e_2^*\mathscr L_{\alpha_T};
\end{equation}

\item the associativity coherence of \eqref{eq:fusion} over triple inertia and the unit trivialization on the identity sector.
\end{enumerate}
These are the inner-local-system data used to contract coefficient lines in products and convolutions. They follow directly from composition of automorphisms in the gerbe and the associativity of that composition; see e.g. \cite{LupercioUribe}.

Standard orbifold Gromov--Witten evaluation maps, however, take values in the rigidified cyclotomic inertia $\overline{I}\widehat X$, see \cite{AGV}. If
$\mathcal F=[F/\Gamma_F]$ is represented by $(\gamma,F)$ and 
$s_{\mathcal F}=\ord(\gamma)$, its rigidification is
\[
 \overline{\mathcal F} \simeq[F/(\Gamma_F/\langle\gamma\rangle)],
\]
and
\begin{equation}\label{eq:rigidification-map}
 \rho_{\mathcal F}:\mathcal F\longrightarrow\overline{\mathcal F}
\end{equation}
is a $B\langle\gamma\rangle$-gerbe.

\begin{proposition}\label{prop:canonical-rigidification-descent}
Let $\alpha_T$ be a banded, ungraded finite gerbe. On every connected unrigidified sector $\mathcal F$ represented by $(\gamma,F)$, the cyclic subgroup $\langle\gamma\rangle$ acts trivially on the transgression line
$\mathscr L_{\alpha_T,\mathcal F}$. Consequently, there is a canonical descended line
$\overline{\mathscr L}_{\alpha_T,\overline{\mathcal F}}$ and a canonical isomorphism
\begin{equation}\label{eq:line-descent}
 \rho_{\mathcal F}^*\overline{\mathscr L}_{\alpha_T,\overline{\mathcal F}}
 \xrightarrow{\ \sim\ }
 \mathscr L_{\alpha_T,\mathcal F}.
\end{equation}
They are compatible with inversion and the unit; fusion and associativity are compared after pullback to the common unrigidified multiple-inertia correspondences. 
\end{proposition}

\begin{proof}
Fix a geometric point $x\in F$ and an object of the gerbe over $x$. Its isotropy is a central extension
\begin{equation}\label{eq:sector-central-extension}
 1\longrightarrow\mu_r\longrightarrow
 \widetilde{\Gamma_x}\longrightarrow\Gamma_x\longrightarrow1.
\end{equation}
The fiber of the transgression line at $(x,\gamma)$ is the weight-one line associated to the $\mu_r$-torsor of lifts of $\gamma$. If $\widetilde\gamma$ is one lift, the centralizer acts on this torsor by conjugation. For every $\zeta\in\mu_r$,
\[
 \widetilde\gamma(\zeta\widetilde\gamma)
 \widetilde\gamma^{-1}=\zeta\widetilde\gamma,
\]
so $\gamma$ acts trivially on the associated line. The same calculation with $\widetilde\gamma^j$ shows that every element of $\langle\gamma\rangle$ acts trivially.  Because the extension is central, this conclusion is independent of the chosen object and lift. It is functorial under restriction in $x$, hence glues over the entire sector.
Lemma~\ref{lem:rigidification-descent} now gives the canonical descent.

Inversion sends a lift to its inverse, fusion multiplies two lifts, and the unit uses the identity lift. These operations commute with the central extension, and multiplication in $\widetilde{\Gamma_x}$ is associative. Thus the descended lines carry inversion and unit identifications; the fusion and associativity identifications agree after pullback to the common unrigidified multiple-inertia spaces. 
\end{proof}

\begin{lemma}
\label{lem:rigidification-descent}
Let $\rho:\mathcal F\to\overline{\mathcal F}$ be a gerbe banded by a finite cyclic group $H$, and let $\mathscr L$ be a rank-one local system on $\mathcal F$. Then $\mathscr L$ is isomorphic to the pullback of a local system on $\overline{\mathcal F}$ if and only if the inertia action of $H$ on one, equivalently every, fiber of $\mathscr L$ is trivial. Under this condition, rational pullback is an isomorphism and
\[
 \rho_*\rho^*=\frac1{|H|}\id,
 \qquad (\rho^*)^{-1}=|H|\rho_*.
\]
\end{lemma}

\begin{proof}
The assertion is local on $\overline{\mathcal F}$ for a topology in which the gerbe is neutral. On such a chart the morphism is $BH\times U\to U$. A local system on $BH\times U$ is a local system on $U$ together with a representation of $H$ on its fiber, and it is a pullback precisely when that representation is trivial. These local descents glue because pullback is fully faithful on the trivial-isotropy summand. The pushforward formulas follow.
\end{proof}

By Proposition~\ref{prop:canonical-rigidification-descent} and
Lemma~\ref{lem:rigidification-descent}, rational pullback is an isomorphism for the transgression coefficient line, and pushforward satisfies
\begin{equation}\label{eq:rho-normalization}
 (\rho_{\mathcal F})_*(\rho_{\mathcal F})^* =\frac1{s_{\mathcal F}}\id.
\end{equation}
Hence the comparison maps are
\begin{equation}\label{eq:Jrig}
 \Jrig_{\mathcal F}=(\rho_{\mathcal F})^*,\qquad
 \Jrig_{\mathcal F}^{-1}=s_{\mathcal F}(\rho_{\mathcal F})_*.
\end{equation}
Put $$\Jrig=\bigoplus_{\mathcal F}\Jrig_{\mathcal F}.$$
A standard rigidified Chen--Ruan operator $F^{\rig,\CR}$ corresponds on the unshifted, unrigidified Chern-character target to
\begin{equation}\label{eq:transport-rigid-to-del}
 F^{\del}
 =\St_{\mathrm{age}}\Jrig F^{\rig,\CR}\Jrig^{-1}\St_{\mathrm{age}}^{-1}.
\end{equation}
On the identity sector $s_{\mathcal F}=1$, so rigidification has no effect.

\section{Fourier--Mukai and Steinberg operators}\label{sec_Fourier-Mukai_Steinberg}

We prove the untwisted identity-sector quotient formula and formulate precise sufficient conditions for a Fourier--Mukai comparison of the independently defined enumerative operators.  

\subsection{The trivial-torsor quotient formula}

Again let $q:X\to\widehat X=[X/\Gamma]$ be the quotient map. Let $B\subset H_2(X,\zz)$ be a finite $\Gamma$-stable set of nonzero effective curve classes and put
\[
 \mathfrak M_B=\coprod_{\beta\in B}\overline M_{0,2}(X,\beta).
\]

Let $$\mathfrak N_B^{\triv}\subset \coprod_{\beta\in B}\overline{M}_{0,2}(\widehat{X}, q_*\beta)$$ 
be the union of loci of maps to $\widehat X$ with identity marked sections whose geometric lifts have classes in $B$. We may work separately for each quotient class $\bar\beta$ occurring. The principal $\Gamma$-bundle is trivial on each geometric source curve, by Lemma~\ref{lem:two-point-monodromy}. In families, it is locally pulled
back from a torsor on the base, not necessarily globally trivial. The condition that the lift class belongs to the $\Gamma$-stable set $B$ is independent of the choice of local trivialization and is locally constant. Put
\[
 Z_B^X=\sum_{\beta\in B}Z_\beta^{\red},\qquad
 Z_B^{\widehat X,\triv}
 = (\ev_1,\ev_2)_*[\mathfrak N_B^{\triv}]^{\red}.
\]
These are sums and pushforwards of cycles, not disjoint unions of cycle classes. The reduced construction commutes with the finite \'etale base changes below: its standard obstruction complex, symplectic cosection, kernel cone and evaluation descend equivariantly.

\begin{theorem}
\label{thm:quotient-formula}
There is a canonical equivalence
\begin{equation}\label{eq:trivial-map-stack}
 \mathfrak N_B^{\triv}\simeq[\mathfrak M_B/\Gamma],
\end{equation}
where $\Gamma$ acts by postcomposition. In equivariant Borel--Moore homology, or in rational equivariant Chow groups, one has
\begin{equation}\label{eq:quotient-pullback}
 (q\times q)^*Z_B^{\widehat X,\triv}
 =\sum_{\eta\in\Gamma}(\id\times\tau_\eta)_*Z_B^X
\end{equation}
and
\begin{equation}\label{eq:quotient-pushforward}
Z_B^{\widehat X,\triv}=\frac1{|\Gamma|}(q\times q)_*Z_B^X.
\end{equation}
If $\beta$ itself is $\Gamma$-invariant, this specializes to
\begin{equation}\label{eq:single-beta-quotient}
 Z_{\{\beta\}}^{\widehat X,\triv}=\frac1{|\Gamma|}(q\times q)_*Z_\beta^X.
\end{equation}
\end{theorem}

\begin{proof}
After an \'etale localization on the parameter scheme, an object of 
$\mathfrak N_B^{\triv}$ admits a trivialization of its principal bundle and is then exactly a stable map to $X$ of class in $B$. To see that the bundle is locally pulled back from the base, its restriction to one identity marking defines a base torsor. The scheme of trivializations on each genus-zero fiber is the same torsor, because that fiber has trivial finite \'etale fundamental group. Proper base change for the finite bundle, or the relative scheme of these trivializations, gives the asserted identification in families. Changing the trivialization by $\eta\in\Gamma$ postcomposes the lift by $\tau_\eta$. Descent therefore gives \eqref{eq:trivial-map-stack}.  Automorphism groups are retained by the quotient-stack description, so no informal passage to coarse spaces occurs.

Now base-change the two evaluation maps by the atlas $q\times q:X^2\to\widehat X^2$. A lift of the first marked evaluation fixes a trivialization of the principal $\Gamma$-bundle at the first marking. Relative to it, a lift at the second marking differs by a
unique element $\eta\in\Gamma$. Consequently there is a canonical disjoint-union presentation
\begin{equation}\label{eq:evaluation-base-change}
 \mathfrak N_B^{\triv}\times_{\widehat X^2}X^2
 \simeq\coprod_{\eta\in\Gamma}\mathfrak M_B,
\end{equation}
and the evaluation on the $\eta$-summand is $(\ev_1,\tau_\eta\circ\ev_2)$. Flat pullback of the reduced cycle along the finite \'etale atlas and proper base change for evaluation give \eqref{eq:quotient-pullback}.

The map $q\times q$ has degree $|\Gamma|^2$, hence
\[
 (q\times q)_*(q\times q)^*=|\Gamma|^2
\]
on rational Chow and Borel--Moore homology. Moreover $q\circ\tau_\eta=q$, so pushing \eqref{eq:quotient-pullback} forward gives
\[
 |\Gamma|^2Z_B^{\widehat X,\triv}=|\Gamma|(q\times q)_*Z_B^X.
\]
Division by $|\Gamma|^2$ proves \eqref{eq:quotient-pushforward}, and the single-class formula follows. Since the identity sector has $s_{\mathcal F}=1$, rigidification does not change this calculation. 
\end{proof}

\begin{remark}
For the degree-zero diagonal, \eqref{eq:quotient-pushforward} gives
\[
 \frac1{|\Gamma|}(q\times q)_*[\Delta_X] =[\Delta_{\widehat X}],
\]
because $\Delta_X\to\Delta_{\widehat X}$ has degree
$|\Gamma|$.  This is a useful check that the factor is neither $1$ nor
$1/|\Gamma|^2$.
\end{remark}

\begin{proposition}
\label{prop:formal-transport}
Fix one additive lift $\widehat\Phi^{H,\mathrm{add}}_{T,k}$ from
Theorem~\ref{thm:additive-lift} and an endomorphism $F_X$ of its source. Then
\begin{equation}\label{eq:formal-transport}
 F_{\widehat X}^{\mathrm{transport}}
 :=\widehat\Phi^{H,\mathrm{add}}_{T,k}
 F_X(\widehat\Phi^{H,\mathrm{add}}_{T,k})^{-1}
\end{equation}
is the unique endomorphism intertwined with $F_X$ by that chosen lift. If a second lift is $U\widehat\Phi^{H,\mathrm{add}}_{T,k}$, then
\[
 F_{\widehat X}^{\mathrm{transport}\prime}
 =UF_{\widehat X}^{\mathrm{transport}}U^{-1}.
\]
\end{proposition}

\begin{proof}
The intertwining identity follows by direct substitution, and uniqueness follows by multiplying on the right by the inverse of the chosen lift. The change-of-lift formula is equally immediate. Since the matrix lift in Theorem~\ref{thm:additive-lift} is arbitrary beyond its residue modulo $z$, all geometric properties of the transported operator require separate proof.
\end{proof}

\subsection{Twisted \texorpdfstring{$K$}{K}-convolution, cohomological convolution and orbifold Riemann-Roch}

Let $Y_0,Y_1,Y_2$ be smooth finite-quotient complex Deligne--Mumford stacks with commuting $T$-actions and proper maps $h_i:Y_i\to \A$. Fix finite $T$-equivariant ungraded twists $\alpha_i$ on $Y_i$. For every pair $Y_i, Y_j$, put
\[
 W_{ij}:=Y_i\times Y_j, \qquad H_{ij}:=Y_i\times_\A Y_j\subset W_{ij}.
\]
The ordinary product $W_{ij}$ is the ambient correspondence space; $H_{ij}$ is only the support locus.

\begin{definition}\label{def:common-support}
For closed $T$-invariant subsets $Z\subset Z'\subset W$, the inclusions of relative pairs give canonical extension-of-support maps
\[
 \KU^*_{T,Z}(W;\tau)\longrightarrow\KU^*_{T,Z'}(W;\tau),
 \qquad
 \widehat H^*_{T,Z,\per}(IW;\mathscr M)
 \longrightarrow
 \widehat H^*_{T,Z',\per}(IW;\mathscr M).
\]
If two supported kernels have respective closed supports $Z_L$ and $Z_R$ in the same ambient outer product, an equality between them means the equality of their canonical images in the group with support $Z_L\cup Z_R$. This convention is used in both $K$-theory and completed cohomology and is compatible with pullback, product, fusion, and proper supported pushforward.
\end{definition}

A supported kernel from $(Y_i,\alpha_i)$ to $(Y_j,\alpha_j)$ is a class
\begin{equation}\label{eq:K-kernel-type}
 L_{ij}\in 
 \KU^0_{T,Z_{ij}}\bigl(W_{ij},\pr_i^*\alpha_i^{-1}\otimes\pr_j^*\alpha_j\bigr),
 \qquad Z_{ij}\subset H_{ij},
\end{equation}
where $Z_{ij}$ is closed and proper over $Y_j$. The pair 
$(W_{ij},W_{ij}\setminus Z_{ij})$ is assumed to have finite equivariant CW type whenever completed Chern characters are used. Pullbacks are derived when necessary; equivalently, one may assume the kernels are perfect and the displayed intersections are Tor-independent.

The supported transform is
\begin{equation}\label{eq:K-transform}
 \Phi^{\KU}_{L_{ij}}(x)=(\pr_j)_{!,Z_{ij}}(\pr_i^*x\cdot L_{ij}).
\end{equation}
Although $\pr_j:W_{ij}\to Y_j$ need not be proper, its restriction to the support is proper, so the complex-oriented supported index is defined. The two source twists cancel and the output is $\alpha_j$-twisted.

For composable kernels, work on the ordinary triple product 
$W_{012}=Y_0\times Y_1\times Y_2$. Assume
\[
 Z_{012}:=\pr_{01}^{-1}Z_{01}\cap\pr_{12}^{-1}Z_{12}
\]
is proper over $W_{02}=Y_0\times Y_2$. Its image is contained in $H_{02}$. Define
\begin{equation}\label{eq:K-convolution}
 L_{12}\star_{\KU}L_{01}
 :=(\pr_{02})_!\bigl(\pr_{01}^*L_{01}\cdot\pr_{12}^*L_{12}\bigr)
 \in
 \KU^0_{T,\pr_{02}(Z_{012})}\bigl(W_{02},\pr_0^*\alpha_0^{-1}\otimes\pr_2^*\alpha_2\bigr).
\end{equation}

\begin{proposition}\label{prop:K-convolution}
Under the preceding properness and perfectness hypotheses,
\begin{equation}\label{eq:K-composition}
 \Phi^{\KU}_{L_{12}\star_{\KU}L_{01}}
 =\Phi^{\KU}_{L_{12}}\circ\Phi^{\KU}_{L_{01}}.
\end{equation}
Consequently, a type-correct ambient-product kernel identity
\begin{equation}\label{eq:K-intertwining-kernels}
P\star_{\KU}K_X=K_{\widehat X}\star_{\KU}P
\end{equation}
implies the corresponding intertwining identity of $K$-theory operators. If the two proper outer supports are $Z_L$ and $Z_R$, then \eqref{eq:K-intertwining-kernels} is, by
Definition~\ref{def:common-support}, an equality in
\[
 \KU^0_{T,Z_L\cup Z_R}
 \bigl(W_{02},\pr_0^*\alpha_0^{-1}\otimes\pr_2^*\alpha_2\bigr).
\]  
If $P$ has a convolution inverse $Q$, the transported kernel is
\begin{equation}\label{eq:K-transported-kernel}
 K_{\widehat X}=P\star_{\KU}K_X\star_{\KU}Q.
\end{equation}
\end{proposition}

\begin{proof}
Pull the input and the two kernels to $W_{012}$. Associativity of the $\KU$-product, the projection formula, and functoriality of supported complex-oriented pushforward identify the transform of \eqref{eq:K-convolution} with the iterated transform. The stated properness is exactly what permits both supported pushforwards. The intertwining and transport assertions follow by applying this composition formula and, for the last assertion, convolving with $Q$.
\end{proof}

\begin{remark}
\label{rem:identity-test}
If $Y_i=Y_j=\A$, both structure maps are the identity, and the kernel is the supported class $\Delta_!1\in\KU^0_{T,\Delta_\A}(\A\times\A)$, then \eqref{eq:K-transform} is the identity. The relative projection is the ordinary ambient projection 
$\A\times \A\to \A$, restricted to diagonal support. This test rules out the earlier convention that treated the fiber product itself as the ambient kernel space while retaining the ambient-product degree and Todd normalization.
\end{remark}

Let $\mathcal F_i\subset IY_i$ be connected unrigidified inertia components, put $d_i=\dim_\C\mathcal F_i$, and abbreviate $\mathscr L_i=\mathscr L_{\alpha_i,\mathcal F_i}$. Let $Z_{ij}^{\mathcal F_i,\mathcal F_j}$ be the restriction of the inertia
support to the ordinary component $\mathcal F_i\times\mathcal F_j$; it is contained in $\mathcal F_i\times_{\mathbb{A}}\mathcal F_j$.

A total-degree-zero periodic kernel from $\mathcal F_i$ to $\mathcal F_j$ is
\begin{equation}\label{eq:H-kernel-type}
 \Xi_{ij}\in
 \widehat H^0_{T,Z_{ij}^{\mathcal F_i,\mathcal F_j},\per}
 \bigl(\mathcal F_i\times\mathcal F_j;
 \pr_i^*\mathscr L_i^\vee\otimes\pr_j^*\mathscr L_j\bigr).
\end{equation}
The completed supported group is the group defined in 
\eqref{eq:supported-completion-def}. Define
\begin{equation}\label{eq:H-transform}
 \Phi^H_{\Xi_{ij}}(a)
 :=u^{d_i}(\pr_j)_{*,Z_{ij}}\bigl(\pr_i^*a\cup\Xi_{ij}\bigr),
\end{equation}
where $\pr_j:\mathcal F_i\times\mathcal F_j\to\mathcal F_j$ is the ordinary ambient projection. Its relative complex dimension is $d_i$, so $u^{d_i}$ restores total periodic degree.

For composable kernels, fusion and evaluation contract the two middle coefficient lines. On the ordinary triple product, define
\begin{equation}\label{eq:H-convolution}
 \Xi_{12}\star_H\Xi_{01}
 :=u^{d_1}(\pr_{02})_*\bigl(\pr_{01}^*\Xi_{01}\cup \pr_{12}^*\Xi_{12}\bigr),
\end{equation}
with the sum taken over all connected middle inertia components.

\begin{proposition}
\label{prop:H-convolution}
Assume the pulled-back supports in \eqref{eq:H-convolution} are proper over the outer product. Then
\begin{equation}\label{eq:H-composition}
 \Phi^H_{\Xi_{12}\star_H\Xi_{01}}
 =\Phi^H_{\Xi_{12}}\circ\Phi^H_{\Xi_{01}}.
\end{equation}
Consequently, an equality
\begin{equation}\label{eq:H-intertwining-kernels}
 \Xi_P\star_H\Xi_X=\Xi_{\widehat X}\star_H\Xi_P
\end{equation}
of ambient-product, coefficient-valued supported kernels implies the corresponding intertwining identity on completed unshifted delocalized cohomology. Sectorwise, if the two proper outer supports are $Z_L$ and $Z_R$, the equality is taken in the explicit common group
\[
 \widehat H^0_{T,Z_L\cup Z_R,\per}
 \bigl(\mathcal F_0\times\mathcal F_2;
 \pr_0^*\mathscr L_0^\vee\otimes\pr_2^*\mathscr L_2\bigr).
\]
\end{proposition}

\begin{proof}
Substitute \eqref{eq:H-transform} twice. The two periodic factors are $u^{d_0+d_1}$, which is the product of the outer factor $u^{d_0}$ and the middle factor in 
\eqref{eq:H-convolution}. The projection formula moves the first supported pushforward past the pullback of the second kernel, and functoriality identifies the iterated integration with the pushforward along $\pr_{02}$. Fusion cancels the middle transgression lines. This proves \eqref{eq:H-composition}; applying it to the two
sides of \eqref{eq:H-intertwining-kernels} proves the last assertion.
\end{proof}

\begin{definition}\label{defn_Td_quotient_stack}
Let $Y=[V/G]$ be a smooth finite quotient stack with a commuting $T$-action. On a connected component $\mathcal F\subset IY$ represented by $F\subset V^g$, decompose the moving normal bundle as $N_{F/V}=\bigoplus_{\lambda\ne1}N_\lambda$. If $x_{\lambda,j}$ are the $T$-equivariant Chern roots, put
\begin{align}
 D^T_{\mathcal F}(Y)
 &:=\prod_{\lambda\ne1}\prod_j
   \bigl(1-\lambda^{-1}e^{-x_{\lambda,j}}\bigr),
 \label{eq:moving-denominator}\\
 \widetilde\Td_T(TY)|_{\mathcal F}
 &:=\frac{\Td_T(T\mathcal F)}{D^T_{\mathcal F}(Y)}.
 \label{eq:orbifold-Todd}
\end{align}
The denominator has invertible degree-zero term. For an even class $c=\sum c_{2j}$, write $c_u=\sum c_{2j}u^{-j}$.
\end{definition}

For a kernel $L_{ij}$ of type \eqref{eq:K-kernel-type}, applying Definition \ref{defn_Td_quotient_stack}, then the supported trace Chern character on the ordinary inertia component has type
\begin{equation}\label{eq:trace-character-type}
 \widetilde\ch_{T,u}(L_{ij})_{\mathcal F_i,\mathcal F_j}
 \in\widehat H^0_{T,Z_{ij}^{\mathcal F_i,\mathcal F_j},\per}
 \bigl(\mathcal F_i\times\mathcal F_j;
 \pr_i^*\mathscr L_i^\vee\otimes\pr_j^*\mathscr L_j\bigr).
\end{equation}
Define
\begin{equation}\label{eq:tau-orb-definition}
 \tau^{T,\orb}(L_{ij})_{\mathcal F_i,\mathcal F_j}
 :=\widetilde\ch_{T,u}(L_{ij})_{\mathcal F_i,\mathcal F_j}
 \cup\pr_i^*\bigl(\widetilde\Td_T(TY_i)|_{\mathcal F_i}\bigr)_u,
\end{equation}
using the relative inertial Todd factor for the ordinary projection
$Y_i\times Y_j\to Y_j$.

\begin{theorem}[Completed supported ambient-product orbifold Riemann--Roch]
\label{thm:bivariant-RR}
Assume the state spaces and all relative pairs in question satisfy the finite equivariant CW hypotheses of Theorems~\ref{thm:general-completed-ch} and
\ref{thm:supported-completion}. Then:
\begin{enumerate}[label=\textup{(\roman*)}]
\item \eqref{eq:tau-orb-definition} is a well-defined completed supported class with exactly the coefficient line in \eqref{eq:trace-character-type}.

\item For $k\in\zz/2$, the diagram
\begin{equation}\label{eq:RR-square}
\begin{CD}
 (\KU_T^k(Y_i,\alpha_i)\otimes\C)^{\wedge}_{I_T}
 @>{\Phi^{\KU}_{L_{ij}}}>>
 (\KU_T^k(Y_j,\alpha_j)\otimes\C)^{\wedge}_{I_T}\\
 @V{\widetilde\ch^{\wedge}_{T,\alpha_i,u}}V{\cong}V
 @VV{\widetilde\ch^{\wedge}_{T,\alpha_j,u}}V\\
 \widehat H^k_{T,\del,\per}(Y_i,\alpha_i)
 @>{\Phi^H_{\tau^{T,\orb}(L_{ij})}}>>
 \widehat H^k_{T,\del,\per}(Y_j,\alpha_j)
\end{CD}
\end{equation}
is commutative.

\item For composable kernels with proper convolution support, put
\[
 Z_{02}^{\star}:=
 \pr_{02}\bigl(\pr_{01}^{-1}Z_{01}\cap\pr_{12}^{-1}Z_{12}\bigr)\subset W_{02}.
\]
For connected outer inertia components $\mathcal F_0\subset IY_0$ and
$\mathcal F_2\subset IY_2$, set
\[
 Z_{02}^{\star,\mathcal F_0,\mathcal F_2}
 :=IZ_{02}^{\star}\cap(\mathcal F_0\times\mathcal F_2).
\]
Then the following is an equality in the completed supported cohomology group of the ordinary outer product with the displayed common support:
\begin{equation}\label{eq:RR-convolution}
 \tau^{T,\orb}(L_{12}\star_{\KU}L_{01})
 =\tau^{T,\orb}(L_{12})\star_H
  \tau^{T,\orb}(L_{01})
\end{equation}
inside
\[
 \widehat H^0_{T,Z_{02}^{\star,\mathcal F_0,\mathcal F_2},\per}
 \bigl(\mathcal F_0\times\mathcal F_2;
 \pr_0^*\mathscr L_0^\vee\otimes\pr_2^*\mathscr L_2\bigr)
\]
for every pair $(\mathcal F_0,\mathcal F_2)$.
\end{enumerate}
\end{theorem}

\begin{proof}
Use $K=S^1$ and the finite free spaces $E_N$. At the $N$-th Borel approximation, the ambient product is
\begin{equation}\label{eq:Borel-ambient-product}
 E_N\times_K(Y_i\times Y_j) \simeq (Y_i)_N\times_{B_N}(Y_j)_N.
\end{equation}
Its projection to $(Y_j)_N$ has vertical tangent bundle equal to the pullback of $TY_i$ and, on the fixed sector $\mathcal F_i$, relative complex dimension $d_i$. The support condition $Z_{ij}\subset Y_i\times_\A Y_j$ changes neither statement. This is where the ambient-product convention is essential.

Choose an integer $s$ annihilating the finite-order twisting classes $\alpha_i$ and represent all of them by banded $\mu_s$-gerbes; then use the central-weight idempotents. The trace of a source-inverse/target-positive kernel has coefficient line $\mathscr L_i^\vee\otimes\mathscr L_j$; weights and middle coefficient lines cancel under tensor product by fusion. Apply the Atiyah--Singer families theorem~\cite{ASFamilies} and the finite-group fixed-point and stack Riemann--Roch formulas~\cite{ASIndex}, \cite{KawasakiRR}, \cite{ToenRR}, \cite{EdidinDMRR} to the supported projection in \eqref{eq:Borel-ambient-product}. The fixed-point ratio of source and target moving-normal factors is precisely $\pr_i^*\widetilde\Td_T(TY_i)$; the target factor cancels. The family index lowers ordinary degree by $2d_i$, so periodization gives the factor $u^{d_i}$. This proves the finite-approximation version of \eqref{eq:RR-square} in even parity. Odd parity follows by external product with the suspension generator; pullback, multiplication by the even kernel, trace projectors, and supported pushforward commute with suspension.

Apply the same index formula to the supported middle projection on 
$Y_0\times Y_1\times Y_2$. At every finite stage, it gives the equality of supported kernel classes corresponding to \eqref{eq:RR-convolution}. The relative tangent bundle of that middle projection is the pullback of $TY_1$, so the factor is $u^{d_1}$ and
the middle inertial Todd class is exactly the one in the cohomological convolution. Functoriality of the family index, the projection formula, and associativity of fusion give the convolution identity.

Finally pass to the inverse limit. Theorem~\ref{thm:general-completed-ch} controls the state-space towers, while Theorem~\ref{thm:supported-completion} controls the relative-pair towers containing the supported kernel classes. Both are Mittag--Leffler and have no $\varprojlim^1$ term. Thus the finite-stage operator square and, separately, the finite-stage equality of supported classes pass to the completed limits.
\end{proof}

\begin{remark}
The theorem is unrigidified and additive. It does not construct any of the kernels, identify a Chow cycle with a virtual structure sheaf, or eliminate nonidentity Kawasaki sectors. Rigidified orbifold Gromov--Witten classes enter only after the canonical descent of Proposition~\ref{prop:canonical-rigidification-descent}, the
stack-normalized transport \eqref{eq:transport-rigid-to-del}, and an independently constructed coefficient-valued reduced orbifold theory.
\end{remark}

\subsection{FM vs Steinberg}\label{subsec_FM_Steinberg}

Fix either a $\Gamma$-invariant effective curve class $\beta$, or a finite orbit $B=\Gamma\cdot\beta$ and replace every $\beta$-indexed cycle, kernel, and operator below by the corresponding orbit sum. The notation $\beta$ below covers both conventions, and a genuinely single-class statement is made only in the $\Gamma$-invariant case.

For the reduced $\SL_r$ cycle, regarded as a class supported in the ordinary smooth ambient product $X\times X$, put
\begin{equation}\label{eq:SL-cohomological-kernel}
 \Xi^X_\beta:=\cl_{T,Z^X_\beta}(Z^X_\beta)u^{-m}.
\end{equation}
Here $\cl_{T,Z^X_\beta}$ is the supported cohomology class in the smooth ambient product corresponding, by duality, to the Borel--Moore cycle class. Its ordinary cohomological degree is $2m$, so the Bott factor gives total degree zero. Its support lies in $X\times_\A X$. The finite-gerbe transgression lines descend canonically by
Proposition~\ref{prop:canonical-rigidification-descent}. For the quotient, assume an independently constructed standard rigidified twisted reduced theory with a fusion-compatible coefficient-valued virtual class and the full inner-local-system data. Let $F^{\widehat X,\rig,\CR}_{\bar\beta}$ be its twisted two-point operator, with the inversion, pairing and stabilizer normalizations of Definition~\ref{defn_Steinberg_PGL}. That earlier definition constructs the untwisted operator; the present twisted coefficient theory remains an additional hypothesis. Set
\begin{equation}\label{eq:PGL-del-operator}
 F^{\widehat X,\del}_{\bar\beta}
 :=\St_{\mathrm{age}}\Jrig F^{\widehat X,\rig,\CR}_{\bar\beta}
 \Jrig^{-1}\St_{\mathrm{age}}^{-1}.
\end{equation}
Let $\Xi^{\widehat X,\del}_{\bar\beta}$ be its unshifted, unrigidified coefficient-valued kernel on ordinary inertia products. On $(\mathcal F_1,\mathcal F_2)$, the reduced cycle has dimension $D_{12}=m-a_{\mathcal F_1}-a_{\mathcal F_2}$ and is periodized by $u^{-D_{12}}$.

\begin{theorem}
\label{thm:cohomological-95}
Assume:
\begin{enumerate}[label=\textup{(C\arabic*)}]
\item an independently constructed standard rigidified twisted reduced orbifold theory whose coefficient-valued virtual class and two-point operator are compatible with the canonical descended transgression lines of 
Proposition~\ref{prop:canonical-rigidification-descent} and with the fusion, inversion, and unit coherences;

\item an ambient-product coefficient-valued kernel $\Xi_P$ from $X$ to $\widehat X$, supported in $X\times_\A\widehat X$, whose transform is an isomorphism and satisfies
\begin{equation}\label{eq:cohomological-GS-reduction}
 \Phi^H_{\Xi_P}\bmod z=\Phi^H_{\GS};
\end{equation}

\item the independently defined enumerative kernels $\Xi^X_\beta$ and 
$\Xi^{\widehat X,\del}_{\bar\beta}$; and

\item the ambient-product supported identity
\begin{equation}\label{eq:cohomological-global-convolution}
 \Xi_P\star_H\Xi^X_\beta=\Xi^{\widehat X,\del}_{\bar\beta}\star_H\Xi_P.
\end{equation}
Writing $Z_P^H$, $Z_X^H$, and $Z_{\widehat X}^H$ for the supports of the three cohomological kernels, define $Z_L^H$ and $Z_R^H$ to be the proper outer images of, respectively,
\[
 \pr_{01}^{-1}Z_X^H\cap\pr_{12}^{-1}Z_P^H,
 \qquad
 \pr_{01}^{-1}Z_P^H\cap\pr_{12}^{-1}Z_{\widehat X}^H.
\]
For every connected outer component $\mathcal F\subset I\widehat X$, identity \eqref{eq:cohomological-global-convolution} is required in
\[
 \widehat H^0_{T,(Z_L^H\cup Z_R^H)\cap(X\times\mathcal F),\per}
 \bigl(X\times\mathcal F;\pr_{\mathcal F}^*\mathscr L_{\alpha,\mathcal F}\bigr),
\]
after the canonical extension-of-support maps of Definition~\ref{def:common-support}.
\end{enumerate}
Then
\begin{equation}\label{eq:del-intertwining}
 F^{\widehat X,\del}_{\bar\beta}\Phi^H_{\Xi_P}=\Phi^H_{\Xi_P}F^X_\beta.
\end{equation}
The horizontal map to the standard rigidified Chen--Ruan-periodized state space is
\begin{equation}\label{eq:CR-horizontal}
 \Phi_P^{\rig,\CR}:=\Jrig^{-1}\St_{\mathrm{age}}^{-1}\Phi^H_{\Xi_P},
\end{equation}
and it satisfies
\begin{equation}\label{eq:CR-intertwining}
 F^{\widehat X,\rig,\CR}_{\bar\beta}\Phi_P^{\rig,\CR}=\Phi_P^{\rig,\CR}F^X_\beta.
\end{equation}
\end{theorem}

\begin{proof}
Proposition~\ref{prop:H-convolution} converts
\eqref{eq:cohomological-global-convolution} into \eqref{eq:del-intertwining}.  Conjugate this identity by $\Jrig^{-1}\St_{\mathrm{age}}^{-1}$ and use
\eqref{eq:PGL-del-operator}; the adjacent regrading and rigidification maps cancel and give \eqref{eq:CR-intertwining}. Condition \eqref{eq:cohomological-GS-reduction} identifies the nonequivariant reduction of the horizontal map with the chosen Groechenig--Shen map. Without that condition the theorem would concern an unspecified
integral transform.
\end{proof}

\begin{remark}\label{rem:independent-C4}
Condition \textup{(C1)} should be a straightforward adaptation of the work of \cite{PRY}. Condition \textup{(C4)} is not a consequence of 
Theorem~\ref{thm:block-decomposition} or of the definition of convolution. Character decomposition only says that each operator is block diagonal. Even when $\Gamma$ is trivial, take a one-dimensional state space, $F_X=0$, $F_{\widehat X}=\id$ and $\Phi=\id$. All equivariance requirements hold, but $F_{\widehat X}\Phi\ne\Phi F_X$.  This elementary example disproves the asserted formal implication; it is not a counterexample to an enumerative comparison for Higgs moduli. A geometric proof of
\eqref{eq:cohomological-global-convolution} is additional work.
\end{remark}

\begin{theorem}[A $K$-theoretic sufficient criterion]\label{thm:corrected-95}
Assume the following stronger data.
\begin{enumerate}[label=\textup{(H\arabic*)}]
\item The independently constructed twisted reduced orbifold theory of \textup{(C1)}, including its fusion-compatible coefficient-valued virtual class and standard rigidified two-point operator.

\item A global ambient-product supported kernel
\[
 P\in\KU^0_{T,Z_P}(X\times\widehat X,\pr_{\widehat X}^*\alpha_T),
 \qquad Z_P\subset X\times_{\mathbb{A}}\widehat X,
\]
whose transform is an isomorphism and whose kernel-induced map of $\KU$-module spectra satisfies
\begin{equation}\label{eq:P-recovers-GS}
 \mbox{For}_T(\Phi_P^{\KU})=\Phi_{\GS}^{\KU}
\end{equation}
after forgetting the $T$-action. $\mbox{For}_T(-)$. If only the weaker equality modulo
$z$ is known, replace \eqref{eq:P-recovers-GS} by $\Phi_P^H\bmod z=\Phi_{\GS}^H$ and call the result only a cohomological lift. If $P$ is called a Fourier equivalence, assume a type-correct convolution inverse $Q$.

\item Ambient-product supported kernels
\begin{align*}
 K_{X,\beta}&\in\KU^0_{T,Z_X}(X\times X),
 &Z_X&\subset X\times_{\mathbb{A}} X,\\
 K_{\widehat X,\bar\beta}&\in
 \KU^0_{T,Z_{\widehat X}}
 (\widehat X\times\widehat X,
 \pr_1^*\alpha_T^{-1}\otimes\pr_2^*\alpha_T),
 &Z_{\widehat X}&\subset\widehat X\times_{\mathbb{A}}\widehat X,
\end{align*}
with proper convolution supports. All three original support pairs, including the support pair of $P$, have finite equivariant CW type. Moreover, for every convolution in \textup{(H5)} and for every inverse convolution involving $Q$ when present, the relative pair defined on the ordinary triple product by the pulled-back intersection
\[
 \pr_{01}^{-1}Z_{01}\cap\pr_{12}^{-1}Z_{12}
\]
and the relative pair defined by its proper image in the ordinary outer product both have finite equivariant CW type. The same requirement is imposed on every support pair used to represent the resulting convolution class and on every finite union of outer supports used as a common support in Definition~\ref{def:common-support}.

\item Their completed ambient-product Riemann--Roch classes are the independently defined enumerative kernels:
\begin{equation}\label{eq:RR-enumerative-identification}
 \tau^{T,\orb}(K_{X,\beta})=\Xi^X_\beta,
 \qquad
 \tau^{T,\orb}(K_{\widehat X,\bar\beta})=\Xi^{\widehat X,\del}_{\bar\beta}.
\end{equation}

\item Define the two proper outer supports in $X\times\widehat X$ by
\begin{align*}
 Z_L^{\KU}
 &:={\pr}_{02}\bigl({\pr}_{01}^{-1}Z_X
       \cap {\pr}_{12}^{-1}Z_P\bigr)
 &&\text{on }X\times X\times\widehat X,\\
 Z_R^{\KU}
 &:={\pr}_{02}\bigl({\pr}_{01}^{-1}Z_P
       \cap {\pr}_{12}^{-1}Z_{\widehat X}\bigr)
 &&\text{on }X\times\widehat X\times\widehat X.
\end{align*}
The ambient-product supported $K$-kernel identity is required after extension of supports in the explicit common group
\begin{equation}\label{eq:missing-global-convolution}
 P\star_{\KU}K_{X,\beta}=K_{\widehat X,\bar\beta}\star_{\KU}P
 \quad\text{in }
 \KU^0_{T,Z_L^{\KU}\cup Z_R^{\KU}}(X\times\widehat X,\pr_{\widehat X}^*\alpha_T).
\end{equation}
\end{enumerate}
Then all hypotheses of Theorem~\ref{thm:cohomological-95} hold, and its intertwining conclusions follow. 
\end{theorem}
\begin{proof}
Proposition~\ref{prop:K-convolution} turns \eqref{eq:missing-global-convolution} into the corresponding equality of completed $K$-theory operators. Theorem~\ref{thm:bivariant-RR}, whose kernel-level passage to completion is justified by
Theorem~\ref{thm:supported-completion}, sends the $K$-kernel identity (\ref{eq:missing-global-convolution}) to
\[
 \tau^{T,\orb}(P)\star_H\Xi^X_\beta 
 =\Xi^{\widehat X,\del}_{\bar\beta}\star_H\tau^{T,\orb}(P).
\]
Put $\Xi_P=\tau^{T,\orb}(P)$. The general completed Chern-character square shows that its transform is an isomorphism. Forgetting the $T$-action in \eqref{eq:P-recovers-GS}, or merely reducing modulo $z$ in the weaker version, gives
\eqref{eq:cohomological-GS-reduction}. Thus Theorem~\ref{thm:cohomological-95} applies.
\end{proof}

\begin{corollary}[Kernel transport and the enumerative identification]
\label{cor:genuine-kernel-transport}
Suppose $P$ satisfies \textup{(H2)}, has convolution inverse $Q$, and all support and finite-type hypotheses needed below hold. Then
\[
K_{\widehat X,\bar\beta}^{\FM}:=P\star_{\KU}K_{X,\beta}\star_{\KU}Q
\]
is its transported kernel. Identifying the transported cohomological operator with the independently constructed enumerative operator still requires, for example, the separate supported equality
\begin{equation}\label{eq:transport-enumerative-test}
 \tau^{T,\orb}(K_{\widehat X,\bar\beta}^{\FM})=\Xi^{\widehat X,\del}_{\bar\beta}.
\end{equation}
\end{corollary}
\begin{proof}
Convolution with $P$ and $Q$ gives the transported kernel by
Proposition~\ref{prop:K-convolution}; the Riemann--Roch square in Theorem \ref{thm:bivariant-RR} gives its cohomological action. That action is the enumerative one if \eqref{eq:transport-enumerative-test} holds. The supported Chern character, followed by the invertible relative Todd factor, is an isomorphism only after complexification and completion under the hypotheses used here.  
\end{proof}

\begin{proposition}[Additional conditions for the full divisor operator]
\label{prop:full-divisor-comparison}
Suppose an isomorphism $\Phi$, independent of the Novikov variables, intertwines all reduced positive-degree two-point operators. Suppose also that a chosen correspondence of curve degrees and divisors matches the Novikov monomials and divisor pairings, and that
\[
 \Phi(D_T\cup-)= (\widehat D_T\cup_{\CR}-)\Phi.
\]
Then, over a common Novikov completion, $\Phi M_D=M_{\widehat D}\Phi$.
\end{proposition}
\begin{proof}
Use the divisor formulas on both sides. The hypothesis matches the degree-zero terms.  The curve-degree convention, divisor pairing and coefficientwise intertwining match every positive-degree term. Completeness permits their sum. This proves the operator assertion stated. 
\end{proof}

Let $\A^{\mathrm{sm}}\subset \A$ be the smooth-spectral locus and let $X^{\mathrm{sm}}$ and $\widehat X^{\mathrm{sm}}$ be its inverse images.

\begin{proposition}
\label{prop:smooth-vanishing}
For every nonzero effective curve class, the reduced two-pointed stable-map
correspondences of $X$ and of $\widehat X$ restrict to $0$ over $\A^{\mathrm{sm}}$.  Hence every positive-degree square in Theorem~\ref{thm:corrected-95} commutes after restriction to that locus, with both vertical maps zero.  
\end{proposition}

\begin{proof}
A smooth fixed-determinant Hitchin fiber is a torsor under an Prym variety. After choosing a point, it becomes an abelian variety, and every morphism from $\mathbb P^1$ to it is constant: pulling back the translation-invariant one-forms shows that its differential vanishes Each irreducible component of a connected genus-zero nodal source is therefore mapped constantly, and connectedness forces the entire map to be constant.

For the quotient $\widehat{X}$, the finite group acts fiberwise by translations. The
coarse quotient of an abelian torsor by a finite translation subgroup is again a torsor under an abelian variety. The coarse map of a twisted genus-zero stable map to $\widehat{X}$ is therefore constant over $\A^{\mathrm{sm}}$. A map contained in a residual finite gerbe may carry orbifold monodromy, but its coarse curve class is $0$.  Thus no positive-degree stable map occurs on either side, proving the assertion.
\end{proof}

\appendix

\section{The cosection-reduced genus-zero virtual cycle}\label{sec:reduced}

Here we construct the reduced virtual fundamental cycle used in this paper, following \cite{KL}. 

\subsection{Rational curves are Hitchin-vertical}

\begin{lemma}\label{lem:hitchin-constant}
Let \(D\) be a connected proper nodal curve and let 
\(f:D\to X\) be a
morphism.  Then
\[
 h\circ f:D\longrightarrow\A
\]
is constant.
\end{lemma}

\begin{proof}
The variety \(\A\) is affine.  Since \(D\) is connected and proper, $H^0(D,\cO_D)=\C$. Every coordinate function on the affine space \(\A\) hence pulls back to a constant function on \(D\), so \(h\circ f\) is constant.
\end{proof}

For a family of connected nodal curves
\(\pi:\mathcal D\to S\), one has
\(\pi_*\cO_{\mathcal D}=\cO_S\).  The preceding argument is consequently
compatible with families and gives a Hitchin-value morphism
\[
 \rho_\beta:\Mbar_{0,n}(X,\beta)\longrightarrow\A.
\]

\begin{proposition}\label{prop:stablemap-proper}
For every \(\beta\) and \(n\), the morphism $\rho_{\beta}$ is proper. For \(n\ge1\), every evaluation morphism
\[
 \ev:\Mbar_{0,n}(X,\beta)\longrightarrow X^n_\A
\]
is proper.
\end{proposition}

\begin{proof}
By Theorem \ref{thm:hitchin-proper}, \(X\to\A\) is projective. Choose an \(h\)-ample line bundle \(H\) and put 
\(e=H\cdot\beta\). The stack of stable maps to the fibers of \(h\), with fixed genus, number of markings, and \(H\)-degree \(e\), is a finite-type proper stack over \(\A\); this is the relative form of the standard properness theorem
for stable maps, see \cite{BM} and \cite{FP}. In a connected family of stable maps, the pushforward fundamental class
\((f_s)_*[\mathcal C_s]\in H_2(X,\zz)\) is locally constant; equivalently, the moduli stack decomposes into open-and-closed substacks indexed by the integral curve class. Consequently, the locus of class \(\beta\) is open and closed in this finite-type stack. It is a finite union of connected components and is therefore proper over \(\A\). This also
proves that \(\Mbar_{0,n}(X,\beta)\) is of finite type.

For the evaluation map, write $\sM=\Mbar_{0,n}(X,\beta), Y=X_\A^n$. The morphism \(\sM\to\A\) is proper and \(Y\to\A\) is separated. The graph $\Gamma_{\ev}\hookrightarrow\sM\times_\A Y$ is therefore closed. The projection $\sM\times_\A Y\longrightarrow Y$ is the base change of the proper morphism \(\sM\to\A\), hence is proper.
Thus
\[
 \text{ev}: \sM\xrightarrow{\,\Gamma_{\ev}\,}\sM\times_\A Y\longrightarrow Y
\]
is proper.
\end{proof}

\begin{corollary}[Proper fixed stable-map locus]\label{cor:fixed-proper}
For every \(\beta\) and \(n\), $\Mbar_{0,n}(X,\beta)^T$ is proper.
\end{corollary}
\begin{proof}
A \(T\)-fixed stable map has \(T\)-fixed Hitchin value. Since $\A^{T}=\{0\}$, this value is \(0\). Therefore
\[
 \Mbar_{0,n}(X,\beta)^T \subset \rho_\beta^{-1}(0).
\]
The fiber \(\rho_\beta^{-1}(0)\) is proper by Proposition \ref{prop:stablemap-proper}, and the fixed locus is closed in it.
\end{proof}

The properness in Corollary \ref{cor:fixed-proper} allows us to define Gromov--Witten invariants of $X$ by virtual localization formula. 
Moreover, the two-point correspondence itself is an ordinary proper pushforward to \(X\times_\A X\), not merely a formal localization expression.

\subsection{The absolute Kiem--Li cosection}
Recall the notation $\sM_{n,\beta}=\Mbar_{0,n}(X,\beta)$. The virtual dimension formula is
\[
 \dim_\C[\Mbar_{0,n}(X,\beta)]^{\vir} = \dim X-3+n+\int_\beta c_1(T_X).
\]
Since \(c_1(T_X)=0\), for \(n=2\), we have 
\begin{equation}\label{eq:vdim-standard}
 \dim_\C [\sM_{2,\beta}]^{\vir}=m-1.
\end{equation}
Removing the symplectic obstruction quotient raises the virtual dimension by one:
\begin{equation}\label{eq:vdim-reduced}
 \dim_\C[\sM_{2,\beta}]^{\red}=m.
\end{equation}

Consider the universal family:
\[
\sM_{n,\beta}\xleftarrow{\,\pi\,}\mathcal{C}\xrightarrow{\, F\,} X.
\]
Relative to the Artin stack $\mathfrak M_{0,n}$ of prestable $n$-pointed curves, the standard perfect obstruction theory is
\[
 (R\pi_*F^*T_X)^\vee\longrightarrow L_{\sM_{n,\beta}/\mathfrak M_{0,n}}.
\]
Combining it with the cotangent triangle for \(\sM_{n,\beta}\to\mathfrak M_{0,n}\) gives the standard absolute perfect
obstruction theory \(E^\bullet\to L_{\sM_{n,\beta}}\).

The weight convention \(t^*\omega=t\omega\) means that contraction with \(\omega\) is an equivariant morphism
\[
 T_X\longrightarrow\Omega_X\otimes\C_1.
\]
Consequently the symplectic form and the differential of \(F\) give
\begin{equation}\label{eq:cosection-complex}
 F^*T_X
 \stackrel{\omega^{\flat}}{\longrightarrow}
 F^*\Omega_X\otimes\C_1
 \stackrel{dF}{\longrightarrow}
 \Omega_\pi\otimes\C_1
 \longrightarrow
 \omega_\pi\otimes\C_1 .
\end{equation}
After applying \(R\pi_*\), the map on relative obstruction sheaves has target
\[
 R^1\pi_*\omega_\pi\otimes\C_1
 \cong \cO_{\sM_{n,\beta}}\otimes\C_1.
\]

The point that must be checked here is absolute compatibility at moving and smoothing nodes. Kiem--Li prove precisely this statement: for a smooth quasi-projective target carrying a regular two-form, the relative map above descends canonically to a cosection of the \emph{absolute} stable-map obstruction sheaf; see \cite[Lemma~6.1 and Corollary~6.2]{KL}. Proposition~6.4 of \cite{KL} identifies its degeneracy locus with the two-form-null maps. Their argument is algebraic and allows a nonproper target. Applied equivariantly to \eqref{eq:cosection-complex}, it gives
\begin{equation}\label{eq:semireg-ob}
 \sigma_\omega:\operatorname{Ob}_{\sM_{n,\beta}}
 \longrightarrow
 \cO_{\sM_{n,\beta}}\otimes\C_1.
\end{equation}

\begin{lemma}[Surjectivity in nonzero degree]\label{lem:semireg-surj}
If \(\beta\ne0\), the cosection \eqref{eq:semireg-ob} is everywhere surjective.
\end{lemma}

\begin{proof}
At a stable map \(f:D\to X\), the relative obstruction map is
\[
 H^1(D,f^*T_X)\longrightarrow H^1(D,\omega_D)\otimes\C_1\cong\C_1.
\]
Its Serre dual sends the generator of \(H^0(D,\cO_D)\) to the section obtained from \(df\) by the nondegenerate pairing \(\omega\). Since \(\beta\ne0\), the map is nonconstant on some irreducible component. In characteristic zero, \(df\) is nonzero at the generic point of that component. The dual map from a one-dimensional vector space is therefore
injective, so the displayed obstruction map is surjective. By \cite[Proposition~6.4]{KL}, the absolute and relative cosections have the same degeneracy locus. Hence \(\sigma_\omega\) is everywhere surjective. Equivalently, Kiem--Li's degeneracy locus consists of \(\omega\)-null maps; nondegeneracy of \(\omega\) makes such a map constant.
\end{proof}

\subsection{The abstract kernel-cone construction}

We isolate the bundle-stack foundation used below. This also fixes the line-bundle-valued and equivariant descent in the reduced-class construction.

\begin{proposition}[Kernel-cone lemma]\label{prop:kernel-cone-general}
Let \(M\) be a finite-type Deligne--Mumford stack with a \(T\)-equivariant perfect obstruction theory
\(E^\bullet\to L_M\) of virtual dimension \(v\). Let \(\mathcal L\) be a \(T\)-equivariant line bundle and suppose that there is a surjection
\[
 \sigma:\operatorname{Ob}_M=h^1\bigl((E^\bullet)^\vee\bigr)\twoheadrightarrow\mathcal L.
\]
Then \(\sigma\) canonically induces a morphism of vector-bundle stacks
\[
 \widetilde\sigma: \mathfrak E:=h^1/h^0\bigl((E^\bullet)^\vee\bigr) \longrightarrow\mathbb V(\mathcal L).
\]
Here \(\mathbb V(\mathcal L)=\operatorname{Spec}_M\operatorname{Sym}(\mathcal L^\vee)\) is the total-space vector bundle. Its kernel \(\mathfrak K\) is a vector-bundle stack of virtual rank \(-v-1\). The intrinsic-normal-cone cycle is supported in that kernel,
\[
 [\mathfrak C_M]\in Z_*^T(\mathfrak K).
\]
Writing \(\pi_{\mathfrak K}:\mathfrak K\to M\), homotopy invariance gives
\[
 \pi_{\mathfrak K}^*:A_j^T(M)_{\qq} \xrightarrow{\,\sim\,}A_{j-v-1}^T(\mathfrak K)_\qq.
\]
Consequently, \(0^!_{\mathfrak K}:=(\pi_{\mathfrak K}^*)^{-1}\) is the presentation-independent zero-section Gysin map, and the class
\begin{equation}\label{eq:abstract-kernel-class}
 [M]^{\mathrm{ker}}_T :=0^!_{\mathfrak K}[\mathfrak C_M] \in A_{v+1}^T(M)_\qq
\end{equation}
satisfies
\begin{equation}\label{eq:abstract-excess}
 [M]^{\vir}_T =c_1^T(\mathcal L)\cap[M]^{\mathrm{ker}}_T.
\end{equation}
\end{proposition}

\begin{proof}
Behrend--Fantechi construct \(h^1/h^0\) intrinsically and functorially from the derived object; see \cite[Propositions~2.4 and~2.5]{BF}. In the derived category, the canonical truncation followed by the cosection,
\[
 (E^\bullet)^\vee
 \longrightarrow h^1\bigl((E^\bullet)^\vee\bigr)[-1]
 \xrightarrow{\ \sigma[-1]\ }\mathcal L[-1],
\]
induces \(\widetilde\sigma\). Work first nonequivariantly and choose, \'etale locally on \(M\), a two-term vector-bundle presentation
\[
 (E^\bullet)^\vee=[E_0\xrightarrow{d}E_1] \quad\text{in degrees }[0,1].
\]
The composite
\[
 E_1\longrightarrow\operatorname{coker}(d)
 =\operatorname{Ob}_M\xrightarrow{\sigma}\mathcal L
\]
is surjective and annihilates \(d(E_0)\). If \(F_1\) is its kernel, then \(F_1\) is a vector bundle and the required morphism and its kernel have the local presentations
\[
 \widetilde\sigma:[E_1/E_0]\longrightarrow\mathbb V(\mathcal L),
 \qquad
 \mathfrak K=[F_1/E_0].
\]
On overlaps, quasi-isomorphisms between two presentations induce the canonical equivalence of their \(h^1/h^0\) stacks and commute with the map to \(\mathcal L\), because both maps induce the fixed sheaf morphism \(\sigma\). Thus the kernels and their zero sections satisfy descent and define the global vector-bundle stack \(\mathfrak K\), of virtual rank \((\rk F_1)-\rk E_0=-v-1\). Finite-type Deligne--Mumford stacks admit stratifications by global quotients
\cite[Proposition~4.5.5(iii)]{Kresch}. Hence homotopy invariance for vector-bundle stacks 
\cite[Proposition~5.3.2]{Kresch} applies to \(\pi_{\mathfrak K}:\mathfrak K\to M\), and shows directly that
\(0^!_{\mathfrak K}=(\pi_{\mathfrak K}^*)^{-1}\) is intrinsic and independent of every local two-term presentation.  For the standard obstruction stack \(\mathfrak E\), this inverse-pullback construction recovers the Behrend--Fantechi virtual class: independence from a chosen global resolution is \cite[Proposition~5.3]{BF}, and Kresch's
\cite[Theorem~6.2.1]{Kresch} removes the global-resolution hypothesis.

After trivializing \(\mathcal L\), Kiem--Li's cone-reduction theorem \cite[Proposition~4.3]{KL} places the intrinsic-cone cycle, with its multiplicities, in \([F_1/E_0]\). Multiplying a trivialization of \(\mathcal L\) by a unit changes neither the kernel nor that cycle, so these local statements descend and give \([\mathfrak C_M]\in Z_*(\mathfrak K)\). This proves that \eqref{eq:abstract-kernel-class} is globally defined.

The excess identity can be proved globally, without assuming that rational Chow identities descend from an \'etale cover. The map $\mathfrak E\to\mathbb V(\mathcal L)$ is locally a surjective linear map. Its zero locus 
$i:\mathfrak K\hookrightarrow\mathfrak E$ is therefore a regular codimension-one embedding with normal line
$\pi_{\mathfrak K}^*\mathcal L$. For $a\in A_*(M)$, flat base change along the vector-bundle-stack projection gives
\[
 i_*\pi_{\mathfrak K}^*a =\pi_{\mathfrak E}^*\bigl(c_1(\mathcal L)\cap a\bigr).
\]
One may also obtain this equality from the zero-section self-intersection formula for $\mathcal L$ and homotopy invariance. These are global Chow operations. Writing $B=\pi_{\mathfrak K}^*a$ and applying the inverse of $\pi_{\mathfrak E}^*$ proves
\[
 0^!_{\mathfrak E}i_*B =c_1(\mathcal L)\cap0^!_{\mathfrak K}B.
\]

Taking
\(B=[\mathfrak C_M]\) proves \eqref{eq:abstract-excess}. Finally, apply the entire construction to the Borel approximations. The descended obstruction sheaf and line bundle are \((\operatorname{Ob}_M\times U)/T\) and
\((\mathcal L\times U)/T\); intrinsic cones, Chern classes, and zero-section Gysin maps commute with the smooth transition maps between good pairs. The compatible ordinary identities therefore stabilize to the stated equivariant Chow identities and prove equivariant presentation independence.
\end{proof}

\begin{remark}\label{rem:kernel-caveat}
Everywhere surjectivity is essential in Proposition \ref{prop:kernel-cone-general}. With a nonempty degeneracy locus the kernel cone need not be a vector-bundle stack globally, and one must instead use Kiem--Li's localized Gysin map.  Moreover, the proposition asserts cycle-theoretic support of \([\mathfrak C_M]\) in \(\mathfrak K\); it does not assert a stronger scheme-theoretic factorization of the entire cone or construct a reduced perfect obstruction theory.
\end{remark}

\end{document}